\documentclass{amsart}
\usepackage{amsmath,amssymb}
\usepackage{graphicx,tikz,mathtools}
\usepackage{color}
\newtheorem{theorem}{Theorem}[section]
\newtheorem{proposition}[theorem]{Proposition}
\newtheorem{corollary}[theorem]{Corollary}

\newtheorem{lemma}[theorem]{Lemma}

\newcommand\gcan{{g_{\mathrm{can}}}}

\newcommand{\tcurl}{\curl'} %\widetilde{\curl}}

\theoremstyle{definition}
\newtheorem{definition}[theorem]{Definition}

\theoremstyle{remark}
\newtheorem{remark}[theorem]{Remark}

\DeclareMathOperator\Harm{Harm}

\newcommand{\al}{\alpha}

\newcommand{\de}{\delta}

\newcommand{\la}{\lambda}
\newcommand{\om}{\omega}

\newcommand{\De}{\Delta}

\newcommand{\Om}{\Omega}

\def\CC{\mathbb{C}}

\def\RR{\mathbb{R}}

\def\ZZ{\mathbb{Z}}
\def\RP{\mathbb{RP}}
\def\TT{\mathbb{T}}
\def\HH{\mathbb{H}}

\renewcommand\SS{\mathbb{S}}
\newcommand{\cC}{{\mathcal C}}

\newcommand{\cE}{{\mathcal E}}
\newcommand{\cF}{{\mathcal F}}

\newcommand{\cH}{{\mathcal H}}

\newcommand{\cR}{{\mathcal R}}

\newcommand{\pd}{\partial}
\newcommand\minus\backslash

\newcommand\lan\langle
\newcommand\ran\rangle

\newcommand\restr{\!\upharpoonright_{\pd\Om}}

\DeclareMathOperator\Div{div}

\DeclareMathOperator\Vol{Vol}

\renewcommand\leq\leqslant
\renewcommand\geq\geqslant
\newlength{\intwidth}

\numberwithin{equation}{section}

 \DeclareMathOperator\curl{curl}

\begin{document}

\title[Optimal metrics for curl]{Optimal metrics for the first curl eigenvalue\\ on $3$-manifolds}

 \author{Alberto Enciso}
 \address{Instituto de Ciencias Matem\'aticas, Consejo Superior de
   Investigaciones Cient\'\i ficas, 28049 Madrid, Spain}
 \email{aenciso@icmat.es}

 \author{Wadim Gerner}
 \address{Sorbonne Universit\'e, Inria, CNRS, Laboratoire Jacques-Louis Lions (LJLL), Paris, France.}
 \email{wadim.gerner@inria.fr}

 \author{Daniel Peralta-Salas}
 \address{Instituto de Ciencias Matem\'aticas, Consejo Superior de
   Investigaciones Cient\'\i ficas, 28049 Madrid, Spain}
 \email{dperalta@icmat.es}

%%    General info
%\subjclass[2010]{35B38, 58J05, 58K45}
%\date{\today}
%
%\keywords{ }
%
\begin{abstract}
In this article we analyze the spectral properties of the curl operator on closed Riemannian $3$-manifolds. Specifically, we study metrics that are optimal in the sense that they minimize the first curl eigenvalue among any other metric of the same volume in the same conformal class. We establish a connection between optimal metrics and the existence of minimizers for the $L^{\frac32}$-norm in a fixed helicity class, which is exploited to obtain necessary and sufficient conditions for a metric to be locally optimal. As a consequence, our main result is that we prove that $\SS^3$ and $\RP^3$ endowed with the round metric are $C^1$-local minimizers for the first curl eigenvalue (in its conformal and volume class). The connection between the curl operator and the Hodge Laplacian allows us to infer that the canonical metrics of $\SS^3$ and $\RP^3$ are locally optimal for the first eigenvalue of the Hodge Laplacian on coexact $1$-forms. This is in strong contrast to what happens in dimension~4.
\end{abstract}

\maketitle

\section{Introduction}\label{S.intro}

The analysis of the eigenvalues of the Hodge Laplacian on compact Riemannian manifolds is a classic topic in spectral geometry, particularly the study of the properties of the first nonzero eigenvalue on a conformal class of metrics. Let us recall that the Hodge Laplacian on a compact $n$-dimensional boundaryless manifold $(M,g)$ is given by
\[
\De_g:= d\de_g+\de_g d\,,
\]
where $d$ and $\de_g$ respectively denote the differential and codifferential; the Hodge Laplacian acting on $p$-forms defines a nonnegative self-adjoint operator. Although we shall focus on the case of 3-dimensional manifolds later on, for the time being we shall not make any assumptions on the dimension. Throughout this article we will assume that all manifolds and metrics are $C^\infty$ and all manifolds are oriented.

\subsection{The eigenvalues of the Hodge Laplacian on a conformal class of metrics}

Consider the decomposition of the space of smooth $p$-forms into exact, coexact and harmonic forms, which we write as
\begin{align*}
\Om^p(M)&= \Om^{p,d}(M) \oplus \Om^{p,\de}(M,g)\oplus \Harm^p(M,g)\,.
\end{align*}
Here $\Om^{p,d}(M):= d\Om^{p-1 }(M)$ and $\Om^{p,\de}(M,g):= \de_g \Om^{p+1}(M) $ and $\Harm^p(M,g)$ is the kernel of~$\De_g$ acting on $p$-forms, whose dimension is given by the $p$-th Betti number of~$M$.
As the Hodge Laplacian commutes with $d$ and $\de_g$, these spaces are invariant under~$\De_g$. It is thus natural to label the nonzero eigenvalues of~$\De_g$ as
\[
\big \{\la_k^{p,d}(M,g),\: \la_k^{p,\de}(M,g): k\geq1\big\}\,.
\]
Furthermore, the commutation properties of the Hodge Laplacian with the differential, codifferential and with the Hodge star operator $\star_g$ imply the well known relations
\begin{equation}\label{E.duality}
\la_k^{p,\de}(M,g)= \la_k^{n-p,d}(M,g)\,,\qquad \la_k^{p,d}(M,g)= \la_k^{p-1,\de}(M,g)\,,	
\end{equation}
and also the corresponding relations where the roles of $d$ and~$\de$ are reversed. In particular,
\[
%\overline\la_k(M,g):=
\{\la_k^{p,\de}(M,g): 0\leq p\leq \lfloor \tfrac n2\rfloor ,\; k\geq1\}
\]
determine the whole (nonzero) spectrum of the Hodge Laplacian, so from now on we will focus our attention on the eigenvalues for coexact $p$-forms, $\la_k^{p,\de}$.

We are interested in the behavior of the eigenvalues of the Laplacian on $p$-forms when the metric ranges over a conformal class. By the scaling properties of eigenvalues, it can be seen that one should only consider metrics with fixed volume. Therefore, we can define the class of conformally equivalent metrics of fixed volume
\begin{equation}\label{E.defcC}
	\cC:=\{g:= f g_0:  |M|_g = |M|_{g_0},\; f\in C^\infty(M,\RR^+))\}\,,
\end{equation}
where $g_0$ is a fixed smooth metric on~$M$ and $\RR^+:=(0,\infty)$.

In the case of the scalar Laplacian, the results are classical. It is well known that, while there are no nontrivial lower bounds~\cite[Remark 30]{Colbo} for the $k$-th eigenvalue (which is just $\la_k^{0,\de}(M,g)$ in the above notation) as~$g$ ranges over~$\cC$, there is indeed an upper bound, which can in fact be estimated  as~\cite{Kore}:
\[
\sup_{g\in\cC} \la_k^{0,\de}(M,g)\leq Ck^{2/n}\,.
\]
The constant $C$  depends on~$g_0$, and it can be arbitrarily large when $n\geq3$ even if $M=\SS^n$; see e.g.~\cite{ElSoufi} and references therein. Of course, this supremum defines an invariant in the class of conformal metrics with fixed volume.

In the case of the first eigenvalue, it is a result of Li and Yau~\cite{LY} and El Soufi and Ilias~\cite{ElSoufi} that the corresponding upper bound can be estimated in terms of the conformal volume of the manifold (see~\cite{LY} for a definition) as
\[
\sup_{g\in\cC} \la_k^{0,\de}(M,g)\leq n \Vol_{\mathrm{c}}(M)^{2/n}\,.
\]
Furthermore, El Soufi and Ilias~\cite{ElSoufi} gave a sufficient
condition for a Riemannian metric to maximize the first eigenvalue in the conformal class~$\cC$. This condition is trivially satisfied by the metric of any homogeneous Riemannian space with irreducible isotropy representation, so as a rule of thumb if there are very symmetric metrics in the conformal class~$\cC$, they are optimal. In particular, the canonical metrics on~$\SS^n$, $\RR P^n$, $\CC P^n$ and~$\HH P^n$ maximize the first eigenvalue, $\la_1^{0,\de}$, in their respective conformal classes.

The behavior of the eigenvalues for the Laplacian on forms is surprisingly different, and can be summarized in the following theorem due to Colbois and El Soufi~\cite{ColEl} and Jammes~\cite{Jammes} %and Ann\'e and Takahashi~\cite{AT21}:

\begin{theorem}[\cite{ColEl,Jammes}]\label{T.AT}
	Let $\cC$ be given by~\eqref{E.defcC}, where $g_0$ is any metric on~$M$. Then:
	\begin{enumerate}
		\item $\sup_{g\in \cC} \la_k^{p}(M,g)<\infty$ if and only if $p\in\{0,1,n-1,n\}$. Here $\la_k^p$ is the $k$-th eigenvalue of the Hodge Laplacian on $p$-forms.
		\item $\inf_{g\in \cC} \la_k^{p,\de}(M,g)>0$  if and only if $p=\lfloor\frac n2\rfloor$ when $n$ is odd and $p\in\{\frac{n}{2}-1,\frac{n}{2}\}$ when $n\geq 4$ is even (in the latter case we recall that $\la_k^{\frac{n}{2},\de}(M,g)=\la_k^{\frac{n}{2}-1,\de}(M,g)$).
\end{enumerate}
\end{theorem}

An interesting consequence of this theorem is that one can define the following conformal invariant of the Hodge Laplacian on forms:
\[
\la_k(M,\cC):= \inf_{g\in \cC} \la_k^{\lfloor \frac n2\rfloor,\de}(M,g)\,.
\]
It is remarkable that the existence of extremal metrics for this quantity is a much more subtle question than the analog for its ``scalar'' counterpart
$$\sup_{g\in\cC} \la_k^{0,\de}(M,g)\,.$$
To make this assertion precise, let us make the following

\begin{definition}\label{D.locoptimal}
	The metric~$g$ is {\em locally optimal}\/ for $\la_1^{\lfloor \frac n2\rfloor,\de}$ if
	\[
	\la_1^{\lfloor \frac n2\rfloor,\de}(M,g)\leq \la_1^{\lfloor \frac n2\rfloor,\de}(M,g')
	\]
	for every metric~$g'$ in the conformal class~$\cC$ of~$g$ that is close enough to~$g$ in the $C^1$-norm. If moreover the equality only holds for $g'=g$, in this neighborhood, we say that~$g$ is {\em strictly locally optimal}\/.
\end{definition}

\begin{remark}
For the purposes of this paper, we could replace the $C^1$ norm by any stronger topology (say $C^k$ with any large~$k$ or $C^\infty$).
\end{remark}

Essentially the only known property of optimal metrics for $\la_1(M,\cC)$ is that~\cite{Jammes}, if $n\geq 4$ is even or $n=3$ mod $4$, a necessary condition for~$g$ to be a locally optimal metric is that some coexact eigenform~$\al$ corresponding to the eigenvalue $\la_1^{\lfloor \frac n2\rfloor,\de}(M,g)$ must have constant length: $|\al|_g=1$. Therefore, manifolds with nonzero Euler characteristic do not admit locally optimal metrics; in particular, the canonical metrics on $\SS^4$, $\RR P^4$, $\CC P^2$ and $\HH P^1$ are not locally optimal in their respective conformal classes.

\subsection{Analysis of the 3-dimensional case via the curl operator}

Our objective in this paper is to prove the existence of smooth locally optimal metrics in certain 3-dimensional manifolds:

\begin{theorem}\label{T.main}
	The canonical metrics on~$\SS^3$ and on~$\RR P^3$ are strictly locally optimal for $\la_1^{1,\de}$.
\end{theorem}
\begin{remark}
In contrast, the flat metric on the standard torus $\TT^3=\RR^3/(2\pi\ZZ)^3$ is not locally optimal for $\la_1^{1,\de}$. See Corollary~\ref{C.torus} for a proof. On the flat $\TT^3$ there are coexact eigenforms corresponding to the eigenvalue $\la_1^{1,\de}$ with constant length, so we conclude that the aforementioned necessary condition for local optimality proved by Jammes~\cite{Jammes} is not sufficient.
\end{remark}

The reason for which we consider $\SS^3$ (and $\RR P^3$) instead of any other 3-manifold is that the lowest eigenforms of~$(\SS^3,\gcan)$ do have constant length, which is a necessary condition for the metric to be optimal. All along this paper, $\mathbb S^3$ denotes the unit sphere in $\mathbb R^4$.

The proof of this result starts off with the formula, valid for coexact 1-forms in dimension $n=3$,
\begin{equation}\label{E.Decurl}
	\De_g \al = -\tcurl_g \tcurl_g \al\,,
\end{equation}
with $\tcurl_g:= \star_g d$ being the \emph{curl operator} acting on 1-forms (which maps 1-forms to 1-forms).
The point now is that, it is well known~\cite{EP12,Bar} that the operator $\tcurl_g$ defines a self-adjoint operator with compact inverse whose domain is dense in $\Om^{1,\de}(M,g)$. Denoting its spectrum by $\mu_k(M,g)$, with $k\in\ZZ\backslash\{0\}$ and
\[
\cdots <  \mu_{-2}(M,g)<\mu_{-1}(M,g)<0<\mu_1(M,g)< \mu_2(M,g)< \cdots\,,
\]
where we are not counting multiplicities of the eigenvalues, it follows from~\eqref{E.Decurl} that the spectrum of~$\De_g$ on coexact 1-forms is precisely the set of squares $\mu_k(M,g)^2$ with $ k\in\ZZ\backslash\{0\}$. We will refer to~$\mu_1(M,g)$ and $\mu_{-1}(M,g)$ as the {\em first
positive eigenvalue}\/ and {\em first negative eigenvalue}\/ of the curl
operator, respectively; $\la_1^{1,\de}(M,g)$ is then the square of one of them (the smallest in absolute value). The following definition is analogous to Definition~\ref{D.locoptimal}:

\begin{definition}
Let $(M,g)$ be a closed $3$-dimensional Riemannian manifold. The metric~$g$ is {\em locally optimal}\/ for the first positive curl eigenvalue if
	\[
	\mu_1(M,g)\leq \mu_1(M,g')
	\]
for every metric~$g'$ in the conformal class~$\cC$ of~$g$ that is close enough to~$g$ in the $C^1$-norm. If moreover the equality only holds for $g'=g$, in this neighborhood, we say that~$g$ is {\em strictly locally optimal}. Analogously, if
\[
|\mu_{-1}(M,g)|\leq |\mu_{-1}(M,g')|\,,
\]
the metric $g$ is locally optimal for the first negative curl eigenvalue.
\end{definition}

Theorem~\ref{T.main} will follow once we show that the canonical metric on~$\SS^3$ and~$\RR P^3$ is locally optimal for both $\mu_1$ and $\mu_{-1}$. For the ease of notation, we shall often write $\mu_k^g\equiv \mu_k(M,g)$.

As an aside, one should recall that the curl operator (usually defined on vector fields, for which we use the notation $\curl_g$) plays a key role in different contexts of mathematical physics, such as electromagnetism and hydrodynamics. For example, objects like the vorticity of a fluid flow, the vector potential of a magnetic field or the rotation tensor in continuum mechanics are defined using the curl. On  Riemannian $3$-manifolds,  since it can be understood as a sort of square root of the Hodge Laplacian on coexact $1$-forms, it has been exploited to establish genericity properties (of Uhlenbeck type) of the Hodge spectrum~\cite{EP12}. The eigenfields of the corresponding spectral problem, i.e., the nontrivial solutions to the linear PDE
\[
\curl _g u_k = \mu_k^g\, u_k\,,
\]
are known as \emph{Beltrami fields} or \emph{force-free fields} and play a key role in fluid dynamics and magnetohydrodynamics.

To describe the idea of the proof, for convenience we will use vector fields instead of 1-forms. To do this, let us denote by $\mathfrak X_{\mathrm{ex}}^g(M)$ the space of $L^2$ vector fields on $M$ that are exact (i.e., the dual 1-form $u^\flat := g(u,\cdot)$ to the vector field~$u$ is coexact), and by $C^\infty\mathfrak X_{\mathrm{ex}}^g(M)$ the space of smooth exact vector fields, which is dual to the space $\Om^{1,\de}(M,g)$. The curl operator on vector fields is then defined as
\[
(\curl_g u)^\flat:= \tcurl_gu^\flat\,.
\]
This defines a self-adjoint operator with compact inverse and domain dense in $\mathfrak X_{\mathrm{ex}}^g(M)$, and its eigenvalues are obviously the same as those of its counterpart on 1-forms, $\tcurl_g$. When there is no risk of confusion, we
shall omit the dependence of curl with the metric, i.e., $\curl\equiv \curl_g$.

We shall exploit that the first positive (or negative) curl eigenvalue admits a useful variational characterization (see e.g.~\cite{Gerner}):
\begin{equation}\label{E.var}
|\mu_{\pm1}(M,g)|=\inf_{w\in \mathfrak X_{\mathrm{ex}}^g(M)\text{, }\pm\mathcal{H}_g(w)>0}\frac{\|w\|^2_{L^2_g}}{|\mathcal{H}_g(w)|}\,.
\end{equation}
Here $\mathcal H_g(w)$ is the {\em helicity}\/ of the vector field $w$, i.e.,
\begin{equation}\label{E.helicity}
\mathcal H_g(w):=\int_M g(\curl^{-1}_g w,w) dV_g\,,
\end{equation}
and $\curl^{-1}_g w$ denotes the unique exact vector field on $M$ whose $\curl_g$ is $w$. We shall use this characterization to show that a sufficient condition for a metric to be optimal for the first positive (or negative) curl eigenvalue is that there exists an associated eigenfield of constant length and which minimizes the $L^{3/2}$-norm $\int_M |u|_g^{3/2}\, dV_g$ in the class of exact fields $\mathfrak X_{\mathrm{ex}}^g(M)$ with unit helicity, $\cH(u)=1$ (resp. $\cH(u)=-1$). A glance at~\eqref{E.helicity} shows that this is a question about a sharp constant (and the corresponding minimizer) for a Sobolev-type inequality for vector fields. Details are given in Section~\ref{S3}. The divergence-free character of the vector field is used in a key way (in particular, to define the helicity), so the sharp Sobolev inequalities on the sphere proven by Beckner~\cite{Beckner} do not seem to provide the sharp constant in this inequality.

However, using a gap condition for the eigenvalues that holds in the projective space $\RR P^3$, we shall show that the lowest Beltrami fields on~$(\RR P^3,\gcan)$ are local minimizers for this Sobolev-type inequality. This will enable us to conclude that $(\RR P^3,\gcan)$ is a local optimizer for $\la_1^{1,\de}$ in its conformal class. The proof for the case of~$(\SS^3,\gcan)$ employs a considerably more involved expansion, but also relies on the variational formulation~\eqref{E.var} and the connection with the aforementioned $L^{3/2}$-norm minimization.

An outstanding open question is whether the lowest Beltrami fields (which are, as we shall see, Hopf fields) are global minimizers for this Sobolev-type inequality on~$\SS^3$, which would imply that the canonical metrics on~$\SS^3$ and~$\RR P^3$ are globally optimal, both for the first curl eigenvalue and the Hodge eigenvalue $\la_1^{1,\de}$. As an aside remark we observe that an analogous optimization problem for the first (positive or negative) eigenvalue of the curl operator in the context of Euclidean domains has been studied only very recently~\cite{EP22,EGP23}, and even the existence of locally optimal Euclidean domains is still wide open; in Appendix~\ref{S.app2} we show that if the Hopf fields are global minimizers, the lower bound for the first curl eigenvalue in Euclidean domains obtained in~\cite[Theorem A.1]{EP22} can be substantially improved.

\subsection{Organization of the paper}

In Section~\ref{S3} we analyze the relation between optimal metrics  and nonvanishing minimizers of the $L^{\frac32}$-energy in the class of exact fields with fixed helicity. This condition is used in Section~\ref{S.RP3} to derive a sufficient condition for a metric to be strictly locally optimal for the first positive curl eigenvalue $\mu_1$. The canonical metric on $\RR P^3$ satisfies this condition, but the one on~$\SS^3$ does not. Therefore, the case of $(\SS^3,\gcan)$ requires a much more involved expansion that is worked out in Section~\ref{S.sphere}. Finally, in Appendix~\ref{S2} we use the variational formulation~\eqref{E.var} to provide alternative proofs of the behavior of the eigenvalues of the Hodge Laplacian stated in Theorem~\ref{T.AT}, in the case $n=3$, using the curl operator, and in Appendix~\ref{S.app2} we present an application of the existence of $L^{\frac32}$-minimizers in the context of Euclidean optimal domains.

\section{$L^{\frac32}$-energy minimization and first curl eigenvalue}\label{S3}

Given a vector field $u$ on a closed Riemannian $3$-manifold $(M,g)$ we define its $L^{\frac32}_g$-energy as
\[
\cE_g(u):=\int_{M}|u|_g^{\frac{3}{2}}dV_g\,,
\]
where $|u|_g:=g(u,u)^{1/2}$.
We say that $u\in\mathfrak X^g_{\mathrm{ex}}(M)$ is an \emph{$L^{\frac32}_g$-minimizer} if
\[
\cE_g(v)\geq \cE_g(u)
\]
for all $v\in \mathfrak X^g_{\mathrm{ex}}(M)$ with $\cH(v)=\cH(u)$ (i.e., in the same helicity class). We observe that, due to simple scaling properties, a vector field $u$ is an $L^{\frac32}_g$-minimizer in its (positive) helicity class if and only if $u$ minimizes the quotient
\[
\frac{\|v\|^2_{L^{\frac32}_g}}{\cH(v)}
\]
among all $v\in\mathfrak X^g_{\mathrm{ex}}(M)$ with $\cH(v)>0$. In what follows we shall use this equivalent characterization of $L^{\frac32}_g$-minimizers without further mention.

In this section we show that there is an intimate relation between the existence of $L^{\frac{3}{2}}_g$-minimizers and the existence of optimal metrics for the first (positive or negative) curl eigenvalue in a conformal class.

\subsection{$L^{\frac32}$-minimizers and optimal metrics}

The main result of this subsection shows a sufficient condition for the existence of an optimal metric in the conformal and volume class of a given metric $g$; the condition is stated in terms of $L^{\frac32}$-minimizers satisfying suitable hypotheses.

\begin{theorem}\label{LPP7}
Let $(M,g)$ be a closed Riemannian $3$-manifold. If there exists a smooth nonvanishing $L^{\frac32}_g$-minimizer in some positive (or negative) helicity class then there exists a metric $\tilde{g}$, conformal to $g$ and of same total volume, which is optimal for the first positive (resp. negative) curl eigenvalue.
%such that all its first (positive) curl eigenfields have constant pointwise norm and minimise the $L_{\tilde{g}}^{\frac{3}{2}}$-energy in their respective helicity class.
\end{theorem}
\begin{proof}
Suppose that there exists a $C^\infty$ nonvanishing $u\in \mathfrak X^g_{\mathrm{ex}}(M)$ that minimizes the $L^{\frac{3}{2}}_g$-energy in its own positive helicity class. The case of negative helicity is analogous. We then define the conformal metric
	\[
	\tilde{g}:=\kappa |u|_g g\,,
	\]
where $\kappa$ is the positive constant
\[
\kappa:=\Big(\frac{|M|_g}{\cE_g(u)}\Big)^{\frac23}\,.
\]
We note that $\tilde{g}$ is smooth and well defined because $u$ is a nonvanishing vector field. One also easily checks that $|M|_{\tilde{g}}=|M|_g$. Using the isomorphism $I:\mathfrak X^g_{\mathrm{ex}}(M)\rightarrow \mathfrak X^{\tilde g}_{\mathrm{ex}}(M)$ introduced at the beginning of the proof of Theorem~\ref{ConfT3} (we also follow the notation there), we see that
$$\star_{\tilde{g}}\star_g \om^{g}_u=\frac{\om^{g}_u}{\sqrt{\kappa|u|_g}}\,.$$
So if we let $\tilde{u}:=I(u)$ denote the associated vector field, then
	\[
	\tilde{g}(\tilde{u},\tilde{u})=\frac{\langle \om^{g}_u,\om^{g}_u\rangle_{\tilde{g}}}{\kappa|u|_g}=\frac{|u|^2_g}{\kappa^2|u|^2_g}=\frac{1}{\kappa^2}=\text{const}.
	\]
Further, as noticed in the proof of Theorem~\ref{ConfT3}, the isomorphism $I$ preserves helicity as well as the $L^{\frac{3}{2}}$-energy, and therefore, since $u$ is an $L^{\frac32}_g$-minimizer in some positive helicity class, $\tilde{u}$ is an $L^{\frac32}_{\tilde g}$-minimizer in the same helicity class.

Next, we claim that $\tilde{u}$ is an $L^2_{\tilde{g}}$-minimizer in its own helicity class, so that (by the variational principle) $\tilde{u}$ is a first curl eigenfield. Indeed, since $\tilde u$ is an $L^{\frac32}_{\tilde g}$-minimizer, using H\"{o}lder's inequality we can write
\[
\|\tilde u\|_{L^{\frac32}_{\tilde g}}\leq \|w\|_{L^{\frac32}_{\tilde g}}\leq \|w\|_{L^2_{\tilde g}}|M|_{\tilde g}^{\frac16}
\]
for all $w\in \mathfrak X^{\tilde g}_{\mathrm{ex}}$ in the same helicity class as $\tilde u$. Moreover, since $|\tilde u|_{\tilde g}$ is constant we have $\|\tilde u\|_{L^{\frac32}_{\tilde g}}=\|\tilde u\|_{L^{2}_{\tilde g}}|M|_{\tilde g}^{\frac16}$, which yields
\[
\|\tilde u\|_{L^{2}_{\tilde g}}\leq \|w\|_{L^2_{\tilde g}}
\]
for all $w\in \mathfrak X^{\tilde g}_{\mathrm{ex}}$ in the same helicity class as $\tilde u$, as we wanted to show.

Finally, the theorem follows using Lemma~\ref{LPL5}, which implies that the metric $\tilde{g}$ is optimal for the first positive (or negative) curl eigenvalue, in the conformal class of $g$ with fixed total volume.
\end{proof}

The following elementary lemma is instrumental in the proof of Theorem~\ref{LPP7} above. Moreover, its corollary below provides a criterion to check whether $\SS^3$ with the round metric is optimal or not.

\begin{lemma}\label{LPL5}
Let $(M,g)$ be a closed Riemannian $3$-manifold. If there exists a first (positive or negative) curl eigenfield $v_0$ such that $|v_0|_g$ is constant and $v_0$ is an $L_g^{\frac{3}{2}}$-minimizer, then $g$ is optimal for the first (positive or negative) curl eigenvalue (among all other conformal metrics of same total volume).
\end{lemma}
\begin{proof}
Let us focus on the case of the first positive curl eigenvalue, the negative one is analogous. For any metric $\tilde g$ that is conformal to $g$ with the same volume, it follows from Equation~\eqref{eq1} and the variational characterization of the first curl eigenvalue that
	\[
	\mu^{\tilde{g}}_1= \inf_{u\in \mathfrak X^{\tilde{g}}_{\mathrm{ex}}(M),\mathcal{H}_{\tilde{g}}(u)>0}\frac{\|u\|^2_{L^2_{\tilde{g}}}}{\mathcal{H}_{\tilde{g}}(u)}\geq \frac{\inf_{v\in \mathfrak X^{g}_{\mathrm{ex}}(M),\mathcal{H}_{g}(v)>0}\frac{\|v\|^2_{L_g^{\frac{3}{2}}}}{\mathcal{H}_g(v)}}{|M|_{g}^{\frac13}}
=\frac{\|v_0\|^2_{{L_g^{\frac{3}{2}}}}}{|M|_g^{\frac13}\mathcal{H}_g(v_0)}\,,
	\]
where we have used that $v_0$ is an $L_g^{\frac{3}{2}}$-minimizer and that $\tilde{g}$ and $g$ have the same volume. We lastly observe that $\mu_1^g \mathcal{H}_g(v_0)=\|v_0\|^2_{L^2_g}$ and that $\frac{\|v_0\|^2_{L^{\frac{3}{2}}_g}}{\|v_0\|^2_{L^2_g}}=|M|_g^{1/3}$ because $|v_0|_g$ is constant. Inserting this identity in the above inequality yields
$$\mu^{\tilde{g}}_1\geq \mu^g_1\,,$$
as we wanted to prove.
	\end{proof}

We finish this section with a simple corollary from Lemma~\ref{LPL5}. We recall that the first (positive or negative) curl eigenvalue of the round sphere $(\mathbb S^3,\gcan )$ is $\pm 2$, and any associated eigenfield is equivalent (modulo an isometry and multiplication by a constant) to the Hopf or anti-Hopf fields
\begin{align*}
B_1:=(-x_2,x_1,-x_4,x_3)|_{\mathbb S^3}\,,\\
B_{-1}:=(-x_2,x_1,x_4,-x_3)|_{\mathbb S^3}\,,
\end{align*}
see e.g.~\cite{PR}. Clearly $|B_{\pm1}|_{\gcan }=1$ and $\cH_{\gcan }(B_{\pm1})=\pm\pi^2$. It is easy to check that $B_1$ and $B_{-1}$ are related by the orientation-reversing isometry
\begin{equation}\label{Eq.isom}
T:(x_1,x_2,x_3,x_4) \to (x_1,x_2,x_3,-x_4)
\end{equation}
restricted to $\mathbb S^3$. Obviously, $\cE_{\gcan}(T_*u)=\cE_{\gcan}(u)$ and $\cH(T_*u)=-\cH(u)$ for any $u\in\mathfrak X^{\gcan}_{\mathrm{ex}}(\mathbb S^3)$, so it is then obvious from Lemma~\ref{LPL5} that

\begin{corollary}\label{C.Hopf}
If the Hopf vector field $B_{1}$ is an $L^{\frac{3}{2}}_{\gcan }(\mathbb S^3)$-minimizer, then the round metric $\gcan $ is optimal for the first (positive and negative) curl eigenvalue (in its conformal class of the same volume). Moreover, $\gcan$ is optimal for the Hodge eigenvalue $\la_1^{1,\de}$.
\end{corollary}

\subsection{$C^0$-Local $L^{\frac{3}{2}}$-minimizers}\label{SL32}

In this subsection we obtain a sufficient condition on a Riemannian manifold $(M,g)$ so that its first curl eigenfields are $C^0$-local $L^{\frac32}_g$-minimizers. To this end, we introduce the space $C^0\mathfrak X^g_{\mathrm{ex}}(M)$ of continuous exact vector fields on $(M,g)$ as the completion of $C^\infty\mathfrak X^g_{\mathrm{ex}}(M)$ with respect to the $C^0$-norm. We say that a vector field $u\in C^0\mathfrak X^g_{\mathrm{ex}}(M)$ is a $C^0$-local $L^{\frac32}_g$-minimizer if $\cE_g(v)\geq \cE_g(u)$
for all $v\in \mathfrak X^g_{\mathrm{ex}}(M)$ in the same helicity class as $u$ and $C^0$-close to $u$. It is standard to check that the functionals $\cE_g$ and $\cH_g$ are continuous maps on $C^0\mathfrak X^g_{\mathrm{ex}}(M)$.

The idea to prove Theorem~\ref{LPT14} below is to use a Taylor expansion (up to second order) of a functional defined in terms of $\cE_g$ and $\cH_g$. If $V$ is a finite dimensional (normed) vector space and the second order derivative of a function $f\in C^2(V)$ gives rise to a negative definite bilinear form $B:V\times V\rightarrow \mathbb{R}$, then (by compactness of the unit sphere) we obtain the obvious estimate
	\[
	B(v,v)=\|v\|^2B\left(\frac{v}{\|v\|},\frac{v}{\|v\|}\right)\leq -\sigma \|v\|^2
	\]
for some $\sigma>0$. However, when $V$ is infinite-dimensional, negative definiteness does not generally imply such a bound. In the proof of Theorem~\ref{LPT14}, the aforementioned bound will be guaranteed as long as $\mu_2^g>2\mu_1^g$, while to consider the borderline case $\mu_2^g=2\mu_1^g$ we will need to demand an additional compactness condition. This is the motivation to introduce the following technical definition.

\begin{definition}[Compactness condition]\label{CC}
Let $(M,g)$ be a closed Riemannian $3$-manifold and denote the curl eigenspaces by $\{E_i\}_{i\in\ZZ_0}$. We say that $(M,g)$ satisfies the \textit{compactness condition} if there exist finitely many eigenspaces $E_{i_j}$, $1\leq j\leq N$ for some $N\in \mathbb{N}$, with $i_1=2$ and $i_j\neq 1$ for all $2\leq j\leq N$, such that the vector space $E:=\oplus_{j=1}^NE_{i_j}$ satisfies the following: for every $u\in E_1$ we have
$$S_u(E)\perp_{L^2_g(M)}S_u(E_k)$$
for every $k\in \mathbb{Z}\setminus \{1,i_1,\dots,i_N\}$, where we define the function space $S_u(E)\subset C^\infty(M)$ as
$$S_u(E):=\{g(u,w): w\in E\}\,.$$
\end{definition}

We are now ready to prove the main result of this section. The proofs of several technical lemmas are relegated to Sections~\ref{SS.techest} and~\ref{SS.derfunct}

\begin{theorem}\label{LPT14}
Let $(M,g)$ be a closed Riemannian $3$-manifold such that all first positive curl eigenfields have constant speed. If $\mu^g_2\geq 2\mu^g_1$, and $(M,g)$ satisfies the compactness condition when $\mu_2^g=2\mu_1^g$, then every first positive curl eigenfield $u$ is a $C^0$-local $L^{\frac32}_g$-minimizer.
\end{theorem}
\begin{proof}
Since the metric $g$ is fixed, the dependence of all the objects with $g$ will be omitted in this proof. First we translate this (local) minimization problem from the setting of an ``infinite dimensional submanifold" of the level set $\mathcal{H}=\text{const}$ to a standard Banach space setting. To this end we define the following functional
	\begin{gather}
		\label{LPE21}
		\mathcal{R}:\left(C^0\mathfrak{X}_{\mathrm{ex}}(M),\|\cdot\|_0\right)\rightarrow \mathbb{R}\text{,     }\cR(X):=\frac{\mathcal{H}(X)}{(\mathcal{E}(X))^{\frac{4}{3}}}\,,
	\end{gather}
provided that $X\not\equiv 0$, and $\cR(0)=0$. The main observation now is that (locally) minimizing $\mathcal{E}$ in a fixed helicity class is equivalent to (locally) maximizing $\mathcal{R}$ on $C^0\mathfrak X_{\mathrm{ex}}(M)$, which is due to the scale invariance $\mathcal{R}(\lambda X)=\mathcal{R}(X)$ for all constants $\lambda\neq 0$.

By means of Lemmas~\ref{LPL17} and~\ref{LPL18} we know that $\mathcal{E}$ and $\mathcal{H}$ are twice continuously differentiable locally around any $X\in C^0\mathfrak{X}_{\mathrm{ex}}(M)$ which is nonvanishing. This implies that $\mathcal{R}$ is also twice continuously differentiable locally around $X$ and we can easily compute
\begin{gather}
	\label{LPE22}
	D\mathcal{R}(X)(Y)=\frac{D\mathcal{H}(X)(Y)}{(\mathcal{E}(X))^\frac{4}{3}}-\frac{4}{3}\mathcal{H}(X)\frac{D\mathcal{E}(X)(Y)}{(\mathcal{E}(X))^\frac{7}{3}}
\end{gather}
for any $Y\in C^0\mathfrak X_{\mathrm{ex}}(M)$. If $X$ is a first curl eigenfield then Lemma~\ref{LPL17} tells us that $D\mathcal{H}(X)(Y)=\frac{2}{\mu_1}\left(X,Y\right)_{L^2}$. Further, as $|X|_g=\text{const}$, we find $$D\mathcal{E}(X)(Y)=\frac{3}{2}\frac{\left(X,Y\right)_{L^2}}{\sqrt{|X|_g}}\,,$$
and by the first eigenfield property $\mathcal{H}(X)=\frac{\mathcal{E}(X)}{\mu_1}\sqrt{|X|_g}$. Consequently, $X$ is a critical point of $\mathcal{R}$, i.e., we have
\[
D\mathcal{R}(X)(Y)=0\text{ for all }Y\in C^0\mathfrak{X}_{\mathrm{ex}}(M)\,.
\]
Denoting the eigenspace of first curl eigenfields by $E_1\subset C^0\mathfrak X_{\mathrm{ex}}(M)$, let us consider the $L^2$-orthogonal projection from $C^0\mathfrak X_{\mathrm{ex}}(M)$ onto $E_1$. For any given $Y\in C^0\mathfrak X_{\mathrm{ex}}(M)$ we write $Y=Y_1+W$ for the decomposition into the $Y_1\in E_1$ and $W\in E^\perp_1$ components. Let $X$ be a first curl eigenfield, fix a constant $0<r_0\leq1$, and let $W\in E^\perp_1$ and $Z\in B_r(X)$ be any elements for small enough $r\leq r_0$. Here and in what follows $B_r(X)$ denotes a $C^0$-neighborhood of radius $r$ at $X$ in $C^0\mathfrak X_{\mathrm{ex}}(M)$. We then compute
\begin{gather}
	\label{LPE23}
D^2\mathcal{R}(Z)(W,W)=\frac{D^2\mathcal{H}(Z)(W,W)}{(\mathcal{E}(Z))^\frac{4}{3}}-\frac{8}{3}\frac{D\mathcal{H}(Z)(W)D\mathcal{E}(Z)(W)}{(\mathcal{E}(Z))^\frac{7}{3}}
\\
\nonumber
+\frac{28}{9}\frac{\mathcal{H}(Z)}{(\mathcal{E}(Z))^\frac{10}{3}}\left(D\mathcal{E}(Z)(W)\right)^2-\frac{4}{3}\frac{\mathcal{H}(Z)}{(\mathcal{E}(Z))^\frac{7}{3}}D^2\mathcal{E}(Z)(W,W)\,.
\end{gather}
Now let $Y\in B_r(X)$ be any fixed element which we decompose as $Y=Y_1+W$. Since $\|Y-X\|_0\leq r$ and $Y_1,X\in E_1$, we infer
\[
\|Y_1-X\|^2_{L^2}\leq \|Y_1-X\|^2_{L^2}+\|W\|^2_{L^2}=\|Y-X\|^2_{L^2}\leq |M|_gr^2
\]
and so $\|Y_1-X\|_{L^2}\leq \sqrt{|M|_g} r$. However, by assumption all first curl eigenfields have constant pointwise norm, from which we infer
\[
\|Y_1-X\|^2_0=\frac{\|Y_1-X\|^2_{L^2}}{|M|_g}\leq r^2\,,
\]
i.e., $\|Y_1-X\|_0\leq r$. In turn, the triangle inequality implies $\|W\|_0\leq 2r$, so that the $L^2$-orthogonal parts of $Y$ are in fact $C^0$-close to the corresponding $L^2$-orthogonal projections of $X$.

The main idea now is to Taylor expand $\mathcal{R}(Y)$ around $Y_1\in E_1$. The previous considerations show that $Y_1,Y\in B_r(X)$ and $\mathcal{R}$ is twice continuously differentiable on $B_r(X)$. Hence the Taylor formula for Banach spaces~\cite[Theorem 4.C]{Zei} implies
\[
	\mathcal{R}(Y)=\mathcal{R}(Y_1)+D\mathcal{R}(Y_1)(W)+\frac{1}{2}D^2\mathcal{R}(Z)(W,W)\,,
\]
where $Z=(1-\lambda)Y_1+\lambda Y=Y_1+\lambda W\in B_r(X)$ for some suitable constant $0\leq \lambda\leq 1$. In particular,
\begin{gather}
	\label{LPE24}
	\|Z-Y_1\|_0=\lambda \|W\|_0\leq \|W\|_0\leq 2r\,.
\end{gather}
As shown before, every first eigenfield is critical for $\mathcal{R}$, and hence
$$D\mathcal{R}(Y_1)(W)=0\,.$$
Moreover, it is easy to check that, because all first eigenfields have constant pointwise norm and $\mathcal{R}$ is scaling invariant, $\mathcal{R}(E_1\setminus \{0\})$ is independent of the choice of the eigenfield. Accordingly, $\mathcal{R}(Y_1)=\mathcal{R}(X)$ and therefore we arrive at
\begin{gather}
	\label{LPE25}
\mathcal{R}(Y)=\mathcal{R}(X)+\frac{1}{2}D^2\mathcal{R}(Y_1)(W,W)+\frac{1}{2}\left(D^2\mathcal{R}(Z)(W,W)-D^2\mathcal{R}(Y_1)(W,W)\right)\,.
\end{gather}

The estimates of the terms $D^2\mathcal{R}(Y_1)(W,W)$ and $D^2\mathcal{R}(Z)(W,W)-D^2\mathcal{R}(Y_1)(W,W)$ are presented in Lemmas~\ref{est.1} and~\ref{est.2}, respectively. Combining these bounds with~\eqref{LPE25} we find
\[
\mathcal{R}(Y)\leq \mathcal{R}(X)+\frac{C\sqrt{r}-\sigma_0}{2}\|W\|^2_{L^2}\,,
\]
where $C$ and $\sigma_0$ are positive constants that depend on $X$ and $r_0$, but not on $r$. So letting $0<r\ll r_0$ so that $C\sqrt{r}\leq \frac{\sigma_0}{2}$ we find
\[
\mathcal{R}(Y)\leq \mathcal{R}(X)-\frac{\sigma_0}{4}\|W\|^2_{L^2}\,,
\]
thus following that $\mathcal{R}(Y)<\mathcal{R}(X)$ whenever $W\not\equiv0$. On the other hand, if $W\equiv 0$, then $Y=Y_1+W=Y_1\in E_1$, and as we had noticed before $\mathcal{R}(Y_1)=\mathcal{R}(X)$. Overall, we conclude that for every first eigenfield $X$ there exists a small enough radius $r>0$ such that for all $Y$ in a $C^0$-neighborhood of radius $r$ at $X$, we have
$$\mathcal{R}(Y)\leq \mathcal{R}(X)$$
with equality if and only if $Y\in E_1$. Any first eigenfield $X$ is then a local maximizer of the functional $\cR$ in $C^0\mathfrak X_{\mathrm{ex}}(M)$, and hence it is a $C^0$-local minimizer of the energy $\cE$ in its helicity class. This completes the proof of the theorem.
\end{proof}

The following result shows that the real projective space $\RR P^3$ endowed with its canonical (round) metric $\gcan $ satisfies all the assumptions of Theorem~\ref{LPT14}, thus providing a remarkable example of first curl eigenfields that are local minimizers for the $L^{\frac32}$-energy.

\begin{corollary}\label{cor.RP}
Every first (positive or negative) curl eigenfield $u$ of the real projective space $(\RR P^3,\gcan )$ is a $C^0$-local $L^{\frac32}_{\gcan }$-minimizer.
\end{corollary}
\begin{proof}
It is well known that $\mu_{\pm1}^{\gcan }(\RR P^3)=\pm2$, $\mu_{\pm2}^{\gcan }(\RR P^3)=\pm4$ and all the first (positive or negative) curl eigenfields descend from the covering $(\SS^3,\gcan )$ of $\RR P^3$, so they have constant pointwise norm. To apply Theorem~\ref{LPT14}, it remains to check the compactness condition. To this end, for $N=2$ we define the vector space $E=E_2\oplus E_{-1}$. Let $u\in E_1$ be a first positive curl eigenfield, it can be checked~\cite{PR} that the function $\gcan (u,w)$ is an eigenfunction of the (scalar) Laplacian of eigenvalue $\mu_k^{\gcan }(\mu_k^{\gcan }-2)$ provided that $w\in E_k$. Accordingly, the function $\gcan (u,v)$ is an eigenfunction of the Laplacian of eigenvalue $8$ for any $v\in E$, while $\gcan (u,\tilde v)$ is an eigenfunction of the Laplacian of eigenvalue $\lambda\neq 8$ for any eigenfield $\tilde v\in E_k$ if $k\not\in \{2,-1\}$. The compactness condition is then immediate from the $L^2_{\gcan }$-orthogonality of Laplacian eigenfunctions with different eigenvalues. Finally, the case of the first negative curl eigenvalue follows from the orientation-reversing isometry in Equation~\eqref{Eq.isom}.
\end{proof}

\subsection{Technical estimates}\label{SS.techest}

The following technical lemmas are used in the proof of Theorem~\ref{LPT14}. They provide bounds for the second derivative of the functional $\cR$ evaluated at a first curl eigenfield $Y_1$. We use the notation and conventions introduced in the proof of Theorem~\ref{LPT14} without further mention. For the ease of notation, throughout we write $\curl $ instead of $\curl_g$.

\begin{lemma}\label{est.1}
We have the estimate
\[
D^2\cR(Y_1)(W,W)\leq -\sigma_0\|W\|^2_{L^2}
\]
for some positive constant $\sigma_0$ that depends on $X$ and $r_0$ (but not on $r\leq r_0$).
\end{lemma}
\begin{proof}
To compute the first and second derivatives of the functionals $\cH$ and $\cE$ we use the lemmas presented in Section~\ref{SS.derfunct}. The first observation is that since $Y_1\in E_1$ we have
\[
\mathcal{H}(Y_1)=\sqrt{|Y_1|_g}\frac{\mathcal{E}(Y_1)}{\mu_1}\text{, }D\mathcal{H}(Y_1)(W)=\frac{2}{\mu_1}\left(Y_1,W\right)_{L^2}=0\text{,}
\]
\[
D\mathcal{E}(Y_1)(W)=\frac{3}{2\sqrt{|Y_1|_g}}\left(Y_1,W\right)_{L^2}=0,
\]
where we have used that $|Y_1|_g$ is constant (by assumption) and that $Y_1$ and $W$ are $L^2$-orthogonal to each other. Therefore Equation~\eqref{LPE23} becomes
\[
D^2\mathcal{R}(Y_1)(W,W)=\frac{D^2\mathcal{H}(Y_1)(W,W)}{(\mathcal{E}(Y_1))^\frac{4}{3}}-\frac{4}{3}\frac{\mathcal{H}(Y_1)}{(\mathcal{E}(Y_1))^\frac{7}{3}}D^2\mathcal{E}(Y_1)(W,W)\,.
\]
We now have $D^2\mathcal{H}(Y_1)(W,W)=2\left(\curl ^{-1}(W),W\right)_{L^2}=2\mathcal{H}(W)$ and
\[
D^2\mathcal{E}(Y_1)(W,W)=\frac{3}{2\sqrt{|Y_1|_g}}\|W\|^2_{L^2}-\frac{3}{4|Y_1|^\frac{5}{2}_g}\int_M\big(Y_1\cdot W\big)^2dV_{g}\,,
\]
using once more that $|Y_1|_g$ is constant. Inserting the above identities we find
\begin{gather}
	\label{LPE26}
D^2\mathcal{R}(Y_1)(W,W)=\frac{2\mu_1\mathcal{H}(W)-2\|W\|^2_{L^2}+\frac{\int_M\big(Y_1\cdot W\big)^2dV_{g}}{|Y_1|^2_g}}{\mu_1(\mathcal{E}(Y_1))^\frac{4}{3}}\,.
\end{gather}

At this point, we need to distinguish the case $\mu_2>2\mu_1$ from the case $\mu_2=2\mu_1$:

\begin{enumerate}
\item $\mu_2>2\mu_1$: We use the rough pointwise estimate $\big(Y_1\cdot W\big)^2\leq |Y_1|^2_g|W|^2_g$ so that we find
	\[
	D^2\mathcal{R}(Y_1)(W,W)\leq \frac{2\mu_1\mathcal{H}(W)-\|W\|^2_{L^2}}{\mu_1(\mathcal{E}(Y_1))^\frac{4}{3}}\,.
	\]
	Now, as $W$ is $L^2$-orthogonal to the eigenspace $E_1$ and the curl eigenfields form an $L^2$-orthonormal basis of $C^0\mathfrak X_{\mathrm{ex}}(M)$, it easily follows that we have the upper bound
	$$\mu_2\mathcal{H}(W)\leq \|W\|^2_{L^2}\,.$$
	Using this estimate we arrive at
	\[
	D^2\mathcal{R}(Y_1)(W,W)\leq -\frac{\frac{1}{\mu_1}-\frac{2}{\mu_2}}{(\mathcal{E}(Y_1))^\frac{4}{3}}\|W\|^2_{L^2}.
	\]
	Letting $\sigma_1:=\frac{\frac{1}{\mu_1}-\frac{2}{\mu_2}}{(\mathcal{E}(Y_1))^\frac{4}{3}}$ we observe that our assumption $\mu_2>2\mu_1$ implies $\sigma_1>0$ and so we find
	\begin{gather}
		\label{LPE27}
		D^2\mathcal{R}(Y_1)(W,W)\leq -\sigma_1\|W\|^2_{L^2}
	\end{gather}
	for some positive constant $\sigma_1$. Moreover, since any first eigenfield $Y_1$ has constant pointwise norm, and is $C^0$-close to $X$, we deduce that
	$$(\mathcal{E}(Y_1))^\frac{4}{3}=\|Y_1\|^2_0|M|_g^\frac{4}{3}\leq c$$
	for some positive constant $c$ that only depends on $X$ and $r_0$ (and, of course, the manifold $(M,g)$). In particular we may replace $\sigma_1$ by some possibly smaller (strictly positive) constant $\sigma_0$ that is independent of $Y_1$ (but it depends on $X$ and the fixed upper radius $r_0$). This completes the proof of the lemma in this case.

\item $\mu_2=2\mu_1$ and $(M,g)$ satisfies the compactness condition: We decompose further $W=W_E+W_0$, where $W_E\in E$ is the $L^2$-orthogonal projection of $W$ into the finite-dimensional vector space space $E$ introduced in the compactness condition, and $W_0:=W-W_E$. It then follows from the compactness assumption that, letting
\[
f(W):=2\mu_1\mathcal{H}(W)-2\|W\|^2_{L^2}+\frac{\int_M\big(Y_1\cdot W\big)^2dV_{g}}{|Y_1|^2_g}\,,
\]
we have the identity $f(W)=f(W_0)+f(W_E)$. Since $E$ contains the subspace $E_2$ by its definition, we see that the smallest positive eigenvalue possibly contributing to $W_0$ is $\mu_3>\mu_2\geq 2\mu_1$, so that the term $f(W_0)$ can be handled identically as case (i) (see Equation~\eqref{LPE26}). As for $f(W_E)$ we can use the same rough estimates as in case (i) in order to deduce that $f(W_E)\leq 0$ and equality holds if and only if $Y_1$ and $W_E$ are pointwise linearly dependent and $W_E$ is a second eigenfield. We claim this is only possible when $W_E=0$. Indeed, if $W_E=hY_1$ for some $h\in C^\infty(M)$, then
$$\mu_2W_E=\operatorname{curl}(W_E)=\operatorname{curl}(hY_1)=\mu_1W_E+\nabla h\times Y_1\,.$$
Multiplying by $W_E$ we deduce that $|W_E|_g^2=0$ on all of $M$, as claimed. Accordingly, when $W_E\neq 0$ we find $f(W_E)<0$ and therefore, since $E$ is finite dimensional, the desired estimate immediately follows.
\end{enumerate}
\end{proof}

\begin{lemma}\label{est.2}
We have the estimate
\[
|D^2\mathcal{R}(Z)(W,W)-D^2\mathcal{R}(Y_1)(W,W)|\leq C\|W\|^2_{L^2}\sqrt{r}
\]
for some suitable $C>0$ that depends on $X$ and $r_0$, but it
is independent of $r$ for all $0<r\leq r_0$.
\end{lemma}
\begin{proof}
Using Equation~\eqref{LPE23}, we can write the difference $D^2\mathcal{R}(Z)(W,W)-D^2\mathcal{R}(Y_1)(W,W)$ as a sum of four terms. The first term
\[
\frac{D^2\mathcal{H}(Z)(W,W)}{\mathcal{E}(Z)^\frac{4}{3}}-\frac{D^2\mathcal{H}(Y_1)(W,W)}{\mathcal{E}(Y_1)^\frac{4}{3}}
\]
can be easily estimated from above because
$$D^2\mathcal{H}(Z)(W,W)=2\|W\|^2_{L^2}=D^2\mathcal{H}(Y_1)(W,W)\,,$$
and $\|Z-Y_1\|_0=\lambda\|W\|_0\leq \|W\|_0\leq 2r$. Since $Z$ and $Y_1$ are $C^0$-close (and in turn, also $C^0$-close to $X$), by Lipschitz continuity of the function $s\mapsto \frac{1}{s^\alpha}$ for $\alpha>0$ in a compact interval of $(0,\infty)$, we can write
\begin{gather}
	\label{LPE28}
\frac{D^2\mathcal{H}(Z)(W,W)}{\mathcal{E}(Z)^\frac{4}{3}}-\frac{D^2\mathcal{H}(Y_1)(W,W)}{\mathcal{E}(Y_1)^\frac{4}{3}}\leq c_1\|W\|^2_{L^2}\left|\|Y_1\|_{L^\frac{3}{2}}-\|Z\|_{L^\frac{3}{2}}\right|
\\
\nonumber
\leq c_1\|W\|^2_{L^2}\|Z-Y_1\|_{L^\frac{3}{2}}\leq c_1|M|^\frac{2}{3}_g\|W\|^2_{L^2}\|W\|_0\leq 2c_1|M|_g^\frac{2}{3}r\|W\|^2_{L^2}
\end{gather}
for some suitable constant $c_1>0$ that only depends on $X$ and $r_0$.

Concerning the second term in Equation~\eqref{LPE23}, we recall that $D\mathcal{H}(Y_1)(W)=0$ and that $\mathcal{E}(Z)$ is contained in some compact interval of $(0,\infty)$ determined by $X$ and $r_0$ alone. Hence it is enough to estimate the term $D\mathcal{H}(Z)(W)D\mathcal{E}(Z)(W)$, which we can do as follows
\[
D\mathcal{H}(Z)(W)D\mathcal{E}(Z)(W)=(D\mathcal{H}(Z)(W)-D\mathcal{H}(Y_1)(W))D\mathcal{E}(Z)(W)
\]
\[
\leq 3\Big|\left(\curl ^{-1}(W),Z-Y_1\right)_{L^2}\Big|\int_M|W|_g\sqrt{|Z|_g}dV_g\,.
\]
By Cauchy-Schwarz we find
$$\Big|\left(\curl ^{-1}(W),Z-Y_1\right)_{L^2}\Big|\leq \frac{\|W\|^2_{L^2}}{\mu_1}$$
because $\|Z-Y_1\|_{L^2}=\lambda\|W\|_{L^2}$ and $0\leq \lambda\leq 1$. Another application of Cauchy-Schwarz yields
\[
\int_M|W|_g\sqrt{|Z|_g}dV_{g}\leq \|W\|_{L^2}\sqrt{\|Z\|_{L^1}}\leq \sqrt{\|Z\|_0|M|_g}\|W\|_{L^2}\,.
\]
The fact that $\|Z\|_0\leq |X|_g+3r_0$ allows us to find a uniform upper bound
\begin{gather}
	\label{LPE29}
	D\mathcal{H}(Z)(W)D\mathcal{E}(Z)(W)\leq c_2\|W\|^3_{L^2}\leq c_2|M|_g^{\frac12}\|W\|^2_{L^2}\|W\|_0\,,
\end{gather}
for some constant $c_2>0$ that only depends on $X$ and $r_0$.

The estimate of the third term in Equation~\eqref{LPE23} is similar to the estimate of the second term. By continuity of helicity we need only to find an upper bound for the quantity $\left(D\mathcal{E}(Z)(W)\right)^2$. We first note that $|D\mathcal{E}(Z)(W)|\leq \|W\|_{L^2}\sqrt{\|Z\|_{L^1}}$ and that $D\mathcal{E}(Y_1)(W)=0$. Hence we find
\[
\left(D\mathcal{E}(Z)(W)\right)^2=D\mathcal{E}(Z)(W)\left(D\mathcal{E}(Z)(W)-D\mathcal{E}(Y_1)(W)\right)
\]
\[
\leq \|W\|_{L^2}\sqrt{\|Z\|_{L^1}}\int_M|W|_g\left|\frac{Z}{\sqrt{|Z|_g}}-\frac{Y_1}{\sqrt{|Y_1|_g}}\right|dV_{g}\,.
\]
Now we use the pointwise estimate derived in~\cite[Lemma 5.5.4]{GDiss},
\[
\left|\frac{Z}{\sqrt{|Z|_g}}-\frac{Y_1}{\sqrt{|Y_1|_g}}\right|\leq \sqrt[4]{6}\sqrt{|Z-Y_1|_g}\leq \sqrt[4]{6}\sqrt{|W|_g}\,.
\]
Therefore
\[
\int_M|W|_g\left|\frac{Z}{\sqrt{|Z|_g}}-\frac{Y_1}{\sqrt{|Y_1|_g}}\right|dV_g\leq \sqrt[4]{6}\int_M|W|_g\sqrt{|W|_g}dV_{g}
\]
\[
\leq \sqrt[4]{6}\|W\|_{L^2}\sqrt{\|W\|_{L^1}}\leq \sqrt[4]{6}\sqrt{|M|_g}\|W\|_{L^2}\sqrt{\|W\|_0}\,.
\]
Overall we infer that
\begin{gather}
	\label{LPE30}
	\left(D\mathcal{E}(Z)(W)\right)^2\leq c_3\|W\|^2_{L^2}\sqrt{\|W\|_0}
\end{gather}
for some constant $c_3>0$ that only depends on $X$ and $r_0$.

We are left with estimating the fourth (and last) term in Equation~\eqref{LPE23}:
\[
\frac{\mathcal{H}(Z)}{\left(\mathcal{E}(Z)\right)^\frac{7}{3}}D^2\mathcal{E}(Z)(W,W)-\frac{\mathcal{H}(Y_1)}{\left(\mathcal{E}(Y_1)\right)^\frac{7}{3}}D^2\mathcal{E}(Y_1)(W,W)\,.
\]
To estimate this quantity, it is convenient to add a zero
\[
0=\frac{\mathcal{H}(Z)}{\left(\mathcal{E}(Z)\right)^\frac{7}{3}}D^2\mathcal{E}(Y_1)(W,W)-
\frac{\mathcal{H}(Z)}{\left(\mathcal{E}(Z)\right)^\frac{7}{3}}D^2\mathcal{E}(Y_1)(W,W)\,.
\]
Observing that
$$|D^2\mathcal{E}(Y_1)(W,W)|\leq c_4\|W\|^2_{L^2}$$
for some constant $c_4>0$ that only depends on $X$ and $r_0$, we can write
\[
\left(\frac{\mathcal{H}(Z)}{(\mathcal{E}(Z))^\frac{7}{3}}-\frac{\mathcal{H}(Y_1)}{(\mathcal{E}(Y_1))^\frac{7}{3}}\right)D^2\mathcal{E}(Y_1)(W,W)\leq c_4\|W\|^2_{L^2}\left|\frac{\mathcal{H}(Z)}{(\mathcal{E}(Z))^\frac{7}{3}}-\frac{\mathcal{H}(Y_1)}{(\mathcal{E}(Y_1))^\frac{7}{3}}\right|.
\]
To estimate the RHS of this inequality, we add another zero $0=\frac{\mathcal{H}(Y_1)}{\left(\mathcal{E}(Z)\right)^\frac{7}{3}}-\frac{\mathcal{H}(Y_1)}{\left(\mathcal{E}(Z)\right)^\frac{7}{3}}$ and then apply the same reasoning as when we estimated the first term in~\eqref{LPE23} to conclude
\[
\left|\frac{\mathcal{H}(Y_1)}{\left(\mathcal{E}(Z)\right)^\frac{7}{3}}-\frac{\mathcal{H}(Y_1)}{\left(\mathcal{E}(Y_1)\right)^\frac{7}{3}}\right|\leq \tilde{c}\Big|\|Z\|_{L^\frac{3}{2}}-\|Y_1\|_{\frac{3}{2}}\Big|
\]
\[
\leq \tilde{c}\|Z-Y_1\|_{L^\frac{3}{2}}\leq \tilde{c}|M|^\frac{2}{3}_g\|W\|_0\,,
\]
for some constant $\tilde c>0$ that only depends on $X$ and $r_0$. The term
\[
\left|\frac{\mathcal{H}(Z)}{\left(\mathcal{E}(Z)\right)^\frac{7}{3}}-\frac{\mathcal{H}(Y_1)}{\left(\mathcal{E}(Z)\right)^\frac{7}{3}}\right|
\]
is also easily estimated using that $|\mathcal{H}(Z)-\mathcal{H}(Y)|\leq \hat{c}\|W\|_{L^2}\leq \hat{c}_2\|W\|_0$, for some constant $\hat c_2$ that only depends on $X$ and $r_0$. Putting these estimates together we find
\begin{gather}
	\label{LPE31}
\left(\frac{\mathcal{H}(Z)}{(\mathcal{E}(Z))^\frac{7}{3}}-\frac{\mathcal{H}(Y_1)}{(\mathcal{E}(Y_1))^\frac{7}{3}}\right)D^2\mathcal{E}(Y_1)(W,W)\leq c_5\|W\|^2_{L^2}\|W\|_0
\end{gather}
for some constant $c_5>0$ that only depends on $X$ and $r_0$.

The final quantity to estimate is
\[
\frac{\mathcal{H}(Z)}{\left(\mathcal{E}(Z)\right)^\frac{7}{3}}\left(D^2\mathcal{E}(Z)(W,W)-D^2\mathcal{E}(Y_1)(W,W)\right)\,.
\]
Since the first factor can be easily bounded from above by a constant that only depends on $X$ and $r_0$, we only need to estimate the second factor $$\left|D^2\mathcal{E}(Z)(W,W)-D^2\mathcal{E}(Y_1)(W,W)\right|\,.$$
Using that
\[
\left|\int_M|W|^2_g\left(\frac{1}{\sqrt{|Z|_g}}-\frac{1}{\sqrt{|Y_1|_g}}\right)dV_{g}\right|\leq \int_M|W|^2_g\left|\frac{1}{\sqrt{|Z|_g}}-\frac{1}{\sqrt{|Y_1|_g}}\right|dV_g\,,
\]
and the pointwise Lipschitz bound (recall that $M$ is compact) for each $p\in M$,
\[
\left|\frac{1}{\sqrt{|Z(p)|_g}}-\frac{1}{\sqrt{|Y_1(p)|_g}}\right|\leq c_6\Big||Z(p)|_g-|Y_1(p)|_g\Big|
\]
\[
\leq c_6|Z(p)-Y_1(p)|_g\leq c_6|W(p)|_g\leq c_6\|W\|_0\,,
\]
(as usual, $c_6>0$ is a constant that only depends on $X$ and $r_0$), we arrive at
\begin{gather}
	\label{LPE32}
	\left|\int_M|W|^2_g\left(\frac{1}{\sqrt{|Z|_g}}-\frac{1}{\sqrt{|Y_1|_g}}\right)dV_g\right|\leq c_6\|W\|^2_{L^2}\|W\|_0\,.
\end{gather}
Taking into account the expression for the second derivative $D^2\cE$, it remains to estimate the quantity
\[
\left|\int_M\Bigg(\frac{\big(Z\cdot W\big)^2}{|Z|^\frac{5}{2}_g}-\frac{\big(Y_1\cdot W\big)^2}{|Y_1|^\frac{5}{2}_g}\Bigg)dV_g\right|.
\]
To this end we write
\begin{align*}
&\frac{\big(Z\cdot W\big)^2}{|Z|^\frac{5}{2}_g}-\frac{\big(Y_1\cdot W\big)^2}{|Y_1|^\frac{5}{2}_g}\\
&=\Bigg(\frac{\big(Z\cdot W\big)^2}{|Z|^\frac{5}{2}_g}-
\frac{\big(Y_1\cdot W\big)^2}{|Z|^\frac{5}{2}_g}\Bigg)+\Bigg(
\frac{\big(Y_1\cdot W\big)^2}{|Z|^\frac{5}{2}_g}-\frac{\big(Y_1\cdot W\big)^2}{|Y_1|^\frac{5}{2}_g}\Bigg)\,.
\end{align*}
The second summand can be estimated (pointwise) as
\[
\left|\frac{\big(Y_1\cdot W\big)^2}{|Z|^\frac{5}{2}_g}-\frac{\big(Y_1\cdot W\big)^2}{|Y_1|^\frac{5}{2}_g}\right|\leq |Y_1|^2_g|W|^2_g\left|\frac{1}{|Z|_g^\frac{5}{2}}-\frac{1}{|Y_1|_g^\frac{5}{2}}\right|\,,
\]
which combined with the bound
\[
|Y_1|^2_g\left|\frac{1}{|Z|_g^\frac{5}{2}}-\frac{1}{|Y_1|_g^\frac{5}{2}}\right|\leq c_7|Z-Y_1|_g\leq c_7\|W\|_0\,,
\]
for some constant $c_7>0$ (which only depends on $X$ and $r_0$), it yields
\[
\left|\int_M\Bigg(\frac{\big(Y_1\cdot W\big)}{|Z|^\frac{5}{2}_g}-\frac{\big(Y_1\cdot W\big)^2}{|Y_1|^\frac{5}{2}_g}\Bigg)dV_{g}\right|\leq c_7\|W\|^2_{L^2}\|W\|_0\,.
\]
To bound the first summand we observe
\[
\frac{\big(Z\cdot W\big)^2-\big(Y_1\cdot W\big)^2}{|Z|^\frac{5}{2}_g}=\frac{\big((Z-Y_1)\cdot W\big) \big((Z+Y_1)\cdot W\big)}{|Z|^\frac{5}{2}_g}
\]
\[
=\lambda |W|^2_g\frac{(Z+Y_1)\cdot W}{|Z|^\frac{5}{2}_g}\leq |W|^2_g\|W\|_0\frac{|Y_1+Z|_g}{|Z|^\frac{5}{2}_g}\leq c_8|W|^2_g\|W\|_0\,,
\]
with $c_8>0$ a constant that, as usual, only depends on $X$ and $r_0$. Combining these bounds with the estimate~\eqref{LPE32}, we infer that
\[
\left|D^2\mathcal{E}(Z)(W,W)-D^2\mathcal{E}(Y_1)(W,W)\right|\leq c_9\|W\|^2_{L^2}\|W\|_0
\]
for some constant $c_9>0$ that only depends on $X$ and $r_0$.

We recall that $\|W\|_0\leq 2r$, and thus for small enough $r$,
$$\|W\|_0\leq \sqrt{\|W\|_0}\leq \sqrt{2r}\,.$$
This combined with the previous estimate for $D^2\mathcal{E}$ and the bounds in Equations~\eqref{LPE30},~\eqref{LPE29} and~\eqref{LPE28}, finally yield
\[
|D^2\mathcal{R}(Z)(W,W)-D^2\mathcal{R}(Y_1)(W,W)|\leq C\|W\|^2_{L^2}\sqrt{r}
\]
for some constant $C>0$ that depends on $X$ and $r_0$, but it is independent of $r$ for all $0<r\leq r_0$. This completes the proof of the lemma.
\end{proof}

\subsection{Differentiability of the functionals}\label{SS.derfunct}

In this subsection we establish the differentiability of the energy functional $\mathcal{E}$ and the helicity $\mathcal{H}$. We recall that for a twice differentiable map $f:V\rightarrow \mathbb{R}$ ($V$ any normed space) its derivative $Df(X)$ at $X\in V$ is a bounded, linear operator from $V$ to $\mathbb{R}$, and the second derivative $D^2f(X)$ is a bounded linear map from $V$ into the dual space of $V$. In particular, $D^2f(X)(Z)$ is a bounded linear operator from $V$ to $\mathbb{R}$ for any $Z\in V$ and we use the notation
$$(D^2f(X))(Z,Y):=((D^2f(X))(Z))(Y)\,.$$

The first lemma is standard, so its proof is omitted (see e.g.~\cite{PNAS}).
\begin{lemma}\label{LPL17}
Let $(M,g)$ be a closed Riemannian $3$-manifold. The helicity $\mathcal{H}$ is twice continuously (Fr\'{e}chet) differentiable with respect to the $C^0$-norm, with
	\begin{gather}
		\nonumber
	(D\mathcal{H}(X))(Y)=2\left(\curl ^{-1}(X),Y\right)_{L^2_g}\,,\\
	\nonumber
	((D^2\mathcal{H}(X))(Z,Y)=2\left(\curl ^{-1}(Y),Z\right)_{L^2_g}\,,
	\end{gather}
for all $X,Y,Z\in C^0\mathfrak X_{\mathrm{ex}}^g(M)$.
\end{lemma}

The second lemma establishes the differentiability properties of the energy functional $\cE$, and provides formulas for its first and second derivatives.

\begin{lemma} \label{LPL18}
Let $(M,g)$ be a closed Riemannian $3$-manifold. Then the functional $\mathcal{E}$ is continuously (Fr\'{e}chet) differentiable with respect to the $C^0$-norm and its derivative is given by
	\[
	D\mathcal{E}(X)(Y)=\frac{3}{2}\int_M\left(F(X)\cdot Y\right)dV_{g}\,,
	\]
	where we define for each point $p\in M$, $F_p:T_pM\rightarrow \mathbb{R}$ as
$$F_p(v)=\begin{cases}
		\frac{v}{\sqrt{|v|_g}} & v\neq 0\,,\\
		0& v=0\,,
	\end{cases}$$
and $F(X)(p):=F_p(X(p))$.
Further, if $X_0\in C^0\mathfrak X^g_{\mathrm{ex}}(M)$ is a nonvanishing vector field, then there exists a constant $r>0$ such that $\mathcal{E}$ is twice continuously (Fr\'{e}chet) differentiable on $B_r(X_0)$ (a $C^0$-neighborhood) with
\[
D^2\mathcal{E}(X)(Z,Y)=\frac{3}{2}\int_M\frac{Z\cdot Y}{\sqrt{|X|_g}}\,dV_g-\frac{3}{4}\int_M\frac{\big(X\cdot Z\big)\big(X\cdot Y\big)}{|X|^{\frac{5}{2}}_g}\,dV_{g}\,,
\]
for all $X\in B_r(X_0)$ and all $Y,Z\in C^0\mathfrak X^g_{\mathrm{ex}}(M)$.
\end{lemma}
\begin{proof}
It follows from~\cite[Lemma 5.5.7]{GDiss} that $\mathcal{E}$ is continuously (Fr\'{e}chet) differentiable with respect to the $L_g^{\frac{3}{2}}$-norm with the stated derivative. Since the $C^0$-norm dominates the $L_g^{\frac{3}{2}}$-norm, the result carries over immediately. As for the second derivative, if we consider any fixed $X_0\in C^0\mathfrak X^g_{\mathrm{ex}}(M)$ which is nowhere vanishing, we may let $r>0$ be so small that all $X\in B_r(X_0)$ are also nowhere vanishing, and hence $F(X)=\frac{X}{\sqrt{|X|_g}}$ on all of $M$. For any fixed $X\in B_r(X_0)$ and any $p\in M$ we can then consider the function $f_p:T_pM\rightarrow \mathbb{R}$ defined as
$$f_p(v)=|X(p)+v|^{-\frac{1}{2}}_g\,.$$
Next we Taylor expand the function $f_p$ around $v=0$ with a second order remainder term. Doing so and considering $h\in C^0\mathfrak X^g_{\mathrm{ex}}(M)$ with $\|h\|_0\ll 1$, the remainder part is defined as
	\[
	R(p):=f_p(h(p))-f_p(0)+\frac{1}{2}\frac{g(X(p),h(p))}{|X(p)|^\frac{5}{2}_g}\,,
	\]
which can be estimated as $|R(p)|\leq c|h(p)|^2_g$ for some constant $c>0$, which may be chosen independent of $X$, $h$ and $p$, but it depends on $X_0$, $g$ and $r$. From this one readily deduces that the remainder term in the definition of the Fr\'{e}chet-derivative converges to zero.

In order to conclude that the second derivatives are continuous one uses that the values of $|X(p)|_g$ are contained in some compact interval of $(0,+\infty)$ for all $p\in M$. Using that for all $\alpha>0$ functions of the form $s\mapsto \frac{1}{s^\alpha}$ are uniformly continuous on such intervals one easily deduces the continuity of the second derivatives, as well as the expression stated in the lemma.
\end{proof}

\section{Optimal metrics}\label{S.RP3}

In this section we establish necessary and sufficient conditions for a metric to be (locally) optimal for the first positive curl eigenvalue. In particular, we exploit the connection between optimality and $L^{\frac32}$-minimizers presented in Section~\ref{S3} to prove that $\mathbb RP^3$ endowed with the round metric is locally optimal. In particular, we prove the second part of Theorem~\ref{T.main} (see Theorem~\ref{LPT20}).

\subsection{Properties of the minimizing eigenfields}

In the following theorem we show that a necessary condition for a metric to be (locally) optimal for the first curl eigenvalue is that all the first curl eigenfields have constant pointwise norm. In the proof we use an intermediate result that is presented in Lemma~\ref{ConfT6} below.

\begin{theorem}\label{ConfC7}
Let $(M,g)$ be a closed Riemannian $3$-manifold and suppose that $g$ is (locally) optimal for the first (positive or negative) curl eigenvalue (in its conformal class of prescribed volume). Then any first curl eigenfield $u$ must satisfy $|u|_g=\text{const}$.
\end{theorem}
\begin{proof}
Let $f\in C^{\infty}(M)$ be a smooth function of zero mean, $\int_Mf dV_g=0$, and $C>0$ a constant such that $\sup_{p\in M}|f(p)|<C$. For $|t|<\frac{1}{C}=:\tau$ we define the family of metrics
$$g_t:=(1+tf)^{\frac{2}{3}}g=\exp\left(\frac{2}{3}\ln(1+tf)\right)g\,,$$
which is well-defined because $|tf|<1$. In the notation of Lemma~\ref{ConfT6} below we have $\phi=\frac{\ln(1+tf)}{3}$, and clearly $\phi(0,p)=0$ for all $p\in M$, and
	\[
	\int_M\exp\left(3\phi(t,\cdot)\right)dV_g=|M|_g+t\int_MfdV_g=|M|_g
	\]
	for all $|t|<\tau$. Therefore $\phi$ satisfies the conditions in Lemma~\ref{ConfT6} so that we conclude
	\[
	\int_M f|u|^2_g\, dV_g=0\,,
	\]
	where we used that $(\partial_t\phi)_{t=0}=\frac{f}{3}$. Since $f\in C^{\infty}(M)$ is any smooth function with zero mean, we infer that $|u|^2_g$ must be constant, as we wanted to show.
\end{proof}

\begin{lemma}\label{ConfT6}
Let $(M,g)$ be a closed Riemannian $3$-manifold and suppose that $g$ is (locally) optimal for the first (positive or negative) curl eigenvalue. Then for any smooth function $\phi:[-1,1]\times M\rightarrow \mathbb{R}$ satisfying
\begin{enumerate}
\item $\phi(0,\cdot)\equiv 0$,
\item $\int_M \exp(3 \phi(t,\cdot))dV_g=|M|_g$ for all $-1\leq t\leq +1$
\end{enumerate}
we have
\[
\int_M(\partial_t\phi)|_{t=0}|u|^2_gdV_g=0\,,
\]
where $u$ is any first (positive or negative) curl eigenfield.
\end{lemma}
\begin{proof}
We define the family of smooth, conformal metrics $g_t:=\exp(2\phi(t,\cdot))g$ and observe that condition (ii) of our assumptions implies that $|M|_{g_t}=|M|_g$ for all $-1\leq t \leq 1$. We consider the case of the first positive curl eigenvalue (the negative one is analogous). Let $u$ be a first curl eigenfield, of helicity $1$, and consider its dual $1$-form $\om$ (using the metric) and the isomorphism $\mathcal{I}_t(\om):=\star_{g_t}\star_g\om$ introduced in the proof of Theorem~\ref{ConfT3}. As remarked in Appendix~\ref{S2}, this isomorphism preserves helicity. It is straightforward to see that $\mathcal{I}_t(\om)=\exp(-\phi(t,\cdot))\om$. Since $g$ is assumed to be locally optimal, the variational characterization of the first curl eigenvalue yields
	\[
	\mu^{g}_{1}\leq \mu^{g_t}_{1}\leq \frac{\int_M\langle \om,\om\rangle_{g_t}\exp(-2\phi(t,\cdot))dV_{g_t}}{\mathcal{H}_{g_t}(\mathcal{I}_t(\om))}=\int_M\langle\om,\om\rangle_{g}\exp(-\phi(t,\cdot))dV_g\,,
	\]
for all $|t|$ small enough. Setting
	\[
	f(t):=\int_M\langle\om,\om\rangle_{g}\exp(-\phi(t,\cdot))dV_g=\int_M|u|^2_g\exp(-\phi(t,\cdot))dV_g\,,
	\]
we infer from condition (i) on the function $\phi$ that $f(0)=\mu^g_{1}$ because $u$ is a first curl eigenfield of helicity $1$. In particular, the function $f$ has a local minimum at $t=0$. Finally, differentiating $f$ we obtain
	\[
	0=\int_M|u|^2_g\partial_t|_{t=0}\exp(-\phi(t,\cdot))dV_g\Leftrightarrow \int_M\left(\partial_t|_{t=0}\phi(t,\cdot)\right)|u|^2_gdV_g=0\,,
	\]
as we wanted to show.
\end{proof}

A straightforward application of Theorem~\ref{ConfC7} yields that the flat metric $g_0$ on the standard torus $\mathbb T^3=\RR^3/(2\pi\ZZ)^3$ is not locally optimal for the first curl eigenvalue, in its conformal class with same volume (we recall that the first curl eigenvalues of $(\TT^3,g_0)$ are $\mu_{\pm 1}=\pm 1$). The reason is that there are first curl eigenfields (both for $\mu_1$ and $\mu_{-1}$), the so called ABC-flows, that do not have constant speed~\cite{AK}. Accordingly:

\begin{corollary}\label{C.torus}
The flat metric $g_0$ on the standard torus $\TT^3$ is not locally optimal for the first (both positive and negative) curl eigenvalue nor for the first Hodge eigenvalue~$\la_1^{1,\de}$.
\end{corollary}

\subsection{Local optimality of $(\RR P^3,\gcan)$}\label{Sopti}

In this section we show that the necessary condition for optimality presented in Theorem~\ref{ConfC7} is also sufficient for local optimality, provided that the spectrum of curl satisfies some gap assumption. This result complements, and makes essential use of, Theorem~\ref{LPT14}.

\begin{theorem}\label{LPT20}
Let $(M,g_0)$ be a closed Riemannian $3$-manifold such that all first positive curl eigenfields have constant speed. If $\mu_2^{g_0}\geq 2\mu_1^{g_0}$, and $(M,g_0)$ satisfies the compactness condition (cf. Definition~\ref{CC}) when $\mu_2^{g_0}=2\mu_1^{g_0}$, then $g_0$ is locally optimal for the first positive curl eigenvalue, and the minimal eigenvalue is attained only at~$g_0$.
\end{theorem}

\begin{proof}
	We argue by contradiction. Suppose there exists a sequence of smooth metrics $g_n=f_ng_0$ such that
\begin{enumerate}
\item $(f_n)_n$ converges in the $C^1$-norm to the constant $1$ function,
\item $|M|_{g_n}=|M|_{g_0}$ for all $n$,
\item $\mu_n:=\mu_1^{g_n}\leq \mu_1^{g_0}=:\mu$,
\item $f_n\not\equiv 1$ for all $n$.
\end{enumerate}
For any fixed $n$, we may fix a first eigenfield $\widetilde{X}_n\in \mathfrak X_{\mathrm{ex}}^{g_n}(M)$ and define an isomorphism
\[
\mathcal{I}_n:\mathfrak X_{\mathrm{ex}}^{g_n}(M)\rightarrow \mathfrak X_{\mathrm{ex}}^{g_0}(M)\text{, }\widetilde{X}\mapsto f^\frac{3}{2}\widetilde{X}\,,
\]
which is, in fact, the same isomorphism we use in the proof of Theorem~\ref{ConfT3}. It is easy to check that this isomorphism preserves helicity as well as the $L^\frac{3}{2}$-norm, and so $\mathcal{H}_{g_n}(\widetilde{X})=\mathcal{H}_{g_0}(\mathcal{I}_n(\widetilde{X}))$ and $\|\widetilde{X}\|_{L^\frac{3}{2}_{g_n}}=\|\mathcal{I}_n(\widetilde{X})\|_{L^\frac{3}{2}_{g_0}}$. In what follows we set $X_n:=\mathcal{I}_n(\widetilde{X}_n)$ and scale $\widetilde{X}_n$ such that $\|X_n\|_{L^2_{g_0}}=1$ for all $n$. When no subindex is indicated in the norm or vector space, we shall assume that it corresponds to the metric $g_0$. Using the definition of $X_n$ we see that they are all eigenfields with eigenvalue $\mu_n$ of the following operator
\[
T_n:\mathfrak X_{\mathrm{ex}}^{g_0}(M)\rightarrow \mathfrak X_{\mathrm{ex}}^{g_0}(M)\text{, }X\mapsto \curl _{g_0}\left(\frac{X}{\sqrt{f_n}}\right)\,.
\]
In fact $\mathcal{I}_n$ induces a $1$-$1$ correspondence between the eigenfields of $\curl _{g_n}$ and $T_n$. Let us for notational simplicity define
\[
\lambda_n:=\frac{1}{\mu_n}\text{ and }\lambda:=\frac{1}{\mu}\,.
\]
We observe that the eigenfield property of $X_n$ implies that
\[
\curl ^{-1}_{g_0}(X_n)=\lambda_n\pi\left(\frac{X_n}{\sqrt{f_n}}\right),
\]
where $\pi$ denotes the $L^2$-orthogonal projection from the space of $L^2$-vector fields onto the space $\mathfrak X_{\mathrm{ex}}^{g_0}(M)$. By assumption we have $\mu_n\leq \mu$ or equivalently $\lambda_n\geq \lambda$. We claim that $\lambda_n\leq c$ for some $c>0$ and all $n$ large enough. To see this we let $n$ be so large that $\frac{1}{\sqrt{f_n}}\geq \frac{1}{2}$ which is certainly possibly as $f_n$ converges in $C^0$-norm to~$1$. Keeping in mind our normalization we find
\[
\frac{\lambda_n}{2}=\lambda_n\int_M\frac{g_0(X_n,X_n)}{2}dV_{g_0}\leq\lambda_n\int_M\frac{g_0(X_n,X_n)}{\sqrt{f_n}}dV_{g_0}=
\lambda_n\left(X_n,\pi\left(\frac{X_n}{\sqrt{f_n}}\right)\right)_{L^2}
\]
\[
=\left(\curl ^{-1}_{g_0}(X_n),X_n\right)_{L^2}\leq c_0\|X_n\|^2_{L^2}=c_0\,,
\]
where we have used that $X_n\in \mathfrak X_{\mathrm{ex}}^{g_0}(M)$ and where $c_0=\frac{1}{\min\{|\mu_{-1}^{g_0}|,\mu_1^{g_0}\}}$. We can therefore extract a subsequence of $(g_n)_n$ (denoted in the same way) such that $\lambda_n$ converges to some $\hat{\lambda}$ with $\lambda\leq \hat{\lambda}\leq c$.

Our next goal is to show that $\hat{\lambda}=\lambda$. First, notice that the $H^1$-boundedness of the sequence $\left(\pi\left(\frac{X_n}{\sqrt{f_n}}\right)\right)_n$ follows~\cite[Proposition 2.2.3]{S95} from the existence of a constant $C_F(g_0)>0$ such that
\[
\lambda \left\|\pi\left(\frac{X_n}{\sqrt{f_n}}\right)\right\|_{H^1}\leq \lambda_n\left\|\pi\left(\frac{X_n}{\sqrt{f_n}}\right)\right\|_{H^1}=\|\curl ^{-1}_{g_0}(X_n)\|_{H^1}
\]
\[
\leq C_F(g_0)\|X_n\|_{L^2}=C_F(g_0)\,.
\]
Accordingly, we can extract a subsequence of $(g_n)_n$ such that $\pi\left(\frac{X_n}{\sqrt{f_n}}\right)$ converges strongly in $L^2$ to some $X\in \mathfrak X^{g_0}_{\mathrm{ex}}(M)\cap H^1(M)$ and such that $\pi\left(\frac{X_n}{\sqrt{f_n}}\right)$ converges weakly in $H^1$ to $X$. In particular,  $\curl _{g_0}\left(\pi\left(\frac{X_n}{\sqrt{f_n}}\right)\right)=\frac{X_n}{\lambda_n}$ converges weakly in $L^2$ to $\curl _{g_0} (X)$. We claim that $X_n$ itself converges strongly in $L^2$ to $X$. If this is shown, and since we know that $\lambda_n$ converges to $\hat{\lambda}>0$, this will establish the identity
\begin{gather}
	\label{LPE34}
	\curl _{g_0}(X)=\frac{X}{\hat{\lambda}}\,.
\end{gather}
Indeed, let us decompose $\frac{X_n}{\sqrt{f_n}}=\pi\left(\frac{X_n}{\sqrt{f_n}}\right)+\rho$, where $\rho$ is $L^2$-orthogonal to $\mathfrak X^{g_0}_{\mathrm{ex}}(M)$. We then have
\[
\left\|\pi\left(\frac{X_n}{\sqrt{f_n}}\right)-X_n\right\|^2_{L^2}\leq \|\rho\|^2_{L^2}+\left\|\pi\left(\frac{X_n}{\sqrt{f_n}}\right)-X_n\right\|^2_{L^2}=\left\|\frac{X_n}{\sqrt{f_n}}-X_n\right\|^2_{L^2}\,,
\]
where we used that $X_n\in \mathfrak X_{\mathrm{ex}}^{g_0}(M)$, thus implying
\[
\left\|\frac{X_n}{\sqrt{f_n}}-X_n\right\|_{L^2}\leq \left\|\frac{1}{\sqrt{f_n}}-1\right\|_0\|X_n\|_{L^2}=\left\|\frac{1}{\sqrt{f_n}}-1\right\|_0\,.
\]
The RHS converges to zero as $n\rightarrow \infty$ because $f_n$ converges to the constant $1$ function in $C^0$-norm. Putting together the previous results, we conclude that $X_n$ converges strongly to $X$ in $L^2$, with $\|X\|_{L^2}=1$, so that $\frac{1}{\hat{\lambda}}$ is a (positive) curl eigenvalue according to Equation~\eqref{LPE34}. Since $\frac{1}{\lambda}$ is the smallest positive curl eigenvalue by definition, we must have
\[
\frac{1}{\lambda}\leq \frac{1}{\hat{\lambda}}\Leftrightarrow \hat{\lambda}\leq \lambda\,,
\]
which together with $\hat{\lambda}\geq \lambda$, imply that $\hat{\lambda}=\lambda$. Therefore~\eqref{LPE34} tells us that the vector field $X$ is in fact a first curl eigenfield.

Our next goal is to upgrade the $L^2$-convergence of $X_n$ to $X$ to a $C^0$-convergence. To this end we use the well known estimate~\cite[Proposition 2.2.3]{S95}
\[
\|X_n-X\|_{H^1}\leq C\|\curl (X_n)-\curl (X)\|_{L^2}\,,
\]
(where $C$ is an $n$-independent constant), that $\curl _{g_0}(X)=\mu X$ (a first eigenfield), and
\[
\curl _{g_0}(X_n)=\curl _{g_0}\left(\sqrt{f_n}\frac{X_n}{\sqrt{f_n}}\right)=\sqrt{f_n}\mu_nX_n+\nabla (\sqrt{f_n})\times \frac{X_n}{\sqrt{f_n}}\,.
\]
Using that $f_n$ converges in $C^1$ to the constant $1$ function and the previous convergence properties, we easily infer that the RHS of this identity converges strongly in $L^2$ to $\mu X$, which in turn implies that $X_n$ converges to $X$ in $H^1$. By means of the Sobolev embedding~\cite[Theorem 1.3.6]{S95}, it also follows that $X_n$ converges to $X$ strongly in $L^4$. Now we can use the Calder\'{o}n-Zygmund type estimate~\cite[Lemma 2.4.10 (i)]{S95}
\[
\|X_n-X\|_{W^{1,4}}\leq \widetilde{C}\|\curl (X_n)-\curl (X)\|_{L^4}\,,
\]
and arguing as before we deduce that the RHS of this inequality can be upper bounded by $\|X_n-X\|_{L^4}$, which converges to zero, so that $X_n$ indeed converges to $X$ in $W^{1,4}$. Then Morrey's inequality implies the convergence with respect to the $C^0$-norm.

In summary, we have established that there exists a first eigenfield $X$ of $\curl _{g_0}$ and a sequence of first eigenfields $(\widetilde{X}_n)_n$ corresponding to the metrics $g_n$, such that the vector fields $X_n=\mathcal{I}_n(\widetilde{X}_n)$ converge to $X$ in the $C^0$-norm. Then Theorem~\ref{LPT14} shows that there exists a $C^0$-neighbourhood around $X$ such that for all $X_n$ within this neighbourhood we have the inequality
\[
\frac{\|X\|^2_{L^\frac{3}{2}}}{\mathcal{H}_{g_0}(X)}\leq \frac{\|X_n\|^2_{L^\frac{3}{2}}}{\mathcal{H}_{g_0}(X_n)}\,,
\]
with equality if and only if $X_n$ is a first eigenfield with respect to $g_0$ itself. Since the $X_n$ converge to $X$ in $C^0$ we may assume that the above inequality holds for all $n$. We now have the following estimate by means of the H\"{o}lder inequality (where equality holds if and only if $|\widetilde{X}_n|_{g_n}$ is constant)
\[
\mu_n=\frac{\|\widetilde{X}_n\|^2_{L^2_{g_n}}}{\mathcal{H}_{g_n}(\widetilde{X}_n)}\geq \frac{\|\widetilde{X}\|^2_{L^\frac{3}{2}_{g_n}}}{\sqrt[3]{|M|_{g_n}}\mathcal{H}_{g_n}(\widetilde{X}_n)}=
\frac{\|X_n\|^2_{L^\frac{3}{2}}}{\sqrt[3]{|M|_{g_0}}\mathcal{H}_{g_0}(X_n)}\,,
\]
which combined with the inequality above yields
\[
\mu_n\geq \frac{\|X\|^2_{L^\frac{3}{2}}}{\sqrt[3]{|M|_{g_0}}\mathcal{H}_{g_0}(X)}\,,
\]
with equality if and only if $\widetilde{X}_n$ has constant $g_n$-norm and $X_n$ is a first eigenfield with respect to $g_0$. Using the assumptions of the theorem and that $X$ is a first eigenfield of $g_0$, we easily conclude that
\[
\mu_n\geq \mu\,.
\]
Finally, assume that $\mu_n=\mu$ for some $n$. Then, recalling that $\widetilde{X}_n=\frac{X_n}{f^\frac{3}{2}_n}$ and $|X_n|_{g_0}=\text{const}$, we infer that $\widetilde{X}_n$ has constant $g_n$-norm if and only if $f_n$ itself is constant, and by the volume constraint $f_n\equiv 1$. But by construction of our sequence of metrics $(g_n)_n$, none of the $f_n$ is a constant function, and therefore we see that our sequence must contain at least one element with $\mu_n>\mu$, which is a contradiction (because $\mu_n\leq \mu$ by construction). Hence such a sequence cannot exist and the theorem follows.
\end{proof}

Arguing exactly as in the proof of Corollary~\ref{cor.RP}, Theorem~\ref{LPT20} yields that the canonical metric $\gcan $ of the real projective space $\RR P^3$ is locally optimal for the first (positive and negative) curl eigenvalue.

\begin{corollary}
The real projective space $(\RR P^3,\gcan )$ is locally optimal (in its conformal class of the same volume) for the first positive and negative curl eigenvalues, and the minimal eigenvalue is attained only at~$\gcan $. Moreover, $(\RR P^3,\gcan )$ is locally optimal for the Hodge eigenvalue $\la_1^{1,\de}$.
\end{corollary}

\section{Local optimality of $(\SS^3,\gcan)$}\label{S.sphere}

In this section we consider $\mathbb S^3$ endowed with the round metric. It is well known that the vector fields
\begin{equation}
	\label{ConfE15}
	B_1=(-x_2,x_1,-x_4,x_3)\text{, }B_2=(-x_3,x_4,x_1,-x_2)\text{, }B_3=(-x_4,-x_3,x_2,x_1)
\end{equation}
form a basis of the smallest positive curl eigenvalue $2$ on $\mathbb S^3$, and that the curl eigenvalues are of the form $\pm(k+2)$ with $k\in \mathbb{N}$, see e.g.~\cite[Proposition 1]{PR}. All the first curl eigenfields have constant pointwise norm, but the second positive curl eigenvalue is~$3$, which is smaller than twice the first eigenvalue $2$, so Theorem~\ref{LPT20} cannot be applied to conclude the local minimality of the round metric. Accordingly, a much more involved strategy is needed to address this problem.

Our goal is to prove that $(\SS^3,\gcan )$ is locally optimal for the first (positive and negative) curl eigenvalue, which in turn implies the local optimality of the Hodge eigenvalue $\la_1^{1,\delta}$. In the proof we shall make use of the following elementary lemmas, whose proofs are omitted. Since we will be concerned with the metric $\gcan $, the explicit dependence of operators, norms, spaces, etc. with respect to the metric will be omitted in this section.

\begin{lemma}
The eigenspace of curl corresponding to the eigenvalue $\mu_{-1}=-2$ is $3$-dimensional and is spanned by the anti-Hopf vector fields:
\begin{gather}\label{ConfE56}
\hat{B}_1:=(-x_4,x_3,-x_2,x_1)\text{,   }\hat{B}_2:=(-x_3,-x_4,x_1,x_2)\text{,   }\hat{B}_3:=(-x_2,x_1,x_4,-x_3)
\end{gather}
Moreover, the eigenspace of the eigenvalue $\mu_2=3$ is $8$-dimensional and an $L^2$-orthonormal basis is given by
	\begin{gather}
		\label{ConfE16}
		u_1:=\frac{x_1B_2-x_2B_3}{\pi},\text{ }u_2:=\frac{x_2B_2+x_1B_3}{\pi},\text{ }u_3:=\frac{x_3B_2-x_4B_3}{\pi} \\
		\nonumber
		u_4:=\frac{x_4B_2+x_3B_3}{\pi}\text{, }u_5:=\frac{x_3B_2+x_4B_3-2x_2B_1}{\sqrt{3}\pi}\text{, }u_6:=\frac{-x_4B_2+x_3B_3+2x_1B_1}{\sqrt{3}\pi} \\
		\nonumber
		u_7:=\frac{x_2B_2-x_1B_3+2x_3B_1}{\sqrt{3}\pi}\text{, }u_8:=\frac{x_1B_2+x_2B_3+2x_4B_1}{\sqrt{3}\pi}.
	\end{gather}
\end{lemma}
\begin{lemma}
Let $V$ be a smooth vector field on $\mathbb S^3$, then we can express it as $V=\sum_{j=1}^3f_jB_j$ for suitable smooth functions $f_j\in C^{\infty}(\mathbb S^3)$ and the following identity holds:
	\[
	\curl \curl V=\nabla \Div V +\sum_{j=1}^3(\Delta_gf_j)B_j+2\curl V \,.
	\]
From this we conclude that if $V$ is an eigenfield of curl, corresponding to a non-zero eigenvalue $\kappa$, then the coefficient functions $f_j$ are Laplacian eigenfunctions corresponding to the eigenvalue $\kappa(\kappa-2)$.
\end{lemma}

\subsection{The main argument}

The proof of the main result of this section is lengthy and technical, so let us first comment on its strategy. To streamline the presentation, the proofs of several technical results will be relegated to a later subsection.

The main idea is to use a Taylor expansion to establish the $C^0$-local $L^{\frac{3}{2}}$-minimality of the first curl eigenfields, which implies the local optimality of the round metric (for the first curl eigenvalue) arguing as in the proof of Theorem~\ref{LPT20}. However, in contrast to the situation $\mu_2\geq 2\mu_1$, the second order derivative of the corresponding functional turns out to be positive semi-definite but not (strictly) definite, so that we cannot conclude. Instead, the idea is to consider a higher order Taylor expansion and, as we shall see, some sort of dichotomy will occur. We will show that there exists a constant $\epsilon_0>0$ such that if certain $L^2$-orthogonal projections are $\epsilon_0$-close to each other (see Equation~\eqref{E}), then a Taylor expansion up to sixth order will imply that in a small enough neighbourhood $B_r(X)$ ($r$ depends on $\epsilon_0$) at a fixed first curl eigenfield $X$ we have $\mathcal{F}(Y)\leq \mathcal{F}(X)$, with equality if and only if $Y$ is itself a first curl eigenfield. On the other hand, we also show that if the aforementioned $\epsilon_0$-closeness is violated, then a fourth order Taylor expansion is already enough to establish the minimization property. Here $\cF$ is a certain functional related to the $L^{\frac{3}{2}}$-energy, which is very convenient in order to compute higher order derivatives.

The following theorem establishes the first part of Theorem~\ref{T.main}. In the proof we only deal with the first positive curl eigenvalue, the case of the first negative eigenvalue following from the orientation-reversing isometry introduced in Equation~\eqref{Eq.isom} as used in the proof of Corollary~\ref{cor.RP}.

\begin{theorem}\label{S3C2}
The round sphere $(\mathbb S^3,\gcan )$ is locally optimal for the first (positive and negative) curl eigenvalue (in its conformal class of the same volume), and the minimal eigenvalue is attained only at $\gcan $. Moreover, $(\mathbb S^3,\gcan )$ is locally optimal for the Hodge eigenvalue $\la_1^{1,\delta}$.
\end{theorem}

\begin{proof}

Let us consider the functional $\mathcal{F}:\left(U,\|\cdot\|_0\right)\rightarrow \mathbb{R}$ defined as $$\cF(X)=\frac{(\mathcal{E}(X))^{\frac{4}{3}}}{\mathcal{H}(X)}\,,$$
where $U\subset C^0\mathfrak X^{g_0}_{\mathrm{ex}}(\mathbb{S}^3)$ is an open subset consisting of exact vector fields with positive helicity. Using identical arguments as in the proof of Theorem~\ref{LPT20} we deduce that the local optimality of the round metric on $\SS^3$ follows if we can prove that any first eigenfield $X\in E_1$ is a $C^0$-local minimizer for the functional $\mathcal{F}$, and any minimizer is a first eigenfield.

To prove the aforementioned $C^0$-local minimality, we compute for $W\in C^0\mathfrak X^{g_0}_{\mathrm{ex}}(\mathbb{S}^3)$
\begin{gather}
		\label{SE1}
		 D\mathcal{F}(X)(W)=\frac{4}{3}(\mathcal{E}(X))^{\frac{1}{3}}\frac{D\mathcal{E}(X)(W)}{\mathcal{H}(X)}-\frac{(\mathcal{E}(X))^{\frac{4}{3}}}{(\mathcal{H}(X))^2}
D\mathcal{H}(X)(W)\,, \\
		\label{SE2}
		D^2\mathcal{F}(X)(W,W)=\frac{4}{9}\left(\mathcal{E}(X)\right)^{-\frac{2}{3}}\frac{\left(D\mathcal{E}(X)(W)\right)^2}{\mathcal{H}(X)} \\
		\nonumber
		+\frac{4}{3}\frac{\left(\mathcal{E}(X)\right)^{\frac{1}{3}}}{\mathcal{H}(X)}D^2\mathcal{E}(X)(W,W)
		-\frac{8}{3}\frac{\left(\mathcal{E}(X)\right)^{\frac{1}{3}}}{(\mathcal{H}(X))^2}D\mathcal{E}(X)(W)D\mathcal{H}(X)(W)
		\\
		\nonumber
		 -\frac{(\mathcal{E}(X))^{\frac{4}{3}}}{(\mathcal{H}(X))^2}D^2\mathcal{H}(X)(W,W)+2\frac{(\mathcal{E}(X))^{\frac{4}{3}}}{(\mathcal{H}(X))^3}
\left(D\mathcal{H}(X)(W)\right)^2\,.
	\end{gather}

Next we observe that $\mathcal{H}$ is a quadratic functional and so $D^k\mathcal{H}(X)=0$ for all $k\geq 3$. In fact, it is straightforward to check that
\begin{equation}
		\label{SE3}
		D\mathcal{H}(X)(W)=2\left(\operatorname{curl}^{-1}(X),W\right)_{L^2}\text{, }D^2\mathcal{H}(X)(W,W)=2\mathcal{H}(W)\,.
\end{equation}

Let $X\in E_1$ be a fixed first curl eigenfield and consider $Y\in B_r(X)$; we recall that $B_r(X)$ is a $C^0$-neighborhood of radius $r$ at $X$. We decompose $Y$ as $Y=Y_1+W$, with $Y_1\in E_1$ being the $L^2$-orthogonal projection of $Y$ into the finite dimensional space $E_1$. By means of Taylor's theorem we find
	\begin{gather}
		\label{SE4}
		\mathcal{F}(Y)=\mathcal{F}(Y_1)+D\mathcal{F}(Y_1)(W)+\frac{D^2\mathcal{F}(Y_1)(W,W)}{2!}
		\\
		\nonumber
		+\frac{D^3\mathcal{F}(Y_1)(W,W,W)}{3!}+\frac{D^4\mathcal{F}(Y_1)(W,W,W,W)}{4!}\\
		\nonumber
		+\frac{D^5\mathcal{F}(Y_1)(W,W,W,W,W)}{5!}+\frac{D^6\mathcal{F}(Y_1)(W,W,W,W,W,W)}{6!}
		\\
		\nonumber
		+\frac{D^6\mathcal{F}(Z)(W,W,W,W,W,W)-D^6\mathcal{F}(Y_1)(W,W,W,W,W,W)}{6!}\,,
	\end{gather}
where the last summand is the remainder term with $Z:=Y_1+\lambda W$ for some $\lambda\in [0,1]$.

Since $\mathcal{F}(\alpha X)=\mathcal{F}(X)$ for all $\alpha\neq 0$ and $X$ has constant pointwise norm, we can safely assume that $|X|=1$ everywhere. Moreover, arguing as in the proof of Theorem~\ref{LPT14}, we infer that $\mathcal{F}(Y_1)=\mathcal{F}(X)$ and that $Y_1$ is $C^0$-close to $X$ and $W$ is $C^0$-close to zero. Thus, after scaling $Y$ by $\frac{1}{|Y_1|}$ we may suppose that $|Y_1|=1$, and one readily checks that this scaled version of $Y$ is contained in $B_{3r}(X)$. Lastly, since for every normalized first curl eigenfield $Y_1$ we may find an isometry of $\mathbb{S}^3$ which carries $Y_1$ into the Hopf field $B_1$, we assume from now on that $Y_1=B_1$.

Now we observe that $D\mathcal{H}(B_1)(W)=2\left(\operatorname{curl}^{-1}(B_1),W\right)_{L^2}=\left(B_1,W\right)_{L^2}=0$ and $D\mathcal{E}(B_1)(W)=0$ (using Lemma~\ref{LPL18}), and hence
	\[
	D\mathcal{F}(B_1)(W)=0
	\]
	and
	\[
	 D^2\mathcal{F}(B_1)(W,W)=\frac{4}{3}\frac{\left(\mathcal{E}(B_1)\right)^{\frac{1}{3}}}{\mathcal{H}(B_1)}D^2\mathcal{E}(B_1)(W,W)-
2\frac{\left(\mathcal{E}(B_1)\right)^{\frac{4}{3}}}{(\mathcal{H}(B_1))^2}\mathcal{H}(W)\,.
	\]
	Using again Lemma~\ref{LPL18} we can write
	\[
	D^2\mathcal{F}(B_1)(W,W)=\frac{\left(\mathcal{E}(B_1)\right)^{\frac{1}{3}}}{\mathcal{H}(B_1)}\left(2\|W\|^2_{L^2}-\int_{\mathbb{S}^3}(B_1\cdot W)^2dV_{g_0}-4\mathcal{H}(W)\right)\,,
	\]
where we have used that $\frac{\mathcal{E}(B_1)}{\mathcal{H}(B_1)}=2$.

Next, let us decompose the vector field $W$ as follows. Let $E:=E_{-1}\oplus E_3$ be the direct sum of the first negative and third positive curl eigenspaces, and let $W_E$,$W_2$ denote the $L^2$-orthogonal projections of $W$ onto $E$ and $E_2$, respectively. We then write
$$W=W_E+W_2+\hat{W}_0\,,$$
and we can easily check that
	\begin{gather}
		\nonumber
	D^2\mathcal{F}(B_1)(W,W)=
	\frac{\left(\mathcal{E}(B_1)\right)^{\frac{1}{3}}}{\mathcal{H}(B_1)}\left(2\|\hat{W}_0\|^2_{L^2}-\int_{\mathbb{S}^3}(B_1\cdot \hat{W}_0)^2dV_{g_0}-4\mathcal{H}(\hat{W}_0)\right)
	\\
	\nonumber
	+\frac{\left(\mathcal{E}(B_1)\right)^{\frac{1}{3}}}{\mathcal{H}(B_1)}\left(2\|W_E\|^2_{L^2}-\int_{\mathbb{S}^3}(B_1\cdot W_E)^2dV_{g_0}-4\mathcal{H}(W_E)\right)
	\\
	\nonumber
	+\frac{\left(\mathcal{E}(B_1)\right)^{\frac{1}{3}}}{\mathcal{H}(B_1)}\left(2\|W_2\|^2_{L^2}-\int_{\mathbb{S}^3}(B_1\cdot W_2)^2dV_{g_0}-4\mathcal{H}(W_2)\right)\,
	\end{gather}
and we have the rough standard estimate
\[
2\|\hat{W}_0\|^2_{L^2}-\int_{\mathbb{S}^3}(B_1\cdot \hat{W}_0)^2dV_{g_0}-4\mathcal{H}(\hat{W}_0)\geq\frac{\|\hat{W}_0\|^2_{L^2}}{5}\,,
\]
which is used in the proof of Proposition~\ref{PE1}. We can now write $W_2=Z_1+Z_2$ with $Z_1=\sum_{i=1}^4a_iu_i$ and $Z_2=\sum_{i=5}^8a_iu_i$ for suitable $a_i\in \mathbb{R}$, where the vector fields $u_i$ are defined in Equation~\eqref{ConfE16}. By Lemma~\ref{LE1} below we can assume that $a_6=a_7=0$. A direct calculation yields
	\[
	2\|W_2\|^2_{L^2}-\int_{\mathbb{S}^3}(B_1\cdot W_2)^2dV_{g_0}-4\mathcal{H}(W_2)=\frac{2}{3}\|Z_1\|^2_{L^2}\,,
	\]
and so we obtain
	\begin{gather}
		\label{SE5}
		D^2\mathcal{F}(B_1)(W,W)= \frac{\left(\mathcal{E}(B_1)\right)^{\frac{1}{3}}}{\mathcal{H}(B_1)} \frac{2\|Z_1\|^2_{L^2}}{3}\\
		\nonumber
		+\frac{\left(\mathcal{E}(B_1)\right)^{\frac{1}{3}}}{\mathcal{H}(B_1)}\left(2\|W_E\|^2_{L^2}-\int_{\mathbb{S}^3}(B_1\cdot W_E)^2dV_{g_0}-4\mathcal{H}(W_E)\right)\,
		\\
		\nonumber
		+\frac{\left(\mathcal{E}(B_1)\right)^{\frac{1}{3}}}{\mathcal{H}(B_1)}\left(2\|\hat{W}_0\|^2_{L^2}-\int_{\mathbb{S}^3}(B_1\cdot \hat{W}_0)^2dV_{g_0}-4\mathcal{H}(\hat{W}_0)\right)\,.
	\end{gather}

Our aim is to establish the inequality
\[
\mathcal{F}(Y)\geq \mathcal{F}(X)+\sigma \left(\|\hat{W}_0\|^2_{L^2}+\|Z_1\|^2_{L^2}+\|W_{-1}\|^2_{L^2}+\|W_3\|^3_{L^2}+\|Z_2\|^6_{L^2}\right)
\]
for some constant $\sigma>0$ and small radius $0<r\ll 1$, where we have further decomposed $W_E=W_{-1}+W_3$ into its projection on the eigenspaces $E_{-1}$ and $E_3$. Of course, if this bound is proved, the $L^{\frac32}$-minimality and hence the proposition would follow. More precisely, writing $W_3=\sum_{i=1}^{15}b_iv_i$ in terms of the $L^2$-orthonormal basis $v_i$ of $E_3$ introduced in Lemma~\ref{LE2} below, we will establish the inequality:
\begin{gather}
	\nonumber
\mathcal{F}(Y)\geq \mathcal{F}(X)+\sigma \left(\|\hat{W}_0\|^2_{L^2}+\|Z_1\|^2_{L^2}+\|W_{-1}\|^2_{L^2} \right.
\\
\nonumber
\left.+\sum_{i\neq 10,12,15}b^2_i+|b_{10}|^3+|b_{12}|^3+|b_{15}|^3+\|Z_2\|^6_{L^2}\right)\,.
\end{gather}
In the following we let $P^1_3$ be the projection of $W_3$ into the span of the $v_i$, $i\neq 10,12,15$, and $P^2_3$ be the remainder $W_3-P^1_3$, and we shall always discard terms which are of order
\begin{align}\label{eqapprox}
\nonumber o(1)\Bigg(\|\hat{W}_0\|^2_{L^2}+\|Z_1\|^2_{L^2}+&\|W_{-1}\|^2_{L^2}\\
&+\sum_{i\neq 10,12,15}b^2_i+|b_{10}|^3+|b_{12}|^3+|b_{15}|^3+\|Z_2\|^6_{L^2} \Bigg)
\end{align}
as $r\searrow 0$, and denote it by using the symbol $\approx$.

To make the proof more tractable we move the discussion of the $4$th order Taylor expansion to the next section. In particular, taking these results for granted it follows from Proposition~\ref{PE1} below that $\mathcal{F}(Y)\geq \mathcal{F}(X)$ for some small enough $r>0$ (which will depend on the upcoming constant $\epsilon_0>0$), with equality if and only if $Y\in E_1$, unless the following relations are satisfied
\begin{gather}
	\nonumber
b_{15}=\frac{a^2_5+a^2_8}{2\sqrt{3}\pi}(1+\delta_{15})\text{, }a^2_5-a^2_8=\frac{6\pi}{\sqrt{7}}b_{12}(1+\delta_{12})\text{, }\\
\label{E}
a_5a_8=-\frac{3\pi}{\sqrt{7}}b_{10}(1+\delta_{10})\,,
\end{gather}
where $|\delta_{15}|,|\delta_{12}|,|\delta_{10}|\leq \epsilon_0$ and $\epsilon_0>0$ is an absolute constant which may be chosen as close to zero as one wishes\footnote{There are strictly speaking two more situations to consider, namely if~\eqref{E1} is satisfied and one of the conditions~\eqref{SE22} or~\eqref{SE24} is satisfied while the other is violated. But these situations can be reduced to the situation we are about to consider so that we omit their discussion. \label{foot1}}. In the following we will assume that $\delta_i=0$ for $i=10,12,15$ because we will establish a strict inequality so that by continuity the obtained inequality will persist upon fixing an $\epsilon_0>0$ close enough to zero.

Now Corollary \ref{CE2} tells us that (in the notation of the corollary, $\rho=0$ if $\delta_i=0$), with $W_0:=\hat{W}_0+Z_1+W_{-1}+P_3^1$, we have
\begin{gather}
\nonumber
6D\mathcal{F}(B_1)(W)+3D^2\mathcal{F}(B_1)(W,W)+D^3\mathcal{F}(B_1)(W,W,W)+\frac{D^4\mathcal{F}(B_1)(W,W,W,W)}{4}
\\
\nonumber
+\frac{D^5\mathcal{F}(B_1)(W,W,W,W,W)}{20}+\frac{D^6\mathcal{F}(B_1)(W,W,W,W,W,W)}{120}
\\
\nonumber
\approx\frac{(\mathcal{E}(B_1))^\frac{1}{3}}{\mathcal{H}(B_1)}\left(\frac{11}{32\pi^4}(a^2_5+a^2_8)^3+\int_{\mathbb S^3}W_0\cdot R dV_{g_0}+I\right)\,,
\end{gather}
where the vector field $R$ is defined as
\begin{gather}
	\nonumber
R:=\left(-6P^2_3\cdot Z_2+15(P^2_3\cdot B_1)(B_1\cdot Z_2)-\frac{45}{4}(B_1\cdot Z_2)^3+\frac{15}{2}|Z_2|^2(B_1\cdot Z_2)\right)B_1
\\
\nonumber
+\left(-6(B_1\cdot P^2_3)-3|Z_2|^2+\frac{15}{2}(B_1\cdot Z_2)^2\right)Z_2-6(B_1\cdot Z_2)P^2_3\,,
\end{gather}
and $I$ satisfies the following lower bound (and we further decompose $\hat{W}_0=W_{-2}+W_{-3}+W_4+W_5+W_R$, where the $W_i$ are the $L^2$-orthogonal projections into the eigenspaces $E_i$ and $W_R$ is the corresponding remainder term)
\begin{gather}
	\nonumber
	I\geq 2\|Z_1\|^2_{L^2}+2\|W_{-1}\|^2_{L^2}+\frac{2}{3}\sum_{\substack{i=1 \\ i\neq 10,12,15}}^{15}b^2_i
	\\
	\label{EE2}
	+\|W_{-3}\|^2_{L^2}+\|W_5\|^2_{L^2}+\frac{9}{7}\|W_R\|^2_{L^2}+\frac{27}{20}\|W_4\|^2_{L^2}+5\|W_{-2}\|^2_{L^2}\,,
\end{gather}
for a proof see Corollary~\ref{CE2} and Lemma~\ref{LE8}. Notice that the vector field $R$ is not divergence-free, while the vector field $W_0$ is, so that we may replace $R$ by its $L^2$-orthogonal projection $C$ into the divergence-free fields:
\[
\int_{\mathbb S^3}W_0\cdot R\, dV_{g_0}=\int_{\mathbb S^3}W_0\cdot C\,dV_{g_0}=\int_{\mathbb S^3}(\hat{W}_0+Z_1)\cdot C\,dV_{g_0}\,,
\]
where we have used that $\left(W_{-1},C\right)_{L^2}=\left(P^1_3,C\right)_{L^2}=0$.

Next we note that the functions $C\cdot B_i$ for $i=1,2,3$ define homogenous polynomials of degree $3$ on $\mathbb{R}^4$ (this follows from the expression of $C$ in the proof of Lemma~\ref{LE5}), and that the functions $B_i\cdot W_5$, $B_i\cdot W_{-3}$ define homogeneous polynomials of degree $4$ on $\mathbb{R}^4$. One can then verify by elementary computations that $\int_{S^3}C\cdot (W_5+W_{-3})\,dV_{g_0}=0$, so that we have the identity
\begin{gather}
	\nonumber
	\int_{\SS^3}W_0\cdot R\, dV_{g_0}=\int_{\SS^3}W_0\cdot C\,dV_{g_0}=\int_{\SS^3}(W_{-2}+W_4+W_R+Z_1)\cdot C\,dV_{g_0}\,,
\end{gather}
and consequently
\begin{gather}
	\nonumber
\left|\int_{\SS^3}W_0\cdot R\, dV_{g_0}\right|\leq \|C\|_{L^2}\|W_{-2}+W_4+W_R+Z_1\|_{L^2}
\\
\nonumber
\leq \frac{\|C\|^2_{L^2}}{4\delta}+\|W_{-2}+W_4+W_R+Z_1\|^2_{L^2}\delta\,
\\
\nonumber
=\frac{\|C\|^2_{L^2}}{4\delta}+\left(\|W_{-2}\|^2_{L^2}+\|W_4\|^2_{L^2}+\|W_R\|^2_{L^2}+\|Z_1\|^2_{L^2}\right)\delta\,,
\end{gather}
for any $\delta>0$. Then we infer from Lemma~\ref{LE5} below that
\[
\frac{\|C\|^2_{L^2}}{4\delta}=\frac{151}{360\pi^4\delta}\|Z_2\|^6_{L^2}\,,
\]
so setting $\delta=\frac{5}{4}$, we therefore obtain
\[
\left|\int_{\SS^3}W_0\cdot R\, dV_{g_0}\right|\leq \frac{5}{4}\left(\|W_{-2}\|^2_{L^2}+\|W_4\|^2_{L^2}+\|W_R\|^2_{L^2}+\|Z_1\|^2_{L^2}\right)+\frac{151}{450\pi^4}\|Z_2\|^6_{L^2}\,.
\]

Plugging this bound into our previous expression of the sixth order expansion we get the estimate
\begin{gather}
	\nonumber
	6D\mathcal{F}(B_1)(W)+3D^2\mathcal{F}(B_1)(W,W)+D^3\mathcal{F}(B_1)(W,W,W)+\frac{D^4\mathcal{F}(B_1)(W,W,W,W)}{4}
	\\
	\nonumber
	+\frac{D^5\mathcal{F}(B_1)(W,W,W,W,W)}{20}+\frac{D^6\mathcal{F}(B_1)(W,W,W,W,W,W)}{120}
	\\
	\nonumber
	 \gtrsim\frac{(\mathcal{E}(B_1))^\frac{1}{3}}{\mathcal{H}(B_1)}\left(\frac{59}{7200\pi^4}(a^2_5+a^2_8)^3-
\frac{5}{4}\left(\|W_{-2}\|^2_{L^2}+\|W_4\|^2_{L^2}+\|W_R\|^2_{L^2}+\|Z_1\|^2_{L^2}\right)+I\right)\,,
\end{gather}
which together with the lower bound~\eqref{EE2} on $I$, it yields
\begin{gather}
	\nonumber
	6D\mathcal{F}(B_1)(W)+3D^2\mathcal{F}(B_1)(W,W)+D^3\mathcal{F}(B_1)(W,W,W)+\frac{D^4\mathcal{F}(B_1)(W,W,W,W)}{4}
	\\
	\nonumber
	+\frac{D^5\mathcal{F}(B_1)(W,W,W,W,W)}{20}+\frac{D^6\mathcal{F}(B_1)(W,W,W,W,W,W)}{120}
	\\
	\nonumber
	 \gtrsim\frac{(\mathcal{E}(B_1))^\frac{1}{3}}{\mathcal{H}(B_1)}\left(\frac{(a^2_5+a^2_8)^3+\|\hat{W}_0\|^2_{L^2}+\|Z_1\|^2_{L^2}+\|W_{-1}\|^2_{L^2}+\sum_{\substack{i=1 \\ i\neq 10,12,15}}^{15}b^2_i}{(10\pi)^4}\right)\,.
\end{gather}
The symbol $\gtrsim$ means similarly as $\approx$ introduced in Equation~\eqref{eqapprox}, that we drop lower order terms in the inequality.

We recall that we assumed the missing coefficients $b_{10},b_{12}$ and $b_{15}$ to satisfy the relations in (\ref{E}) and that control of $(a^2_5+a^2_8)^3$ gives us automatically control of $|a^2_5-a^2_8|^3$ and $|a_5a_8|^3$ so that, upon fixing a suitable $\epsilon_0>0$ close enough to zero, we eventually arrive at an estimate of the following form
\[
\mathcal{F}(Y)\gtrsim \mathcal{F}(B_1)+\frac{\|W_0\|^2_{L^2}+\|P^2_3\|^3_{L^2}+\|Z_2\|^6_{L^2}}{(100\pi)^4}
\]
\[
+\frac{D^6\mathcal{F}(Z)(W,W,W,W,W,W)-D^6\mathcal{F}(Y_1)(W,W,W,W,W,W)}{6!}\,,
\]
with $W_0=\hat{W}_0+Z_1+W_{-1}+P^1_3$, provided the relations in~\eqref{E} are satisfied. The error term in this sixth order expansion can also be seen to be of order $o(1)\left(\|W_0\|^2_{L^2}+\|P^2_3\|^3_{L^2}+\|Z_2\|^6_{L^2}\right)$ as $r\searrow 0$, and since $\mathcal{F}(B_1)=\mathcal{F}(X)$ we finally see that
\[
\mathcal{F}(Y)\gtrsim \mathcal{F}(X)+\frac{\|W_0\|^2_{L^2}+\|P^2_3\|^3_{L^2}+\|Z_2\|^6_{L^2}}{(100\pi)^4}\,,
\]
which establishes the desired minimality.

In the case that the conditions~\eqref{E} are violated, since we have fixed a suitable $\epsilon_0>0$, it follows from Proposition~\ref{PE1} that
\[
\mathcal{F}(Y)\geq \mathcal{F}(X)
\]
 provided that $r=r(\epsilon_0)$ is small enough, with equality if and only if $Y\in E_1$.

Altogether we conclude that upon fixing a suitable $\epsilon_0>0$ there is $r=r(\epsilon_0)>0$ such that if $Y\in B_r(X)$, then $\mathcal{F}(Y)\geq \mathcal{F}(X)$ with equality if and only if $Y=B_1$. This completes the proof of the theorem.
\end{proof}

\subsection{Some technical lemmas}

In this long subsection we establish several auxiliary technical results that are used in the proof of Theorem~\ref{S3C2}. In what follows we use the notation and objects introduced in the proof of the aforementioned theorem without further mention (in particular, we employ the asymptotic approximation symbol $\approx$ introduced in Equation~\eqref{eqapprox}).

\begin{lemma}\label{LE1}
For any given $Z_1\in \operatorname{span}\{u_1,u_2,u_3,u_4\}$ and $Z_2\in \operatorname{span}\{u_5,u_6,u_7,u_8\}$, there exists an isometry $\mathcal{O}$ of $\mathbb{S}^3$ with $\mathcal{O}_{*}B_1=B_1$, $\mathcal{O}_{*}Z_1\in \operatorname{span}\{u_1,u_2,u_3,u_4\}$ and $\mathcal{O}_{*}Z_2\in \operatorname{span}\{u_5,u_8\}$.
\end{lemma}
\begin{proof}
We observe that for every $\alpha,\beta\in \mathbb{R}$, the following matrix
	\[
	\mathcal{O}:=\begin{pmatrix}
		\cos(\alpha) & \sin(\alpha) & 0 & 0 \\
		-\sin(\alpha) & \cos(\alpha) & 0 & 0 \\
		0  & 0  & \cos(\beta) & \sin(\beta)  \\
		0 & 0 & -\sin(\beta) & \cos(\beta)
	\end{pmatrix}
	\]
acting on vectors of $\RR^4$, induces isometries of $\SS^3$ with the following additional properties:
	\[
	\mathcal{O}_{*}B_1=B_1\,,
	\]
	\[
	\mathcal{O}_{*}Z_2=\left(a_5\cos(\alpha)+a_6\sin(\alpha)\right)u_5+\left(a_6\cos(\alpha)-a_5\sin(\alpha)\right)u_6
	\]
	\[
	+\left(a_7\cos(\beta)+a_8\sin(\beta)\right)u_7+\left(a_8\cos(\beta)-a_7\sin(\beta)\right)u_8\,,
	\]
	\[
	\mathcal{O}_{*}Z_1\in \text{span}\{u_1,u_2,u_3,u_4\}\,.
	\]
Accordingly, for the right choice of angles $\alpha,\beta$, we can always force that the components of $\mathcal{O}_{*}Z_2$ in the directions $u_6$ and $u_7$ vanish, as we wanted to show.
\end{proof}

The next lemma provides an approximate expression (up to terms of order $o(1)$) of the third derivative of the functional $\cF$:

\begin{lemma}\label{SE6}
	We have
\begin{align*}
&D^3\mathcal{F}(B_1)(W,W,W)\approx \frac{(\mathcal{E}(B_1))^\frac{1}{3}}{2\mathcal{H}(B_1)}\Bigg(\int_{\SS^3}(P^2_3\cdot B_1)\left(-6|P^2_3|^2+5(B_1\cdot P^2_3)^2\right)dV_{g_0}\\
&+\int_{\SS^3}\left(15(B_1\cdot W_3)(B_1\cdot Z_2)^2-12(W_3\cdot Z_2)(B_1\cdot Z_2)-6(B_1\cdot W_3\right)|Z_2|^2)dV_{g_0}\\
&+\int_{\SS^3}W_0\cdot R_1dV_{g_0}\\
&-12\int_{S^3}(B_1\cdot Z_2)\left((B_2\cdot W_{-1})(B_2\cdot Z_2)+(B_3\cdot W_{-1})(B_3\cdot Z_2)\right)dV_{g_0}\Bigg)\,,
\end{align*}
with $R_1:=-12(B_1\cdot Z_2)P^2_3-12(B_1\cdot P^2_3)Z_2+(-12(P^2_3\cdot Z_2)+30(P^2_3\cdot B_1)(Z_2\cdot B_1))B_1$ and $W_0:=\hat{W}_0+Z_1+W_{-1}+P_3^1$.
\end{lemma}
\begin{proof}
A first observation is that not only $W$ is $C^0$-close to $0$, but all of the $\hat{W}_0,Z_1,W_{-1},b_iv_i,Z_2$ are $C^0$-close to $0$ individually, which follows from the $L^2$-orthogonality of the decomposition and the fact that the $C^0$-norm and the $L^2$-norm are equivalent on the (finite dimensional) eigenspaces.

Next, using that the first derivatives of $\mathcal{E}$ and $\mathcal{H}$ vanish at $B_1$, we infer that the only terms in $D^3\mathcal{F}(B_1)$ are those which contain third order derivatives acting on $\mathcal{E}$ (recall $D^k\mathcal{H}\equiv 0$ for all $k\geq3$), and thus we find
\[
D^3\mathcal{F}(B_1)(W,W,W)=\frac{4}{3}\frac{\left(\mathcal{E}(B_1)\right)^{\frac{1}{3}}}{\mathcal{H}(B_1)}D^3\mathcal{E}(B_1)(W,W,W)\,.
\]
An additional straightforward computation yields
\[
D^3\mathcal{E}(B_1)(W,W,W)=\frac{3}{8}\left(-6\int_{\mathbb{S}^3}(B_1\cdot W)|W|^2 dV_{g_0}+5\int_{\mathbb{S}^3}(B_1\cdot W)^3dV_{g_0}\right)\,.
\]
Letting $W_0:=\hat{W}_0+Z_1+P^1_3+W_{-1}$ and discarding lower order terms, we then find
\begin{align*}
&5\int_{\SS^3}(B_1\cdot W)^3 dV_{g_0}-6\int_{\SS^3}(B_1\cdot W)|W|^2 dV_{g_0}\approx \\
&\int_{\SS^3}\left(-6|Z_2|^2+5(B_1\cdot Z_2)^2\right)(B_1\cdot Z_2)dV_{g_0}\\
&+\int_{\SS^3}\left(-12(W_0\cdot Z_2)(B_1\cdot Z_2)-6(B_1\cdot W_0)|Z_2|^2+15(B_1\cdot W_0)(B_1\cdot Z_2)^2\right)dV_{g_0}\\
&+\int_{\SS^3}\left(-12(P^2_3\cdot Z_2)(B_1\cdot Z_2)-6(B_1\cdot P^2_3)|Z_2|^2+15(B_1\cdot P^2_3)(B_1\cdot Z_2)^2\right)dV_{g_0}\\
&+\int_{\SS^3}\left(-6|P^2_3|^2(B_1\cdot Z_2)-12(P^2_3\cdot Z_2)(B_1\cdot P^2_3)+15(B_1\cdot P^2_3)^2(B_1\cdot Z_2)\right)dV_{g_0}\\
&+\int_{\SS^3}(P^2_3\cdot B_1)\left(-6|P^2_3|^2+5(B_1\cdot P^2_3)^2\right)dV_{g_0}+\int_{\SS^3}W_0\cdot R_1dV_{g_0}\,.
\end{align*}

To simplify this expression we first observe that by explicit calculations we have
	\begin{gather}
		\nonumber
		\int_{\SS^3}\left(-6|Z_2|^2+5(B_1\cdot Z_2)^2\right)(B_1\cdot Z_2)\,dV_{g_0}=0\,,\\
		\nonumber
		\int_{\SS^3}\left(-6|P^2_3|^2(B_1\cdot Z_2)-12(P^2_3\cdot Z_2)(B_1\cdot P^2_3)+15(B_1\cdot P^2_3)^2(B_1\cdot Z_2)\right)\,dV_{g_0}=0\,.
	\end{gather}
	
Next we look at the summand
$$\int_{\SS^3}\left(-12(W_0\cdot Z_2)(B_1\cdot Z_2)-6(B_1\cdot W_0)|Z_2|^2+15(B_1\cdot W_0)(B_1\cdot Z_2)^2\right)dV_{g_0}\,,$$
for which we expand $Z_2=\sum_{i=1}^3(B_i\cdot Z_2)B_i$, thus finding that this term is equal to
\begin{gather}
		\nonumber
		=-12\int_{\SS^3}(B_1\cdot Z_2)\left((B_2\cdot W_0)(B_2\cdot Z_2)+(B_3\cdot W_0)(B_3\cdot Z_2)\right)dV_{g_0}
		\\
		\nonumber
		+\int_{\SS^3}(B_1\cdot W_0)\left(3(B_1\cdot Z_2)^2-6|Z_2|^2\right)dV_{g_0}\,.
	\end{gather}
Now we note that $(B_1\cdot Z_2)(B_i\cdot Z_2)$ for $i\neq 1$ is a homogeneous harmonic polynomial of degree $2$ in $\mathbb{R}^4$, i.e., they are eigenfunctions of the Laplacian with eigenvalue $8$, so that
	\begin{gather}
		\nonumber
		-12\int_{\SS^3}(B_1\cdot Z_2)\left((B_2\cdot W_0)(B_2\cdot Z_2)+(B_3\cdot W_0)(B_3\cdot Z_2)\right)dV_{g_0}
		\\
		\nonumber
		=-12\int_{\SS^3}(B_1\cdot Z_2)\left((B_2\cdot (W_{-1}+P^1_3))(B_2\cdot Z_2)+(B_3\cdot (W_{-1}+P^1_3))(B_3\cdot Z_2)\right)dV_{g_0}\,.
	\end{gather}
Further, using that $\|Z_2\|^2_{L^2}=a^2_5+a^2_8=(a^2_5+a^2_8)(x^2_1+x^2_2+x^2_3+x^2_4)$ on $\SS^3$, we have
	\[
	(B_1\cdot Z_2)^2-2|Z_2|^2=-\frac{2\left(a^2_5(x^2_2-x^2_1)+a^2_8(x^2_4-x^2_3)+2a_5a_8(x_1x_3-x_2x_4)+\|Z_2\|^2_{L^2}\right)}{3\pi^2}\,,
	\]
	which, up to an additive constant, is again a homogeneous harmonic polynomial of degree $2$. Therefore
\begin{align*}
&\int_{\SS^3}(B_1\cdot W_0)\left(3(B_1\cdot Z_2)^2-6|Z_2|^2\right)dV_{g_0}=\\
&\int_{\SS^3}(B_1\cdot (W_{-1}+P^1_3))\left(3(B_1\cdot Z_2)^2-6|Z_2|^2\right)dV_{g_0}=\\
&\int_{\SS^3}(B_1\cdot P^1_3)\left(3(B_1\cdot Z_2)^2-6|Z_2|^2\right)dV_{g_0}\,,
\end{align*}
where the last equality follows from an explicit computation by expressing $W_{-1}$ as a linear combination of the anti-Hopf fields. Finally, recall that $W_3=P^1_3+P^2_3$ and $W_E=W_{-1}+W_3$, so that with the above considerations we arrive at the desired result.
\end{proof}

Next we compute (approximately) the fourth order derivative of the functional~$\cF$ (Lemma~\ref{SE8}) and establish an elementary identity (Lemma~\ref{L4}).
\begin{lemma}\label{SE8}
	We have
	\begin{gather}
		\nonumber
		D^4\mathcal{F}(B_1)(W,W,W,W)\approx
		\\
		\nonumber
		 \frac{(\mathcal{E}(B_1))^\frac{1}{3}}{\mathcal{H}(B_1)}\left(\frac{13\|Z_2\|^4_{L^2}}{9\pi^2}+\frac{1}{\pi^2}\left(-2\|Z_2\|^2_{L^2}\|P^2_3\|^2_{L^2}+3\|Z_2\|^2_{L^2}\|B_1\cdot P^2_3\|^2_{L^2}\right)\right.
		\\
		\nonumber
		-6\int_{\SS^3}|P^2_3|^2|Z_2|^2dV_{g_0}-12\int_{\SS^3}(P^2_3\cdot Z_2)^2dV_{g_0}+15\int_{\SS^3}|P^2_3|^2(B_1\cdot Z_2)^2dV_{g_0}
		\\
		\nonumber
		+15\int_{\SS^3}|Z_2|^2(B_1\cdot P^2_3)^2dV_{g_0}-\frac{135}{2}\int_{\SS^3}(B_1\cdot P^2_3)^2(B_1\cdot Z_2)^2dV_{g_0}
		\\
		\nonumber
		\left.+60\int_{\SS^3}(P^2_3\cdot Z_2)(B_1\cdot P^2_3)(B_1\cdot Z_2)dV_{g_0}+\int_{\SS^3}W_0\cdot R_2dV_{g_0}\right)\,,
	\end{gather}
	where $R_2:=-12|Z_2|^2Z_2-45(B_1\cdot Z_2)^3B_1+30(B_1\cdot Z_2)^2Z_2+30|Z_2|^2(B_1\cdot Z_2)B_1$.
\end{lemma}
\begin{proof}
	We recall that $D^3\mathcal{H}= 0$ and that $D\mathcal{E}(B_1)(W)=D\mathcal{H}(B_1)(W)=0$ for every $W\in E^\perp_1$. Using these properties we find
	\begin{gather}
		\label{SE7}
		 D^4\mathcal{F}(B_1)(W,W,W,W)=\frac{(\mathcal{E}(B_1))^\frac{1}{3}}{\mathcal{H}(B_1)}\left(\frac{4}{3}(\mathcal{E}(B_1))^{-1}(D^2\mathcal{E}(B_1)(W,W))^2\right.
		\\
		\nonumber
		-\frac{8}{\mathcal{H}(B_1)}D^2\mathcal{H}(B_1)(W,W)D^2\mathcal{E}(B_1)(W,W)+\frac{4}{3}D^4\mathcal{E}(B_1)(W,W,W,W)
		\\
		\nonumber
		\left.+6\frac{\mathcal{E}(B_1)}{(\mathcal{H}(B_1))^2}(D^2\mathcal{H}(B_1)(W,W))^2\right)\,.
	\end{gather}

Next, it is easy to check that $D^2\mathcal{H}(B_1)(W,W)=2\mathcal{H}(W)$,
	\[
	D^2\mathcal{E}(B_1)(W,W)=\frac{3}{4}\left(2\|W\|^2_{L^2}-\int_{\mathbb{S}^3}(B_1\cdot W)^2 dV_{g_0}\right)\,,
	\]
and
	\[
	D^4\mathcal{E}(B_1)(W,W,W,W)
	\]
	\[
	=\frac{3}{16}\left(-12\|W\|^4_{L^4}+60\int_{\SS^3}|W|^2(B_1\cdot W)^2dV_{g_0}-45\int_{\SS^3}(B_1\cdot W)^4dV_{g_0}\right)\,.
	\]

The lemma then follows from the following identities, which can be verified by explicit calculations and properties of the curl eigenfields,
	\[
	\mathcal{H}(Z_2)=\frac{\|Z_2\|^2_{L^2}}{3}\text{, }\int_{\mathbb{S}^3}(B_1\cdot Z_2)^2 dV_{g_0}=\frac{2}{3}\|Z_2\|^2_{L^2}\text{, }\int_{\mathbb{S}^3}|Z_2|^4 dV_{g_0}=\frac{2}{3\pi^2}\|Z_2\|^4_{L^2}\,,
	\]
	\[
	\int_{\mathbb{S}^3}|Z_2|^2(B_1\cdot Z_2)^2 dV_{g_0}=\frac{14}{27\pi^2}\|Z_2\|^4_{L^2}\text{, }\int_{\mathbb{S}^3}(B_1\cdot Z_2)^4 dV_{g_0}=\frac{4}{9\pi^2}\|Z_2\|^4_{L^2}.\,,
	\]
	\[
	\left(B_1\cdot Z_2,B_1\cdot W_E\right)_{L^2}=0\text{, }\left(B_1\cdot W_0,B_1\cdot Z_2\right)_{L^2}=0\,,
	\]
	\[
	\int_{\SS^3}(P^2_3\cdot Z_2)(B_1\cdot Z_2)^2 dV_{g_0}=\int_{\SS^3}(P^2_3\cdot B_1)(Z_2\cdot B_1)|Z_2|^2dV_{g_0}=0\,,
	\]
	\[
	\int_{\SS^3}(P^2_3\cdot B_1)(B_1\cdot Z_2)^3dV_{g_0}=\int_{\SS^3}(P^2_3\cdot Z_2)|Z_2|^2dV_{g_0}=0\,.
	\]
\end{proof}

\begin{lemma}
	\label{L4}
	We have
	\[
	2\|W_E\|^2_{L^2}-4\mathcal{H}(W_E)-\int_{\mathbb{S}^3}(B_1\cdot W_E)^2dV_{g_0}= 4\|W_{-1}\|^2_{L^2}+\|W_3\|^2_{L^2}-\int_{\mathbb{S}^3}(B_1\cdot W_E)^2dV_{g_0}
	\]
	\[
	=\frac{11}{3}\|W_{-1}\|^2_{L^2}+\|W_3\|^2_{L^2}-\|B_1\cdot W_3\|^2_{L^2}-2\int_{\SS^3}(B_1\cdot W_{-1})(B_1\cdot W_3)dV_{g_0}\,.
	\]
\end{lemma}
\begin{proof}
Since $W_E=W_{-1}+W_3$, it follows that
	\[
	\|W_E\|^2_{L^2}=\|W_{-1}\|^2_{L^2}+\|W_3\|^2_{L^2}\text{, }\mathcal{H}(W_E)=\mathcal{H}(W_{-1})+\mathcal{H}(W_3)\,.
	\]
Further,
	\[
	\int_{\mathbb{S}^3}(B_1\cdot W_E)^2dV_{g_0}=\int_{\mathbb{S}^3}\left((B_1\cdot W_{-1})^2+(B_1\cdot W_3)^2+2(B_1\cdot W_{-1})(B_1\cdot W_3)\right)dV_{g_0}\,.
	\]
Since an explicit basis of $E_{-1}$ is given by the anti-Hopf fields~\eqref{ConfE56}, a direct computation then yields
$$\|B_1\cdot W_{-1}\|^2_{L^2}=\frac{\|W_{-1}\|^2_{L^2}}{3}\,,$$
which concludes the proof.
\end{proof}

Putting together the previous lemmas and the explicit computation of the first and second derivative of $\cF$ in the proof of Theorem~\ref{S3C2}, we easily obtain the following corollary, which is used in the proof of Proposition~\ref{PE1} below.

\begin{corollary}
	\label{CE1}
	The following estimate holds
	\begin{gather}
		\nonumber
		6D\mathcal{F}(B_1)(W)+3D^2\mathcal{F}(B_1)(W,W)+D^3\mathcal{F}(B_1)(W,W,W)+\frac{D^4\mathcal{F}(B_1)(W,W,W,W)}{4} \\
		\nonumber
		\gtrsim \frac{(\mathcal{E}(B_1))^\frac{1}{3}}{\mathcal{H}(B_1)}\left(\left[\frac{3}{5}\|\hat{W}_0\|^2_{L^2}+2\|Z_1\|^2_{L^2}+11\|W_{-1}\|^2_{L^2}+
3\|W_3\|^2_{L^2}-3\|B_1\cdot W_3\|^2_{L^2}\right.\right.
		\\
		\nonumber
		-6\int_{\SS^3}(B_1\cdot W_{-1})(B_1\cdot W_3)dV_{g_0}+\frac{15}{2}\int_{\SS^3}(B_1\cdot W_3)(B_1\cdot Z_2)^2dV_{g_0}
		\\
		\nonumber	
		-6\int_{\SS^3}(W_3\cdot Z_2)(B_1\cdot Z_2)dV_{g_0}-3\int_{\SS^3}(B_1\cdot W_3)|Z_2|^2dV_{g_0}+\frac{13}{36\pi^2}\|Z_2\|^4_{L^2}
		\\
		\nonumber
		\left.-6\int_{\SS^3}(B_1\cdot Z_2)\left((B_2\cdot W_{-1})(B_2\cdot Z_2)+(B_3\cdot W_{-1})(B_3\cdot Z_2)\right)dV_{g_0}\right]
		\\
		\nonumber
		+\left[\frac{1}{2}\int_{\SS^3}(P^2_3\cdot B_1)(-6|P^2_3|^2+5(B_1\cdot P^2_3)^2)dV_{g_0}+\int_{\SS^3}W_0\cdot RdV_{g_0}\right.
		\\
		\nonumber
		+\frac{1}{4\pi^2}\left(-2\|Z_2\|^2_{L^2}\|P^2_3\|^2_{L^2}+3\|Z_2\|^2_{L^2}\|B_1\cdot P^2_3\|^2_{L^2}\right)-\frac{3}{2}\int_{\SS^3}|P^2_3|^2|Z_2|^2dV_{g_0}
		\\
		\nonumber
		-3\int_{\SS^3}(Z_2\cdot P^2_3)^2dV_{g_0}+\frac{15}{4}\int_{\SS^3}|P^2_3|^2(B_1\cdot Z_2)^2dV_{g_0}+\frac{15}{4}\int_{\SS^3}|Z_2|^2(B_1\cdot P^2_3)^2dV_{g_0}
		\\
		\label{SE9}
		\left.\left.-\frac{135}{8}\int_{\SS^3}(B_1\cdot P^2_3)^2(B_1\cdot Z_2)^2dV_{g_0}+15\int_{\SS^3}(P^2_3\cdot Z_2)(B_1\cdot P^2_3)(B_1\cdot Z_2)dV_{g_0}\right]\right)\,,
	\end{gather}
	where
	\[
	R:=\left(-6P^2_3\cdot Z_2+15(P^2_3\cdot B_1)(B_1\cdot Z_2)-\frac{45}{4}(B_1\cdot Z_2)^3+\frac{15}{2}|Z_2|^2(B_1\cdot Z_2)\right)B_1
	\]
	\[
	+\left(-6(B_1\cdot P^2_3)-3|Z_2|^2+\frac{15}{2}(B_1\cdot Z_2)^2\right)Z_2-6(B_1\cdot Z_2)P^2_3\,.
	\]
\end{corollary}

Now we compute a basis of the eigenspace corresponding to the third positive curl eigenvalue $\mu_3=4$.

\begin{lemma}
	\label{LE2}
	The eigenspace of curl corresponding to the eigenvalue $4$ is $15$-dimensional and is spanned by
	\begin{gather}
		\label{SE12}
		v_1=\sqrt{\frac{3}{2}}\frac{(x^2_1-x^2_2)B_2-2x_1x_2B_3}{\pi}\text{, }v_2=\sqrt{\frac{3}{2}}\frac{(x^2_3-x^2_4)B_2-2x_3x_4B_3}{\pi}\,,
		\\
		\nonumber
		v_3=\sqrt{\frac{3}{2}}\frac{2x_1x_2B_2+(x^2_1-x^2_2)B_3}{\pi}\text{, }v_4=\sqrt{\frac{3}{2}}\frac{2x_3x_4B_2+(x^2_3-x^2_4)B_3}{\pi}\,,
		\\
		\nonumber
		v_5=\sqrt{3}\frac{(x_2x_4-x_1x_3)B_2+(x_1x_4+x_2x_3)B_3}{\pi}\text{,}
		\\
		\nonumber
		v_6=\sqrt{3}\frac{(x_1x_4+x_2x_3)B_2+(x_1x_3-x_2x_4)B_3}{\pi}\,,
		\\
		\nonumber
		v_7=\sqrt{3}\frac{(x_1x_2+x_3x_4)B_1+(x_2x_3-x_1x_4)B_3}{\pi}\,,
	\\
		\nonumber
		v_8=\sqrt{\frac{3}{7}}\frac{8x_1x_3B_1+2(x_1x_2-x_3x_4)B_2+(3(x^2_3-x^2_1)+(x^2_4-x^2_2))B_3}{2\pi}\,,
	\\
		\nonumber
		v_9=\frac{4(x_1x_4-x_2x_3)B_1+(x^2_1-x^2_2+x^2_3-x^2_4)B_2+2(x_1x_2+x_3x_4)B_3}{2\pi}\,,
				\end{gather}
	\begin{gather}
		\nonumber
		v_{10}=\frac{(14x_2x_4+2x_1x_3)B_1+4(x_1x_2-x_3x_4)B_2+(x^2_1-x^2_3+5(x^2_2-x^2_4))B_3}{2\pi\sqrt{7}}\,,
		\\
		\nonumber
		v_{11}=\sqrt{\frac{3}{7}}\frac{2(x^2_1-x^2_3)B_1-(x_1x_4+x_2x_3)B_2+(3x_1x_3+x_2x_4)B_3}{\pi}\,,
		\\
		\nonumber
		v_{12}=\frac{(7(x^2_2-x^2_4)+x^2_1-x^2_3)B_1-4(x_1x_4+x_2x_3)B_2-(2x_1x_3+10x_2x_4)B_3}{2\pi\sqrt{7}}\,,
	\\
		\nonumber
		v_{13}=\frac{\sqrt{3}}{\pi}\left((x_3x_4-x_1x_2)B_1+(x_1x_3+x_2x_4)B_2\right)\,,
		\\
		\nonumber
		v_{14}=\frac{\sqrt{3}}{2\pi}\left(2(x_1x_4+x_2x_3)B_1+(x^2_1-x^2_4+x^2_2-x^2_3)B_2\right)\,,
		\\
		\nonumber
		v_{15}=\frac{\sqrt{3}}{2\pi}\left((x^2_2-x^2_3+x^2_4-x^2_1)B_1+2(x_1x_4-x_2x_3)B_2\right)\,.
	\end{gather}
\end{lemma}
\begin{proof}
The eigenspace $E_3$ is $15$-dimensional~\cite[Theorem 5.2]{Bar}. Expressing any smooth divergence-free vector field $Z$ on $\SS^3$ as $Z=\sum_{i=1}^3f_iB_i$ for some functions $f_i\in C^\infty(\SS^3)$, we find
	\[
	\operatorname{curl}(Z)=2Z+\sum_{i=1}^3 \nabla (f_i)\times B_i=2Z+\sum_{i,j=1}^3B_j(f_i)B_j\times B_i\,,
	\]
and therefore
	\begin{gather}
		\label{SE11}
		\operatorname{curl}(Z)-2Z
		\\
		\nonumber
		=(B_2(f_3)-B_3(f_2))B_1+(B_3(f_1)-B_1(f_3))B_2+(B_1(f_2)-B_2(f_1))B_3\,.
	\end{gather}
If $Z\in E_3$ one can easily check, using~\eqref{SE11}, that the stated $15$ vector fields form an $L^2$-orthonormal basis of $E_3$.
\end{proof}

The following key proposition establishes a sufficient condition for $L^{\frac32}$-minimality which is used in the proof of Theorem~\ref{S3C2}.

\begin{proposition}
	\label{PE1}
	For all $\epsilon_0>0$ small enough it holds that $\mathcal{F}(Y)\geq \mathcal{F}(X)$, with equality if and only if $Y\in E_1$, unless
	\begin{gather}
		\label{E1}
		0<a^2_5+a^2_8\text{, }b_{15}=\frac{a^2_5+a^2_8}{2\sqrt{3}\pi}(1+\delta_{15})\,,
	\end{gather}
for some $\delta_{15}$ with $|\delta_{15}|\leq \epsilon_0$, and at least one more (possibly both) of the following two identities are satisfied
\begin{gather}
		\label{SE22}
		a^2_5-a^2_8=\frac{6\pi}{\sqrt{7}}b_{12}(1+\delta_{12})\,,
\\
		\label{SE24}
	a_5a_8=-\frac{3\pi}{\sqrt{7}}b_{10}(1+\delta_{10})\,,
\end{gather}
for some $|\delta_{12}|\leq \epsilon_0$ and $|\delta_{10}|\leq \epsilon_0$, where $X\in E_1$ and $Y\in B_r(X)$ for a suitable $r=r(\epsilon_0)>0$.
\end{proposition}

\begin{proof}
We start by observing that the terms in the second square bracket in the Taylor expansion in Corollary~\ref{CE1} are of order $o(1)\left(\|P^2_3\|^2_{L^2}+\|Z_2\|^4_{L^2}\right)$ as $r\searrow 0$. We claim that under the hypothesis of the proposition the first square bracket contains terms of the form $\|P^2_3\|^2_{L^2}+\|Z_2\|^4_{L^2}$, so that we may neglect the second bracket.
We notice that $\hat{W}_0$ and $Z_1$ (within the first square bracket) only appear in the first two terms and hence we control their norms quadratically, more precisely we have the following terms in the lower bound involving $\hat{W}_0$ and $Z_1$
\[
\frac{(\mathcal{E}(B_1))^{\frac{1}{3}}}{\mathcal{H}(B_1)}\left(\frac{3}{5}\|\hat{W}_0\|^2_{L^2}+2\|Z_1\|^2_{L^2}\right)
\]
and no other term contains $\hat{W}_0$ or $Z_1$. To prove the proposition, let us show that under our hypothesis we have in fact a quadratic control of the norms of all terms except for $Z_2$, which we control by $\|Z_2\|^4_{L^2}$, i.e., we claim that under our hypothesis there exists a constant $\sigma(\epsilon_0)>0$ such that
\begin{gather}
	\nonumber
6D\mathcal{F}(B_1)(W)+3D^2\mathcal{F}(B_1)(W,W)+D^3\mathcal{F}(B_1)(W,W,W)+\frac{D^4\mathcal{F}(B_1)(W,W,W,W)}{4}
\\
\nonumber
\gtrsim \sigma(\epsilon_0)\left(\|W-Z_2\|^2_{L^2}+\|Z_2\|^4_{L^2}\right)\,.
\end{gather}
To this end we perform some explicit computations. We can express $W_{-1}=\sum_{j=1}^3\beta_j\frac{\hat{B}_j}{\sqrt{2}\pi}$ using the anti-Hopf fields $\hat{B}_j$, and we have $Z_2=a_5u_5+a_8u_8$ by means of Lemma~\ref{LE1}, so that
	\begin{gather}
		\label{SE10}
		\int_{\SS^3}(B_1\cdot Z_2)\left((B_2\cdot Z_2)(B_2\cdot W_{-1})+(B_3\cdot Z_2)(B_3\cdot W_{-1})\right)dV_{g_0}
		\\
		\nonumber
		=\frac{2\sqrt{2}}{3\pi}\left[\frac{\beta_1}{3}a_5a_8+\frac{\beta_3}{6}(a^2_8-a^2_5)\right]\,.
	\end{gather}
We can now express $W_3=\sum_{i=1}^{15}b_iv_i$, cf. Lemma~\ref{LE2}, for suitable constants $b_i\in \mathbb{R}$. An explicit calculation then yields $\|W_3\|^2_{L^2}=\sum_{i=1}^{15}b^2_i$ and
	\begin{gather}
		\label{SE13}
		\|B_1\cdot W_3\|^2_{L^2}=\frac{b^2_7}{2}+\sum_{i=13}^{15}\frac{b^2_i}{2}+\frac{2}{3}b^2_9+\left(\frac{4}{7}b^2_8+\frac{2}{7\sqrt{3}}b_8b_{10}+\frac{25}{42}b^2_{10}\right) \\
		\nonumber
		+\left(\frac{4}{7}b^2_{11}+\frac{2}{7\sqrt{3}}b_{11}b_{12}+\frac{25}{42}b^2_{12}\right)\,.
	\end{gather}
	Further, we find
	\begin{gather}
		\label{SE14}
		\int_{\SS^3}(B_1\cdot W_{-1})(B_1\cdot W_3)\,dV_{g_0} \\
		\nonumber
		 =\sqrt{\frac{2}{21}}b_8\beta_1+\frac{4\beta_1b_{10}}{3\sqrt{14}}-\frac{\sqrt{2}}{3}\beta_2b_9+\sqrt{\frac{2}{21}}b_{11}\beta_3+\frac{4\beta_3b_{12}}{3\sqrt{14}}\,.
	\end{gather}
Now,
	\begin{gather}
		\label{SE15}
		\frac{1}{2}\int_{\SS^3}(B_1\cdot W_3)\left[(B_1\cdot Z_2)^2-2|Z_2|^2\right]\,dV_{g_0}
		\\
		\nonumber
		=-\frac{1}{3\pi}\left[\frac{2b_8a_5a_8}{\sqrt{21}}	 -\frac{b_{10}}{\sqrt{7}}a_5a_8+\frac{b_{11}}{\sqrt{21}}(-a^2_5+a^2_8)+\frac{b_{12}}{2\sqrt{7}}(-a^2_8+a^2_5)+\frac{b_{15}}{2\sqrt{3}}(a^2_5+a^2_8)\right]\,,
	\end{gather}
	\begin{gather}
		\label{SE16}
		\int_{\SS^3}(B_1\cdot Z_2)(B_2\cdot Z_2)(B_2\cdot W_3)\,dV_{g_0}\\
		\nonumber
		 =\frac{2}{3\pi}\left[-\frac{b_3a_5a_8}{2\sqrt{6}}+\frac{b_4a_5a_8}{2\sqrt{6}}+\frac{b_6(a^2_8-a^2_5)}{4\sqrt{3}}-\frac{b_8a_5a_8}{2\sqrt{21}}-\frac{b_{10}a_5a_8}{3\sqrt{7}}\right.
		\\
		\nonumber
		\left.+\frac{b_{11}(a^2_5-a^2_8)}{4\sqrt{21}}+\frac{b_{12}(a^2_5-a^2_8)}{6\sqrt{7}}+\frac{b_{15}(a^2_8+a^2_5)}{4\sqrt{3}}\right]\,,
	\end{gather}
	\begin{gather}
		\label{SE17}
		\int_{\SS^3}(B_1\cdot Z_2)(B_3\cdot Z_2)(B_3\cdot W_3)\,dV_{g_0}\\
		\nonumber
		=\frac{2}{3\pi}\left[\frac{(b_3-b_4)a_5a_8}{2\sqrt{6}}+\frac{b_6}{4\sqrt{3}}(-a^2_8+a^2_5)+\frac{b_8}{2\sqrt{21}}a_5a_8\right.	\\
		\nonumber
		\left.-\frac{5b_{10}a_5a_8}{6\sqrt{7}}+\frac{b_{11}}{4\sqrt{21}}(a^2_8-a^2_5)+\frac{5b_{12}(a^2_5-a^2_8)}{12\sqrt{7}}\right]\,.
	\end{gather}
	With this at hand we can express the first square bracket terms, excluding $\hat{W}_0$ and $Z_1$, as (we divide the whole expression by a factor $3$ to simplify some of the terms)
	\[
	\frac{11}{3}\|W_{-1}\|^2_{L^2}+\|W_3\|^2_{L^2}-\|B_1\cdot W_3\|^2_{L^2}
	\]
	\[
	-2\int_{\SS^3}(B_1\cdot Z_2)\left((B_2\cdot W_{-1})(B_2\cdot Z_2)+(B_3\cdot W_{-1})(B_3\cdot Z_2)\right)dV_{g_0}
	\]
	\[
	-\int_{\SS^3}(B_1\cdot W_3)|Z_2|^2dV_{g_0}-2\int_{\SS^3}(W_3\cdot Z_2)(B_1\cdot Z_2)dV_{g_0}+\frac{5}{2}\int_{\SS^3}(B_1\cdot W_3)(B_1\cdot Z_2)^2dV_{g_0}
	\]
	\[
	-2\int_{\SS^3}(B_1\cdot W_{-1})(B_1\cdot W_3)dV_{g_0}+\frac{13}{108\pi^2}\|Z_2\|^4_{L^2}
	\]
	\[
	 =\left\{\sum_{i=1}^6b^2_i+\frac{b^2_7}{2}+\frac{b^2_{13}+b^2_{14}}{2}\right\}+\left\{\frac{11}{3}\beta^2_2+\frac{b^2_9}{3}+\frac{2\sqrt{2}}{3}\beta_2b_9\right\}
	\]
	\[
	 +\left\{\frac{11}{3}\beta^2_1+\frac{3}{7}b^2_8-\frac{2\sqrt{3}}{21}b_8b_{10}+\frac{17}{42}b^2_{10}-\frac{2\sqrt{2}}{\sqrt{21}}\beta_1b_8-\frac{4\sqrt{14}}{21}\beta_1b_{10}\right.
	\]
	\[
	\left.-\frac{4\sqrt{2}}{9\pi}\beta_1a_5a_8+\frac{a_5a_8(-4\sqrt{21}b_8+34\sqrt{7}b_{10})}{126\pi}\right\}
	\]
	\[
	 +\left\{\frac{11}{3}\beta^2_3+\frac{3}{7}b^2_{11}-\frac{2\sqrt{3}}{21}b_{11}b_{12}+\frac{17}{42}b^2_{12}-\frac{2\sqrt{2}}{\sqrt{21}}\beta_3b_{11}-\frac{4\sqrt{14}}{21}\beta_3b_{12}+\frac{b^2_{15}}{2}\right.
	\]
	\[
	+\frac{13}{108\pi^2}(a^2_5+a^2_8)^2+\frac{2\sqrt{2}}{9\pi}\beta_3(a^2_5-a^2_8)
	\]
	\[
	\left.+\frac{(a^2_5-a^2_8)(2\sqrt{21}b_{11}-17\sqrt{7}b_{12})-21\sqrt{3}b_{15}(a^2_5+a^2_8)}{126\pi}\right\}\,.
	\]

We note that the terms in the first curly bracket all appear quadratically and neither of the quantities $b_1,\dots,b_7,b_{13},b_{14}$ appears in the remaining terms. As for the second curly bracket, we see that they are the only ones containing $\beta_2$ and $b_9$, and that an application of the elementary inequality $2ab\leq a^2+ b^2$ yields
	\[
	\frac{11}{3}\beta^2_2+\frac{b^2_9}{3}+2\frac{b_9}{3}(\sqrt{2}\beta_2)\geq b^2_9\left(\frac{1}{3}-\frac{1}{9}\right)+\beta^2_2\left(\frac{11}{3}-2\right)=\frac{2b^2_9}{9}+\frac{5\beta^2_2}{3}\,,
	\]
	and therefore we control $\beta_2$, as well as $b_9$ both quadratically.

	The fourth curly bracket can be equivalently expressed as follows
	\[
	 \frac{11}{3}\beta^2_3+\frac{3}{7}b^2_{11}-\frac{2\sqrt{3}}{21}b_{11}b_{12}+\frac{17}{42}b^2_{12}-\frac{2\sqrt{2}}{\sqrt{21}}\beta_3b_{11}-\frac{4\sqrt{14}}{21}b_{12}\beta_3
	\]
	\[
	 +\frac{(a^2_5-a^2_8)}{3\pi}\left(\frac{2\sqrt{2}}{3}\beta_3+\frac{b_{11}}{\sqrt{21}}-\frac{17}{6\sqrt{7}}b_{12}\right)+\left(\frac{b_{15}}{\sqrt{2}}-\frac{(a^2_5+a^2_8)}{2\sqrt{6}\pi}\right)^2+\frac{17}{216\pi^2}(a^2_5+a^2_8)^2
	\]
	\[
	 =\frac{57\beta^2_3+7b^2_{11}-2\sqrt{42}\beta_3b_{11}}{17}+\frac{17}{54\pi^2}a^2_5a^2_8+\left(\frac{b_{15}}{\sqrt{2}}-\frac{(a^2_5+a^2_8)}{2\sqrt{6}\pi}\right)^2
	\]
	\[
	 +\left(\frac{a^2_5-a^2_8}{\pi}\sqrt{\frac{17}{216}}+\frac{\frac{2\sqrt{2}}{3}\beta_3+\frac{b_{11}}{\sqrt{21}}-\frac{17}{6\sqrt{7}}b_{12}}{6}
\sqrt{\frac{216}{17}}\right)^2\,.
	\]
On the other hand, the third curly bracket can be expressed as
	\[
	 \frac{11}{3}\beta^2_1+\frac{3}{7}b^2_8-\frac{2\sqrt{3}}{21}b_8b_{10}+\frac{17}{42}b^2_{10}-\frac{2\sqrt{2}}{\sqrt{21}}\beta_1b_8-\frac{4\sqrt{14}}{21}\beta_1b_{10}
	\]
	\[
	-\frac{17}{54\pi^2}a^2_5a^2_8-\frac{54}{17}\left(\frac{2\sqrt{2}}{9}\beta_1+\frac{2\sqrt{21}b_8-17\sqrt{7}b_{10}}{126}\right)^2
	\]
	\[
	+\left(\frac{a_5a_8}{\pi}\sqrt{\frac{17}{54}}-\frac{2\sqrt{2}}{9}\beta_1+\frac{2\sqrt{21}b_8-17\sqrt{7}b_{10}}{126}\sqrt{\frac{54}{17}}\right)^2
	\]
	\[
	 =\frac{57\beta^2_1+7b^2_8-2\sqrt{42}\beta_1b_8}{17}+\left(\frac{a_5a_8}{\pi}\sqrt{\frac{17}{54}}-\frac{2\sqrt{2}}{9}\beta_1+\frac{2\sqrt{21}b_8-17\sqrt{7}b_{10}}{126}\sqrt{\frac{54}{17}}\right)^2
	\]
	\[
	-\frac{17}{54\pi^2}a^2_5a^2_8\,.
	\]

Now we note that
	\[
	\frac{57\beta^2_1+7b^2_8-2(\sqrt{42}\beta_1)b_8}{17}\geq\frac{15\beta^2_1+6b^2_8}{17}\,.
	\]
Using the same inequality for the term containing $\beta_3$, $b_{11}$ and adding the terms of the third and fourth bracket, we see that overall the third and fourth curly brackets can be estimated from below by
	\begin{gather}
		\label{SE18}
		 \frac{15(\beta^2_1+\beta^2_3)+6(b^2_8+b^2_{11})}{17}+\left(\frac{a_5a_8}{\pi}\sqrt{\frac{17}{54}}-\frac{2\sqrt{2}}{9}\beta_1+\frac{2\sqrt{21}b_8-17\sqrt{7}b_{10}}{126}\sqrt{\frac{54}{17}}\right)^2
		\\
		\nonumber
		 +\left(\frac{b_{15}}{\sqrt{2}}-\frac{(a^2_5+a^2_8)}{2\sqrt{6}\pi}\right)^2+\left(\frac{a^2_5-a^2_8}{\pi}\sqrt{\frac{17}{216}}+\frac{\frac{2\sqrt{2}}{3}\beta_3+
\frac{b_{11}}{\sqrt{21}}-\frac{17}{6\sqrt{7}}b_{12}}{6}\sqrt{\frac{216}{17}}\right)^2\,.
	\end{gather}

We observe that we therefore control the following terms quadratically (i.e., modulo a positive constant these terms all appear squared in our lower bound): $b_1,\dots, b_7,b_9,b_{13},b_{14}$ and $\beta_2$, so the only terms that we have to control more carefully are those that involve $a_5,a_8,b_8,b_{10},b_{11},b_{12},b_{15},\beta_1$ and $\beta_3$.

First, we will have a closer look at the term containing $b_{15}$, which is $\left(\text{after factoring }\frac{1}{\sqrt{2}}\right)$
	\[
	\left(b_{15}-\frac{a^2_5+a^2_8}{2\sqrt{3}\pi}\right)^2\,.
	\]
	Our goal is  to determine circumstances in which we control $b^2_{15}$ as well as $(a^2_5+a^2_8)^2$. We observe that if $b_{15}\leq \frac{a^2_5+a^2_8}{100}$, then
	\[
	\left(b_{15}-\frac{a^2_5+a^2_8}{2\sqrt{3}\pi}\right)^2=b^2_{15}-\frac{b_{15}(a^2_5+a^2_8)}{\sqrt{3}\pi}+\frac{(a^2_5+a^2_8)^2}{12\pi^2}
	\]
	\[
	\geq b^2_{15}+(a^2_5+a^2_8)^2\left(\frac{1}{12\pi^2}-\frac{1}{100\sqrt{3}\pi}\right)\,,
	\]
	and so we control $b^2_{15}$ as well as $(a^2_5+a^2_8)^2=\|Z_2\|^4_{L^2}$. In particular, the only situations in which we possibly do not control $b^2_{15}$ is that in which $b_{15}\geq \frac{a^2_5+a^2_8}{100}$ and whence in particular $b_{15}$ is non-negative. Similarly, if $a^2_5+a^2_8\leq \frac{b_{15}}{100}$, then
	\[
	\left(b_{15}-\frac{a^2_5+a^2_8}{2\sqrt{3}\pi}\right)^2\geq b^2_{15}\left(1-\frac{1}{100\sqrt{3}\pi}\right)+\frac{(a^2_5+a^2_8)^2}{12\pi^2}\,,
	\]
	where we assumed that $b_{15}\geq 0$ (since for negative $b_{15}$ we already know that we have quadratic control of $b_{15}$). Overall, we control $b^2_{15}$ and $\|Z_2\|^4_{L^2}$, unless
	\begin{gather}
		\label{SE19}
		\frac{a^2_5+a^2_8}{100}\leq b_{15}\leq 100(a^2_5+a^2_8)\,.
	\end{gather}
	If we now fix any absolute constant $\epsilon_0>0$ (i.e., a constant independent of the coefficients $\beta_j,b_i$, $W_0$ etc.), and consider the situation
	\[
	\left|\frac{b_{15}}{a^2_5+a^2_8}-\frac{1}{\sqrt{12}\pi}\right|\geq \epsilon_0\,,
	\]
we obtain the estimate
	\[
	\left(b_{15}-\frac{a^2_5+a^2_8}{2\sqrt{3}\pi}\right)^2=(a^2_5+a^2_8)^2\left(\frac{b_{15}}{a^2_5+a^2_8}-\frac{1}{\sqrt{12}\pi}\right)^2\geq (a^2_5+a^2_8)^2\epsilon^2_0\,.
	\]
We recall that if~\eqref{SE19} is violated, we already control $b_{15}$ and $a^2_5+a^2_8$ quadratically, i.e., these terms appear squared up to a positive constant in our lower bound. On the other hand, if~\eqref{SE19} holds, then
	\[
	\epsilon^2_0(a^2_5+a^2_8)^2=\frac{\epsilon^2_0}{2}\left((a^2_5+a^2_8)^2+(a^2_5+a^2_8)^2\right)\geq \frac{\epsilon^2_0}{2}\left((a^2_5+a^2_8)^2+\frac{b^2_{15}}{100^2}\right)\,,
	\]
	and so once more we control $b_{15}$ and $(a^2_5+a^2_8)$ quadratically, modulo a positive constant that now depends on $\epsilon_0$.

	Overall, the only situation in which we do not have a quadratic control of $b_{15}$ and $(a^2_5+a^2_8)$ is when
\begin{gather}
		\label{SE20}
		b_{15}=\frac{a^2_5+a^2_8}{\sqrt{12}\pi}(1+\delta_{15})
\end{gather}
for some $|\delta_{15}|\leq \epsilon_0$, which is the relation~\eqref{E1} in the statement of the proposition (notice that~\eqref{SE20} implies the estimate~\eqref{SE19} for small enough $\epsilon_0$).

Note that the quadratic control which we get otherwise can in general be as bad as $\epsilon^2_0$. So, keeping in mind that we intend to use a Taylor expansion, we must pick the radius $r$ depending on $\epsilon_0$ in order to guarantee that the error term in this expansion becomes negligible compared to $\epsilon^2_0$. Hence, in the limit $\epsilon_0\searrow 0$ it might be that as $r\searrow 0$ the neighbourhood $B_r(X)$ becomes a single point, so that we really need to take a fixed $\epsilon_0$ and analyze the situation in which $b_{15}$ is $\epsilon_0$-close to $\frac{a^2_5+a^2_8}{\sqrt{12}\pi}$ separately.

We now turn to the terms in~\eqref{SE18} containing $b_{12}$ and $a^2_5-a^2_8$. The relevant parts are
	\[
	 \frac{15}{17}\beta^2_3+\frac{6}{17}b^2_{11}+\left(\sqrt{\frac{216}{17}}\left(\frac{\sqrt{2}}{9}\beta_3+\frac{b_{11}}{6\sqrt{21}}\right)+
\sqrt{\frac{17}{216}}\frac{a^2_5-a^2_8}{\pi}-\frac{17}{36\sqrt{7}}b_{12}\sqrt{\frac{216}{17}}\right)^2\,.
	\]
	We first note that if either $10^{12}b^2_{11}\geq \left(\sqrt{\frac{17}{216}}\frac{a^2_5-a^2_8}{\pi}-\frac{17}{36\sqrt{7}}b_{12}\sqrt{\frac{216}{17}}\right)^2$ or if instead $10^{12}\beta^2_{3}\geq \left(\sqrt{\frac{17}{216}}\frac{a^2_5-a^2_8}{\pi}-\frac{17}{36\sqrt{7}}b_{12}\sqrt{\frac{216}{17}}\right)^2$, we can obtain an estimate of the form
\begin{align*}
&\frac{15}{17}\beta^2_3+\frac{6}{17}b^2_{11}+\left(\sqrt{\frac{216}{17}}\left(\frac{\sqrt{2}}{9}\beta_3+\frac{b_{11}}{6\sqrt{21}}\right)+
\sqrt{\frac{17}{216}}\frac{a^2_5-a^2_8}{\pi}-\frac{17}{36\sqrt{7}}b_{12}\sqrt{\frac{216}{17}}\right)^2\\
&\geq \frac{15}{17}\beta^2_3+\frac{6}{17}b^2_{11}\\
&\geq \sigma\left(\beta^2_3+b^2_{11}+\left(\sqrt{\frac{17}{216}}\frac{a^2_5-a^2_8}{\pi}-\frac{17}{36\sqrt{7}}b_{12}\sqrt{\frac{216}{17}}\right)^2\right)
\end{align*}	
for some absolute constant $\sigma>0$. On the other hand, if
	\[
	b^2_{11}\leq \frac{\left(\sqrt{\frac{17}{216}}\frac{a^2_5-a^2_8}{\pi}-\frac{17}{36\sqrt{7}}b_{12}\sqrt{\frac{216}{17}}\right)^2}{10^{12}}
	\]
and the same estimate holds for $\beta^2_3$, then we can expand the third quadratic term and find
	\[
	 \frac{15}{17}\beta^2_3+\frac{6}{17}b^2_{11}+\left(\sqrt{\frac{216}{17}}\left(\frac{\sqrt{2}}{9}\beta_3+\frac{b_{11}}{6\sqrt{21}}\right)+\sqrt{\frac{17}{216}}\frac{a^2_5-a^2_8}{\pi}-\frac{17}{36\sqrt{7}}b_{12}\sqrt{\frac{216}{17}}\right)^2
	\]
	\[
	 =\frac{15}{17}\beta^2_3+\frac{6}{17}b^2_{11}+\frac{216}{17}\left(\frac{\sqrt{2}}{9}\beta_3+\frac{b_{11}}{6\sqrt{21}}\right)^2+\left(\sqrt{\frac{17}{216}}\frac{a^2_5-a^2_8}{\pi}-\frac{17}{36\sqrt{7}}b_{12}\sqrt{\frac{216}{17}}\right)^2
	\]
	\[
	 +2\sqrt{\frac{216}{17}}\left(\frac{\sqrt{2}}{9}\beta_3+\frac{b_{11}}{6\sqrt{21}}\right)\left(\sqrt{\frac{17}{216}}\frac{a^2_5-a^2_8}{\pi}-
\frac{17}{36\sqrt{7}}b_{12}\sqrt{\frac{216}{17}}\right)\,,
	\]
and the last summand can be estimated by
	\[
	 2\sqrt{\frac{216}{17}}\left(\frac{\sqrt{2}}{9}\beta_3+\frac{b_{11}}{6\sqrt{21}}\right)\left(\sqrt{\frac{17}{216}}\frac{a^2_5-a^2_8}{\pi}-\frac{17}{36\sqrt{7}}b_{12}\sqrt{\frac{216}{17}}\right)
	\]
	\[
	\geq -2\sqrt{\frac{216}{17}}\frac{\frac{\sqrt{2}}{9}+\frac{1}{6\sqrt{21}}}{10^6}\left(\sqrt{\frac{17}{216}}\frac{a^2_5-a^2_8}{\pi}-
\frac{17}{36\sqrt{7}}b_{12}\sqrt{\frac{216}{17}}\right)^2\,,
	\]
which can be absorbed by the term $\left(\sqrt{\frac{17}{216}}\frac{a^2_5-a^2_8}{\pi}-\frac{17}{36\sqrt{7}}b_{12}\sqrt{\frac{216}{17}}\right)^2$, so that in either case we arrive at the inequality
	
	\[
	 \frac{15}{17}\beta^2_3+\frac{6}{17}b^2_{11}+\left(\sqrt{\frac{216}{17}}\left(\frac{\sqrt{2}}{9}\beta_3+\frac{b_{11}}{6\sqrt{21}}\right)+\sqrt{\frac{17}{216}}\frac{a^2_5-a^2_8}{\pi}-\frac{17}{36\sqrt{7}}b_{12}\sqrt{\frac{216}{17}}\right)^2
	\]
	\[
	\geq \sigma\left(\beta^2_3+b^2_{11}+\left(\sqrt{\frac{17}{216}}\frac{a^2_5-a^2_8}{\pi}-\frac{17}{36\sqrt{7}}b_{12}\sqrt{\frac{216}{17}}\right)^2\right)
	\]
	for some absolute constant $\sigma>0$.

We can now argue similarly to the situation of $b_{15}$ and $a^2_5+a^2_8$ that if the inequality
	\begin{gather}
		\label{SE21}
		\frac{(a^2_5-a^2_8)^2}{10^{12}}\leq b^2_{12}\leq 10^{12}(a^2_5-a^2_8)^2
	\end{gather}
is violated, we already have quadratic control of $b_{12}$ and $a^2_5-a^2_8$ (i.e., up to a positive constant these terms appear squared in our lower bound). Again, just like in the situation of $b_{15}$, we can argue that if on the other hand (\ref{SE21}) is valid, then as long as
	\[
	\left|\frac{a^2_5-a^2_8}{b_{12}}-\frac{6\pi}{\sqrt{7}}\right|\geq \epsilon_0
	\]
	for some absolute constant $\epsilon_0>0$, we once more have quadratic control of $b_{12}$ and $a^2_5-a^2_8$ modulo a positive constant which will depend on $\epsilon_0$. Observe also that the situation $b_{12}=0$ can be ruled out by looking at
	\[
	\sigma\left(\beta^2_3+b^2_{11}+\left(\sqrt{\frac{17}{216}}\frac{a^2_5-a^2_8}{\pi}-\frac{17}{36\sqrt{7}}b_{12}\sqrt{\frac{216}{17}}\right)^2\right)\,,
	\]
since if $b_{12}=0$ we immediately obtain quadratic control of $a^2_5-a^2_8$. Therefore, the only situation in which we do not control $b^2_{12}$ as well as $a^2_5-a^2_8$ quadratically is when~\eqref{SE22} is satisfied for a fixed absolute constant $\epsilon_0>0$.

The case of of the terms in~\eqref{SE18} involving $b_{10}$ and $a_5a_8$ can be handled identically as that of $b_{12}$ and $a^2_5-a^2_8$, so that we see that we have quadratic control of $b_{10}$ and $a_5a_8$, unless the relation~\eqref{SE24} holds for an absolute constant $\epsilon_0>0$.

Next we note that if~\eqref{SE20} (i.e., the relation~\eqref{E1}) is violated we control $b_{15}$ quadratically and $(a^2_5+a^2_8)$ quadratically. But quadratic control of $a^2_5+a^2_8$ gives also quadratic control of $a^2_5-a^2_8$ and $a_5a_8$ (because $(a^2_5+a^2_8)^2=(a^2_5-a^2_8)^2+4a^2_5a^2_8)$, which in turn gives us, either by means of~\eqref{SE22} and~\eqref{SE24}, quadratic control of $b_{12}$ and $b_{10}$ respectively or else, if one of the identities~\eqref{SE22} or~\eqref{SE24} is violated, we already have quadratic control of $b^2_{12}$ and $b^2_{10}$ respectively. Therefore, in these cases we control $b^2_{10},b^2_{12},b^2_{15}$ and $\|Z_2\|^4_{L^2}$ and a Taylor expansion up to fourth order shows that
	\[
	\mathcal{F}(Y)\geq \mathcal{F}(B_1)=\mathcal{F}(X)
	\]
	with equality if and only if $Y=B_1$. So without loss of generality we may assume that~\eqref{E1} is satisfied. Finally, if both the conditions~\eqref{SE22} and~\eqref{SE24} are violated, then we control $b^2_{10},b^2_{12}$ and $(a^2_5-a^2_8)^2,a^2_5a^2_8$ so that in particular we have control of $(a^2_5+a^2_8)^2$ and consequently in any case we also control $b^2_{15}$ because of~\eqref{E1}. A Taylor expansion up to fourth order is again enough to conclude the desired result. This completes the proof of the proposition.
\end{proof}

In the next two lemmas we compute the fifth and sixth order derivatives of the functional $\cF$. They are used in the proof of Corollary~\ref{CE2} below.

\begin{lemma}
	We have the following $5$th order leading order approximation
\begin{gather}
	\label{SE25}
	D^5\mathcal{F}(B_1)(W,W,W,W,W)
	\\
	\nonumber
	\approx \frac{(\mathcal{E}(B_1))^\frac{1}{3}}{\mathcal{H}(B_1)}\frac{a^2_5+a^2_8}{\pi^3}\left(-\frac{5\sqrt{3}}{18}b_{15}(a^2_5+a^2_8)-
\frac{5\sqrt{7}}{9}b_{12}(a^2_5-a^2_8)+\frac{10\sqrt{7}}{9}b_{10}a_5a_8\right)\,.
\end{gather}
\end{lemma}
\begin{proof}
	We have
	\[
	 D^5\mathcal{F}(B_1)(W,W,W,W,W)=\frac{(\mathcal{E}(B_1))^\frac{1}{3}}{\mathcal{H}(B_1)}\left(\frac{20}{9\pi^2}D^2\mathcal{E}(B_1)(W,W)D^3\mathcal{E}(B_1)(W,W,W)\right.
	\]
	\[
	\left.-\frac{40}{3\pi^2}D^3\mathcal{E}(B_1)(W,W,W)D^2\mathcal{H}(B_1)(W,W)+\frac{4}{3}D^5\mathcal{E}(B_1)(W,W,W,W,W)\right)
	\]
	and

\begin{align*}
&D^5\mathcal{E}(B_1)(W,W,W,W,W)\\
&=\frac{15}{32}\Bigg(60\int_{\SS^3}|W|^4(B_1\cdot W)dV_{g_0}-180\int_{\SS^3}(B_1\cdot W)^3|W|^2dV_{g_0}\\
&+117\int_{\SS^3}(B_1\cdot W)^5dV_{g_0}\Bigg)\,.
\end{align*}

As usual, dropping higher order terms, we arrive at
	\begin{gather}
		\nonumber
		D^5\mathcal{F}(B_1)(W,W,W,W,W)
		\\
		\nonumber
		\approx \frac{(\mathcal{E}(B_1))^\frac{1}{3}}{\mathcal{H}(B_1)}\left[-\frac{5}{2}\frac{\|Z_2\|^2_{L^2}}{\pi^2}\left(\int_{\SS^3}\big(15(B_1\cdot P^2_3)(B_1\cdot Z_2)^2-6(B_1\cdot P^2_3)|Z_2|^2\right.\right.
		\\
		\nonumber
		-12(Z_2\cdot P^2_3)(B_1\cdot Z_2)\big)dV_{g_0}\bigg)+\frac{5}{8}\left(60\int_{\SS^3}\big(|Z_2|^4(B_1\cdot P^2_3)\right.\\
		\nonumber
		+4(P^2_3\cdot Z_2)(B_1\cdot Z_2)|Z_2|^2\big)dV_{g_0}+117\cdot 5\int_{\SS^3}(B_1\cdot Z_2)^4(B_1\cdot P^2_3)dV_{g_0}
		\\
		\nonumber
		\left.	\left.-180\int_{\SS^3}\big(2(B_1\cdot Z_2)^3(Z_2\cdot P^2_3)+3(B_1\cdot Z_2)^2(B_1\cdot P^2_3)|Z_2|^2\big)dV_{g_0}\right)\right]\,,
	\end{gather}
	where we have used that
	\[
	\int_{\SS^3}(B_1\cdot Z_2)^5dV_{g_0}=\int_{\SS^3}|Z_2|^4(B_1\cdot Z_2)dV_{g_0}=\int_{\SS^3}(B_1\cdot Z_2)^3|Z_2|^2dV_{g_0}=0\,,
	\]
which can be directly checked by explicit computations. The result then follows after computing all the summands using our orthonormal bases.
\end{proof}

\begin{lemma}
		\[
	D^6\mathcal{F}(B_1)(W,W,W,W,W,W)\approx \frac{(\mathcal{E}(B_1))^\frac{1}{3}}{\mathcal{H}(B_1)}\frac{685}{36\pi^4}\|Z_2\|^6_{L^2}\,.
	\]
\end{lemma}
\begin{proof}
A straightforward computation of the $6$-th order derivative yields
	\[
	D^6\mathcal{F}(B_1)(W,W,W,W,W,W)
	\]
	\[
	=\frac{(\mathcal{E}(B_1))^\frac{1}{3}}{\mathcal{H}(B_1)}\left[-\frac{10}{\pi^4}(D^2\mathcal{E}(B_1)(W,W))^2D^2\mathcal{H}(B_1)(W,W)\right.
	\]
	\[
	+\frac{10}{3\pi^2}D^2\mathcal{E}(B_1)(W,W)D^4\mathcal{E}(B_1)(W,W,W,W)
	\]
	\[
	-\frac{20}{\pi^2}D^4\mathcal{E}(B_1)(W,W,W,W)D^2\mathcal{H}(B_1)(W,W)
	\]
	\[
	+\frac{120}{\pi^4}D^2\mathcal{E}(B_1)(W,W)(D^2\mathcal{H}(B_1)(W,W))^2+\frac{20}{9\pi^2}(D^3\mathcal{E}(B_1)(W,W,W))^2
	\]
	\[
	+\frac{4}{3}D^6\mathcal{E}(B_1)(W,W,W,W,W,W)-\frac{180}{\pi^4}(D^2\mathcal{H}(B_1)(W,W))^3
	\]
	\[
	\left.-\frac{10}{9\pi^4}(D^2\mathcal{E}(B_1)(W,W))^3\right]\,,
	\]
	where
\begin{align*}
&D^6\mathcal{E}(B_1)(W,W,W,W,W,W)\\
&=\frac{15}{64}\left(120\int_{\SS^3}|W|^6dV_{g_0}-1620\int_{\SS^3}|W|^4(B_1\cdot W)^2dV_{g_0}\right.\\
&\left.+3510\int_{\SS^3}|W|^2(B_1\cdot W)^4dV_{g_0}-1755\int_{\SS^3}(B_1\cdot W)^6dV_{g_0}\right)\,.
\end{align*}
After long but elementary computations, and discarding higher order terms, the result follows.
\end{proof}

In the following corollary and lemma we use the vector field $R$ that we introduced in Corollary~\ref{CE1}.

\begin{corollary}
	\label{CE2}
	If the $3$ relations~\eqref{E1}, \eqref{SE22} and~\eqref{SE24} hold, then
	\begin{gather}
		\nonumber
		6D\mathcal{F}(B_1)(W)+3D^2\mathcal{F}(B_1)(W,W)+D^3\mathcal{F}(B_1)(W,W,W)+\frac{D^4\mathcal{F}(B_1)(W,W,W,W)}{4}
		\\
		\nonumber
		+\frac{D^5\mathcal{F}(B_1)(W,W,W,W,W)}{20}+\frac{D^6\mathcal{F}(B_1)(W,W,W,W,W,W)}{120}
		\\
		\nonumber
		\approx\frac{(\mathcal{E}(B_1))^\frac{1}{3}}{\mathcal{H}(B_1)}\left(\left(\frac{11}{32\pi^4}+\rho\right)(a^2_5+a^2_8)^3+\int_{\SS^3}W_0\cdot R \, dV_{g_0}+I\right)\,,
	\end{gather}
where $\rho$ is a constant that can be made as close to zero as one wishes by letting $\epsilon_0>0$ be small enough in the aforementioned relations, and $I$ is a remainder term which is lower bounded as
\begin{gather}
	\nonumber
	I\geq 3\left(2\|\hat{W}_0\|^2_{L^2}-4\mathcal{H}(\hat{W}_0)-\int_{\SS^3}(B_1\cdot \hat{W}_0)^2dV_{g_0}\right)
	\\
	\label{EE1}
	+2\|Z_1\|^2_{L^2}+2\|W_{-1}\|^2_{L^2}+\frac{2}{3}\sum_{\substack{i=1 \\ i\neq 10,12,15}}^{15}b^2_i\,.
\end{gather}
\end{corollary}
\begin{proof}
Let $I$ be the collection of all terms from the fourth order Taylor expansion (see Corollary~\ref{CE1}), excluding those terms contained in the second square bracket in Equation~\eqref{SE9}. The lower bound on $I$ follows from the proof of Proposition~\ref{PE1}.

Then by means of the previous lemmas we have
	\begin{gather}
		\nonumber
		6D\mathcal{F}(B_1)(W)+3D^2\mathcal{F}(B_1)(W,W)+D^3\mathcal{F}(B_1)(W,W,W)+\frac{D^4\mathcal{F}(B_1)(W,W,W,W)}{4}
		\\
		\nonumber
		+\frac{D^5\mathcal{F}(B_1)(W,W,W,W,W)}{20}+\frac{D^6\mathcal{F}(B_1)(W,W,W,W,W,W)}{120}
		\\
		\nonumber
		 \approx\frac{(\mathcal{E}(B_1))^\frac{1}{3}}{\mathcal{H}(B_1)}\frac{1}{6048\pi^4}\left(2304\pi^2\sqrt{21}b_{15}b_{10}(-a_5a_8)-3888\pi^3\sqrt{3}b^2_{10}b_{15}\right.
		\\
		\nonumber
		+1152\pi^2\sqrt{21}b_{15}b_{12}(a^2_5-a^2_8)-3888 \pi^3\sqrt{3}b_{15}b^2_{12}+4500\pi^2(a^2_5+a^2_8)(b^2_{12}+b^2_{10})
		\\
		\nonumber
		+2268\pi^2b^2_{15}(a^2_5+a^2_8)-84\pi\sqrt{3}b_{15}(a^2_5+a^2_8)^2-168\pi\sqrt{7}b_{12}(a^2_5-a^2_8)(a^2_5+a^2_8)
		\\
		\nonumber
		\left.+336\pi\sqrt{7}b_{10}a_5a_8(a^2_5+a^2_8)+959(a^2_5+a^2_8)^3\right)
		\\
		\label{SE27}
		+\frac{(\mathcal{E}(B_1))^\frac{1}{3}}{\mathcal{H}(B_1)}\left(\int_{\SS^3}W_0\cdot RdV_{g_0}+I\right)\,.
	\end{gather}

By assumption $b_{15}=\frac{a^2_5+a^2_8}{2\sqrt{3}\pi}$ (modulo a factor which is as close to one as one wishes). With this specific value for $b_{15}$, we obtain
	\begin{gather}
		\nonumber
		6D\mathcal{F}(B_1)(W)+3D^2\mathcal{F}(B_1)(W,W)+D^3\mathcal{F}(B_1)(W,W,W)+\frac{D^4\mathcal{F}(B_1)(W,W,W,W)}{4}
		\\
		\nonumber
		+\frac{D^5\mathcal{F}(B_1)(W,W,W,W,W)}{20}+\frac{D^6\mathcal{F}(B_1)(W,W,W,W,W,W)}{120}
		\\
		\nonumber
		 \approx\frac{(\mathcal{E}(B_1))^\frac{1}{3}}{\mathcal{H}(B_1)}\frac{a^2_5+a^2_8}{3024\pi^4}\left(408\pi\sqrt{7}b_{10}(-a_5a_8)+1278\pi^2b^2_{10}+204\pi\sqrt{7}b_{12}(a^2_5-a^2_8)\right.
		\\
		\label{SE28}
		\left.+1278\pi^2b^2_{12}+553(a^2_5+a^2_8)^2\right)+\frac{(\mathcal{E}(B_1))^\frac{1}{3}}{\mathcal{H}(B_1)}\left(\int_{\SS^3}W_0\cdot RdV_{g_0}+I\right).
	\end{gather}

From~\eqref{SE22} and~\eqref{SE24}, we see that either we can ignore the terms containing $b_{10}$ and $b_{12}$ respectively (because we have quadratic control of these terms) or else $b_{10}$ behaves modulo a positive constant like $-a_5a_8$ and $b_{12}$ behaves like $a^2_5-a^2_8$, so that the terms in~\eqref{SE28}) can either be ignored or are non-negative except for possibly $\int_{\SS^3}W_0\cdot RdV_{g_0}$.

Assuming now that $b_{10}$ and $b_{12}$ attain the values in~\eqref{SE22} and~\eqref{SE24} respectively, we infer
	\begin{gather}
		\nonumber
		6D\mathcal{F}(B_1)(W)+3D^2\mathcal{F}(B_1)(W,W)+D^3\mathcal{F}(B_1)(W,W,W)+\frac{D^4\mathcal{F}(B_1)(W,W,W,W)}{4}
		\\
		\nonumber
		+\frac{D^5\mathcal{F}(B_1)(W,W,W,W,W)}{20}+\frac{D^6\mathcal{F}(B_1)(W,W,W,W,W,W)}{120}
		\\
		\nonumber
		\approx\frac{(\mathcal{E}(B_1))^\frac{1}{3}}{\mathcal{H}(B_1)}\left(\frac{11}{32\pi^4}(a^2_5+a^2_8)^3+\int_{\SS^3}W_0\cdot RdV_{g_0}+I\right)\,.
	\end{gather}
	More precisely, by letting $\epsilon_0$ be as small as we want, we may obtain a constant which is as close to $\frac{11}{32\pi^4}$ as we wish, which completes the proof of the corollary.
\end{proof}

\begin{lemma}
	\label{LE5}
Suppose that the relations~\eqref{E1}, \eqref{SE22} and~\eqref{SE24} are satisfied. Then the $L^2$-orthogonal projection $C$ of $R$ into the divergence-free vector fields satisfies
	\[
	\|C\|^2_{L^2}=\frac{151}{90\pi^4}\|Z_2\|^6_{L^2}
	\]
	up to a factor that can be chosen as close to $1$ as one wishes by letting $\epsilon_0>0$ be small enough.
\end{lemma}
\begin{proof}
It is enough to consider the case where $\delta_i=0$, $i=10,12,15$, because identical in spirit calculations show that if $\delta_i\neq 0$ this only adds an additional factor which is as close to~$1$ as desired by letting $\epsilon_0>0$ be small. We define
	\[
	d_1:=-\frac{-177\sqrt{21}a_8b_{10}\pi-259a^2_8a_5\sqrt{3}-112a^3_5\sqrt{3}+87\sqrt{21}b_{12}a_5\pi+189\pi a_5b_{15}}{945\pi^3}\,,
	\]
	\[
	d_2:=0\,,
	\]
	\[
	d_3:=-\frac{267\sqrt{21}\pi a_8b_{12}-182\sqrt{3}a_8a^2_5-189\pi b_{15}a_8+259a^3_8\sqrt{3}-3\sqrt{21}\pi a_5b_{10}}{945\pi^3}\,,
	\]
	\[
	d_4:=\frac{177\sqrt{21}\pi a_8b_{10}+259\sqrt{3}a^2_8a_5-14\sqrt{3}a^3_5+75\sqrt{21}\pi b_{12}a_5+945\pi a_5b_{15}}{945\pi^3}\,,
	\]
	\[
	d_5:=-\frac{3\sqrt{21}\pi a_8b_{10}+182\sqrt{3}a^2_8a_5-259\sqrt{3}a^3_5+267\sqrt{21}\pi a_5b_{12}+189\pi a_5b_{15}}{945\pi^3}\,,
	\]
	\[
	d_6:=0\,,
	\]
	\[
	d_7:=-\frac{267\sqrt{21}\pi a_8b_{12}-14\sqrt{3}a_8a^2_5-189\pi b_{15}a_8+259a^3_8\sqrt{3}+15\sqrt{21}\pi a_5b_{10}}{945\pi^3}\,,
	\]
	\[
	d_8:=0,d_9:=0\,,
	\]
	\[
	d_{10}:=-\frac{87\sqrt{21}\pi a_8b_{12}+259\sqrt{3}a_8a^2_5-189\pi b_{15}a_8+112\sqrt{3}a^3_8+177\sqrt{21}\pi a_5b_{10}}{945\pi^3}\,,
	\]
	\[
	d_{11}:=0\,,
	\]
	\[
	d_{12}:=-\frac{2}{315\pi^3}\left(-57\sqrt{21}a_5b_{10}\pi-91\sqrt{3}a_8a^2_5+27\sqrt{21}\pi a_8b_{12}+189\pi b_{15}a_8\right)\,,
	\]
	\[
	d_{13}:=0\,,
	\]
	\[
	d_{14}:=-\frac{-15\sqrt{21}\pi a_8b_{10}+14\sqrt{3}a^2_8a_5-259a^3_5\sqrt{3}+267\sqrt{21}\pi b_{12}a_5+189\pi a_5b_{15}}{945\pi^3}\,,
	\]
	\[
	d_{15}:=0,d_{16}:=0\,,
	\]
	\[
	d_{17}:=-\frac{2}{315\pi^3}\left(27\sqrt{21}\pi a_5b_{12}+91\sqrt{3}a^2_8a_5-189\pi a_5b_{15}+57\sqrt{21}\pi a_8b_{10}\right)\,,
	\]
	\[
	d_{18}:=\frac{75\sqrt{21}\pi a_8b_{12}-259\sqrt{3}a_8a^2_5-945\pi b_{15}a_8+14\sqrt{3}a^3_8-177\sqrt{21}\pi a_5b_{10}}{945\pi^3}\,,
	\]
	\[
	d_{19}:=0\text{, }d_{20}:=0\,,
	\]
	and
	\[
	G:\mathbb{R}^4\rightarrow \mathbb{R}\text{, }(x,y,z,w)\mapsto d_1x^3+d_2x^2y+d_3x^2z+d_4xy^2+d_5xz^2+d_6xyz
	\]
	\[
	+d_7y^2z+d_8yz^2+d_9y^3+d_{10}z^3+d_{11}x^2w+d_{12}xyw+d_{13}xzw
	\]
	\[
	+d_{14}xw^2+d_{15}y^2w+d_{16}z^2w+d_{17}wyz+d_{18}w^2z+d_{19}w^2y+d_{20}w^3\,.
	\]

Setting $\hat{G}:=G|_{S^3}$ we observe that the vector field $C:=R-\nabla (\hat{G})$ (here we compute the gradient with respect to the round metric on $\SS^3$) is divergence-free. A tedious but explicit computation then yields the claim.
\end{proof}

The following two lemmas and corollary provide more information on the curl eigenfields corresponding to the eigenvalues $\mu_{-2}$ and $\mu_{4}$. Both results are used in the proof of Lemma~\ref{LE8} below, which is crucially used in the proof of Theorem~\ref{S3C2}.

\begin{lemma}
	\label{LE6}
	Let $E_{-2}$ be the eigenspace corresponding to the curl eigenvalue $\mu_{-2}=-3$. Then for all $W_{-2}\in E_{-2}$ we have
	\[
	\int_{\SS^3}(B_1\cdot W_{-2})^2 dV_{g_0}=\frac{\|W_{-2}\|^2_{L^2}}{3}\,.
	\]
\end{lemma}
\begin{proof}
We first note that if $W_2\in E_2$ and $W_{-1}\in E_{-1}$, then by an explicit calculation (since we know a basis for both spaces)
	\[
	\int_{\SS^3}(W_{-1}\cdot W_2)^2dV_{g_0}=\frac{\|W_{-1}\|^2_{L^2}\|W_2\|^2_{L^2}}{6\pi^2}\,,
	\]
	and therefore, by means of an orientation reversing isometry, we find
	\begin{gather}
		\label{SE30}
		\int_{\SS^3}(B_1\cdot W_{-2})^2dV_{g_0}=\frac{\|B_1\|^2_{L^2}\|W_{-2}\|^2_{L^2}}{6\pi^2}=\frac{\|W_{-2}\|^2_{L^2}}{3}\,.
	\end{gather}
\end{proof}

\begin{lemma}
	\label{LE7}
	The eigenspace of curl corresponding to the eigenvalue $\mu_4=5$ is $24$-dimensional and is spanned by
	\[
	w_1=\sqrt{\frac{6}{\pi^2}}\left((xz^2-xw^2-2yzw)B_2+(w^2y-z^2y-2xzw)B_3\right)\,,
	\]
	\[
	w_2=\sqrt{\frac{2}{\pi^2}}\left((3xy^2-x^3)B_2+(3x^2y-y^3)B_3\right)\,,
	\]
	\[
	w_3=\sqrt{\frac{2}{\pi^2}}\left((y^3-3x^2y)B_2+(3xy^2-x^3)B_3\right)\,,
	\]
	\[
	w_4=\sqrt{\frac{6}{\pi^2}}\left((yz^2+2xzw-w^2y)B_2+(xz^2-xw^2-2wyz)B_3\right)\,,
	\]
	\[
	w_5=\sqrt{\frac{6}{\pi^2}}\left((y^2w-2xyz-x^2w)B_2+(y^2z-x^2z+2xyw)B_3\right)\,,
	\]
	\[
	w_6=\sqrt{\frac{2}{\pi^2}}\left((3z^2w-w^3)B_2+(z^3-3w^2z)B_3\right)\,,
	\]
	\[
	w_7=\sqrt{\frac{6}{\pi^2}}\left((x^2z-y^2z-2xyw)B_2+(y^2w-x^2w-2xyz)B_3\right)\,,
	\]
	\[
	w_8=\sqrt{\frac{2}{\pi^2}}\left((z^3-3w^2z)B_2+(w^3-3z^2w)B_3\right)\,,
	\]
	\[
	w_9=\sqrt{\frac{32}{15\pi^2}}\left((x^3-3xw^2)B_1+(w^3-3x^2w)B_2-\frac{\sqrt{6}\pi}{16}w_5+\frac{\sqrt{2}\pi}{16}w_6\right)\,,
	\]
	\[
	w_{10}=\sqrt{\frac{32}{5\pi^2}}\left((x^2y-2xzw-w^2y)B_1+(w^2z-2xyw-x^2z)B_2-\frac{\sqrt{6}\pi}{48}w_7+\frac{\sqrt{2}\pi}{16}w_8\right)\,,
	\]
	\[
	w_{11}=\sqrt{\frac{32}{5\pi^2}}\left((x^2z+2xyw-w^2z)B_1+(x^2y-2xzw-w^2y)B_2+\frac{\sqrt{2}\pi}{16}w_3+\frac{\sqrt{6}\pi}{48}w_4\right)\,,
	\]
	\[
	w_{12}=\sqrt{\frac{32}{5\pi^2}}\left((xy^2-xz^2-2wyz)B_1+(wz^2-wy^2-2xyz)B_2-\frac{\sqrt{6}\pi}{48}w_5-\frac{\sqrt{2}\pi}{16}w_6\right)\,,
	\]
	\[
	w_{13}=\sqrt{\frac{32}{5\pi^2}}\left((y^2w-z^2w+2xyz)B_1+(xy^2-xz^2-2yzw)B_2-\frac{\sqrt{6}\pi}{48}w_1-\frac{\sqrt{2}\pi}{16}w_2\right)\,,
	\]
	\[
	w_{14}=-\sqrt{\frac{32}{15\pi^2}}\left((3y^2z-z^3)B_1+(y^3-3yz^2)B_2-\frac{\sqrt{2}\pi}{16}w_3+\frac{\sqrt{6}\pi}{16}w_4\right)\,,
	\]
	\[
	w_{15}=\sqrt{\frac{32}{15\pi^2}}\left((3yz^2-y^3)B_1+(3y^2z-z^3)B_2+\frac{\sqrt{6}\pi}{16}w_7+\frac{\sqrt{2}\pi}{16}w_2\right)\,,
	\]
	\[
	w_{16}=-\sqrt{\frac{32}{15\pi^2}}\left((3x^2w-w^3)B_1+(x^3-3xw^2)B_2-\frac{\sqrt{6}\pi}{16}w_1+\frac{\sqrt{2}\pi}{16}w_2\right)\,,
	\]
	\[
	w_{17}=-\sqrt{\frac{36}{5\pi^2}}\left((w^2z+2xyw-y^2z)B_1+(xy^2-xw^2+2wyz)B_3\right.\]
	\[
	\left.-\frac{\sqrt{2}\pi}{16}w_3+\frac{\sqrt{6}\pi}{48}w_4-\frac{\sqrt{10}\pi}{48}w_{11}+\frac{\sqrt{30}\pi}{38}w_{14}\right)\,,
	\]
	\[
	w_{18}=\sqrt{\frac{12}{5\pi^2}}\left((3x^2z-z^3)B_1+(3xz^2-x^3)B_3\right.
	\]
	\[
	\left.-\frac{\sqrt{2}\pi}{16}w_3-\frac{\sqrt{6}\pi}{16}w_4-\frac{\sqrt{10}\pi}{16}w_{11}-\frac{\sqrt{30}\pi}{48}w_{14}\right)\,,
	\]
	\[
	w_{19}=\sqrt{\frac{36}{5\pi^2}}\left((xy^2-xw^2+2wyz)B_1+(y^2z-2xyw-w^2z)B_3\right.
	\]
	\[
	\left.+\frac{\sqrt{6}\pi}{48}w_5-\frac{\sqrt{2}\pi}{16}w_6-\frac{\sqrt{30}\pi}{48}w_9+\frac{\sqrt{10}\pi}{48}w_{12}\right)\,,
	\]
	\[
	w_{20}=-\sqrt{\frac{36}{5\pi^2}}\left((z^2w+2xyz-x^2w)B_1+(yz^2-x^2y-2xzw)B_3\right.
	\]
	\[
	\left.-\frac{\sqrt{6}\pi}{48}w_1+\frac{\sqrt{2}\pi}{16}w_2-\frac{\sqrt{10}\pi}{48}w_{13}+\frac{\sqrt{30}\pi}{48}w_{16} \right)\,,
	\]
	\[
	w_{21}=-\sqrt{\frac{12}{5\pi^2}}\left((3w^2y-y^3)B_1+(3y^2w-w^3)B_3\right.
	\]
	\[
	\left. -\frac{\sqrt{6}\pi}{16}w_7+\frac{\sqrt{2}\pi}{16}w_8+\frac{\sqrt{10}\pi}{16}w_{10}-\frac{\sqrt{30}\pi}{48}w_{15}\right)\,,
	\]
	\[
	w_{22}=\sqrt{\frac{12}{5\pi^2}}\left((3xz^2-x^3)B_1+(z^3-3x^2z)B_3\right.
	\]
	\[
	\left.-\frac{\sqrt{6}\pi}{16}w_5-\frac{\sqrt{2}\pi}{16}w_6+\frac{\sqrt{30}\pi}{48}w_9+\frac{\sqrt{10}\pi}{16}w_{12}\right)\,,
	\]
	\[
	w_{23}=-\sqrt{\frac{36}{5\pi^2}}\left((yz^2-x^2y-2xzw)B_1+(x^2w-wz^2-2xyz)B_3\right.
	\]
	\[
	\left.-\frac{\sqrt{6}\pi}{48}w_7-\frac{\sqrt{2}\pi}{16}w_8-\frac{\sqrt{10}\pi}{48}w_{10}-\frac{\sqrt{30}\pi}{48}w_{15} \right)\,,
	\]
	\[
	w_{24}=\sqrt{\frac{12}{5\pi^2}}\left((3y^2w-w^3)B_1+(y^3-3w^2y)B_3\right.
	\]
	\[
	\left.+\frac{\sqrt{6}\pi}{16}w_1+\frac{\sqrt{2}\pi}{16}w_2-\frac{\sqrt{10}\pi}{16}w_{13}-\frac{\sqrt{30}\pi}{48}w_{16}\right)\,.
	\]
\end{lemma}
\begin{proof}
	It follows from \cite[Theorem 5.2]{Bar} that $E_4$ is $24$-dimensional. One can then use Equation~\eqref{SE11} in order to verify that the above $24$ vector fields $w_i$ define an $L^2$-orthonormal basis of this eigenspace.
\end{proof}

\begin{corollary}
	\label{CE3}
For all $W_4\in E_4$ we have the sharp inequality
	\[
	\|B_1\cdot W_4\|^2_{L^2}\leq \frac{3}{5}\|W_4\|^2_{L^2}\,.
	\]
\end{corollary}
\begin{proof}
	By means of Lemma~\ref{LE7} we can express
	\[
	W_4=\sum_{i=1}^{24}c_iw_i
	\]
for some constant $c_i\in \mathbb{R}$. Letting
	\[
	f:\mathbb{R}^4\rightarrow \mathbb{R},(x,y,z,w)
	\]
	\[
	\mapsto \frac{8}{15}(x^2+w^2)+x\left(-\frac{\sqrt{2}}{15}y+\frac{\sqrt{6}}{15}z\right)+\frac{7}{15}(y^2+z^2)-\frac{\sqrt{2}}{15}zw-\frac{\sqrt{6}}{15}yw\,,
	\]
	we observe that by means of an explicit computation we have
	\[
	\|B_1\cdot W_4\|^2_{L^2}=f(c_9,c_{22},c_{19},c_{12})+f(c_{10},c_{23},c_{21},c_{15})
	\]
	\[
	+f(c_{11},c_{17},c_{18},c_{14})+f(c_{13},c_{20},c_{24},c_{16})\,.
	\]
From this it is easy to see that the optimal constant $\kappa$ in the inequality
$$\|B_1\cdot W_4\|^2_{L^2}\leq \kappa \|W_4\|^2_{L^2}$$
for all $W_4\in E_4$ is given by the global maximum of the function $f$ upon restricting it to $\SS^3$. By means of a standard Lagrange multiplier approach this maximum is computed to be $\kappa=\frac{3}{5}$, and the lemma follows.
\end{proof}

The final lemma of this section is an estimate for $\hat{W}_0$ in terms of its components $W_{-2},W_{-3},W_4,W_5$ and $W_R$.

\begin{lemma}
	\label{LE8}
The following lower bound holds
	\begin{gather}
		\nonumber
	2\|\hat{W}_0\|^2_{L^2}-4\mathcal{H}(\hat{W}_0)-\int_{\SS^3}(B_1\cdot \hat{W}_0)^2dV_{g_0}\geq
\\
\nonumber
\frac{\|W_{-3}\|^2_{L^2}+\|W_5\|^2_{L^2}}{3}+\frac{3}{7}\|W_R\|^2_{L^2}+\frac{9}{20}\|W_4\|^2_{L^2}+\frac{5}{3}\|W_{-2}\|^2_{L^2}\,.	
\end{gather}
\end{lemma}
\begin{proof}
We let $\hat{W}_0=(W_4+W_{-2})+(W_5+W_{-3})+W_R$ and observe that we can write and estimate the part involving $W_{-3}$ and $W_5$ by standard arguments:
	\[
	2\|\hat{W}_0\|^2_{L^2}-4\mathcal{H}(\hat{W}_0)-\int_{\SS^3}(B_1\cdot \hat{W}_0)^2dV_{g_0}
	\]
	\[
	\geq \frac{1}{3}\left(\|W_{-3}\|^2_{L^2}+\|W_5\|^2_{L^2}\right)+\left(2\|W_R\|^2_{L^2}-4\mathcal{H}(W_R)-\int_{\SS^3}(B_1\cdot W_R)^2dV_{g_0}\right)
	\]
	\[
	+\left(2\|W_4+W_{-2}\|^2_{L^2}-4\mathcal{H}(W_4+W_{-2})-\int_{\SS^3}(B_1\cdot (W_4+W_{-2}))^2dV_{g_0}\right)\,.
	\]

For the terms with $W_R$ we can use the rough estimates
$$\|W_R\|^2_{L^2}-\int_{\SS^3}(B_1\cdot W_R)^2dV_{g_0}\geq 0$$
and $\mathcal{H}(W_R)\leq \frac{\|W_R\|^2_{L^2}}{7}$, since the curl eigenfield corresponding to the smallest positive curl eigenvalue in the eigenfield expansion of $W_R$ is $W_6$ (which corresponds to the eigenvalue $7$). Accordingly,
	\[
	2\|W_R\|^2_{L^2}-4\mathcal{H}(W_R)-\int_{\SS^3}(B_1\cdot W_R)^2dV_{g_0}\geq \frac{3}{7}\|W_R\|^2_{L^2}\,.
	\]

As for the remaining term we first compute
	\[
	2\|W_4+W_{-2}\|^2_{L^2}-4\mathcal{H}(W_4+W_{-2})-\int_{\SS^3}(B_1\cdot (W_4+W_{-2}))^2dV_{g_0}
	\]
	\[
	=2\|W_4\|^2_{L^2}+2\|W_{-2}\|^2_{L^2}-4\mathcal{H}(W_4)-4\mathcal{H}(W_{-2})-\int_{\SS^3}(B_1\cdot (W_4+W_{-2}))^2dV_{g_0}
	\]
	\[
	=\frac{6}{5}\|W_4\|^2_{L^2}+\frac{10}{3}\|W_{-2}\|^2_{L^2}-\int_{\SS^3}\big((B_1\cdot W_4)^2+(B_1\cdot W_{-2})^2+2(B_1\cdot W_4)(B_1\cdot W_{-2})\big)dV_{g_0}\,.
	\]
Using Lemma~\ref{LE6} and the elementary estimate $2ab\leq a^2\epsilon+\frac{b^2}{\epsilon}$ for all $\epsilon>0$, we obtain that the last expression is lower bounded by
	\[
	\geq \frac{6}{5}\|W_4\|^2_{L^2}-\int_{\SS^3}(B_1\cdot W_4)^2dV_{g_0}\left(1+\epsilon\right)+\frac{10}{3}\|W_{-2}\|^2_{L^2}-\int_{\SS^3}(B_1\cdot W_{-2})^2dV_{g_0}\left(1+\frac{1}{\epsilon}\right)
	\]
	\[
	=\frac{6}{5}\|W_4\|^2_{L^2}-\int_{\SS^3}(B_1\cdot W_4)^2dV_{g_0}\left(1+\epsilon\right)+\|W_{-2}\|^2_{L^2}\left(3-\frac{1}{3\epsilon}\right)\,.
	\]
Finally, it follows from Corollary~\ref{CE3} that we have the optimal estimate
	\[
	\|B_1\cdot W_4\|^2_{L^2}\leq \frac{3}{5}\|W_4\|^2_{L^2}
	\]
	for all $W_4$, and consequently
	\[
	\frac{6}{5}\|W_4\|^2_{L^2}-\int_{\SS^3}(B_1\cdot W_4)^2dV_{g_0}\left(1+\epsilon\right)\geq \frac{3\|W_4\|^2_{L^2}}{5}\left(1-\epsilon\right)\,.
	\]
Choosing $\epsilon=\frac{1}{4}$, we get the desired estimate.
\end{proof}
\section*{Acknowledgements}

The authors are grateful to Iosif Polterovich for valuable comments concerning eigenvalues of the Laplacian in conformal classes. This work has received funding from the European Research Council (ERC) under the European Union's Horizon 2020 research and innovation program through the Consolidator Grant agreement~862342 (A.E.). It is supported by the grants CEX2019-000904-S and PID2019-106715GB GB-C21 (D.P.-S.) funded by MCIN/ AEI/ 10.13039/501100011033.

\appendix

\section{The first eigenvalue in fixed volume and conformal classes}\label{S2}

The goal of this Appendix is to show that when we consider a closed Riemannian $3$-manifold $(M,g_0)$, the first curl eigenvalue in its conformal class of the same volume is bounded from below by a positive constant that only depends on $(M,g_0)$, while such a lower bound does not exist if we only fix the volume of the manifold. These results are presented in Theorems~\ref{MET1} and~\ref{ConfT3} below, the latter implying the estimate for coexact $1$-forms in Theorem~\ref{T.AT}, in the case $n=3$.

\subsection{Fixed volume class}\label{SS.A1}

In this subsection we need to consider the curl operator on a bounded domain of a manifold, so let us briefly summarize its spectral properties. Given a smooth bounded domain $\Om\subset M$ endowed with the induced metric $g|_\Om$, it is classical~\cite{Giga} (see also~\cite{Hiptmair}) that curl defines a self-adjoint operator on $\Om$ with compact resolvent whose domain is dense in the space
\begin{align*}
\mathcal K(\Om)=\Big\{ v\in L^2_g(\Om): \Div_g v=0\,,\; v\restr\cdot N=0\,,\; \\\int_\Om
v\cdot h\, dV_g=0 \; \text{for all } h\in\cH_\Om \Big\}\,.
\end{align*}
Here $\cH_\Om$ denotes the space of harmonic fields on~$\Om$
that are tangent to the boundary, and $N$ is a unit normal to the boundary.

The eigenfunctions of curl are then vector fields on~$\Om$ that satisfy
\begin{align*}
\curl v_k=\mu_k^g(\Om)\, v_k\quad\text{in }\Om\,,
\end{align*}
and belong to $\mathcal K(\Om)$. As in the case of closed manifolds, there are infinitely many positive and negative
eigenvalues $\{\mu_k^g(\Om)\}_{k=-\infty}^\infty$ of curl, which one can label (taking into account multiplicities) so that
\[
\cdots \leq \mu_{-3}^g(\Om)\leq \mu_{-2}^g(\Om)\leq \mu_{-1}^g(\Om) <0<\mu_1^g(\Om)\leq \mu_2^g(\Om)\leq \mu_3^g(\Om)\leq \cdots
\]

In the following theorem we prove that fixing the volume of the manifold is not enough to get a lower bound on the first curl eigenvalue (which may be as close to zero as desired for a suitable choice of metric). This is in strong contrast with the case of Euclidean bounded domains $\Om\subset\RR^3$, where it can be shown~\cite{EP22} that $\mu_1(\Om)$ is lower bounded by a constant that only depends on $|\Omega|$.

\begin{theorem}\label{MET1}
Let $M$ be a closed $3$-manifold. Then there exists a sequence of smooth metrics $\{g_n\}_{n\in \mathbb{N}}$ on $M$ such that $|M|_{g_n}=1$ for all $n\in \mathbb{N}$ and the sequence of first positive curl eigenvalues $\{\mu^{g_n}_{1}\}_n$ converges to zero.
\end{theorem}
\begin{remark}
This result can be derived from the study of the spectrum of the Hodge Laplacian on coexact $p$-forms done in~\cite[Theorem 1.2]{AT21}, but our proof is totally different and considerably simpler.
\end{remark}
\begin{proof}

Let us consider a domain $\Om\subset M$ that is diffeomorphic to a solid toroidal annulus $\TT^2\times (0,2\pi)$. We further assume that $\Om$ is contained in a contractible domain of $M$. We first construct a sequence of smooth metrics $\{\tilde g_n\}$ on $\Om$, all of the same volume, and such that the first curl eigenvalues satisfy $\mu_1^{\tilde g_n}(\Om)\equiv\mu_n=\frac{1}{n}$.
To this end we introduce coordinates $(\phi_1,\phi_2,t)\in \TT^2\times [0,2\pi]$ and we set the metric to be
\begin{equation}\label{MEE1}
\tilde g_n:=\frac{d\phi_1^2+d\phi_2^2}{n}+n^2dt^2\,.
\end{equation}
It is straightforward to check that the space of harmonic fields $\mathcal{H}_\Om$ is spanned by the vector fields $\{\partial_{\phi_1},\partial_{\phi_2}\}$. Further if we let $v:=v^1\partial_{\phi_1}+v^2\partial_{\phi_2}+v^3\partial_t$ be a smooth vector field on $\Om$ satisfying $\curl_{\tilde g_n}v=\lambda v$ for some $\lambda\neq 0$, we can apply the curl twice to $v$ to obtain the following equations for the coordinate functions
\begin{equation}
\label{MEE2}
-n\partial^2_{\phi_1}v^i-n\partial^2_{\phi_2}v^i-\frac{\partial^2_{t}v^i}{n^2}=\lambda^2v^i
\end{equation}
for $1\leq i \leq 3$. Since we are looking for solutions $v\in \mathcal{K}(\Om)$ which are, in particular, tangent to $\partial\Om$, we have the boundary condition $v^3(\phi_1,\phi_2,0)=v^3(\phi_1,\phi_2,2\pi)=0$. Taking the separation ansatz $v^3=\Phi_1(\phi_1)\Phi_2(\phi_2)T(t)$, the functions $\Phi_i$ satisfy the ODE
\[
\Phi^{\prime\prime}_i+\frac{k^2_i}{n}\Phi_i=0
\]
for some constants $k_i$. The solutions to this ODE are $$\Phi_i(\phi_i)=a_i\sin\left(\frac{k_i}{\sqrt{n}}\phi_i\right)+b_i\cos\left(\frac{k_i}{\sqrt{n}}\phi_i\right)\,,$$
with $k_i=\sqrt{n}m_i$ for some $m_i\in \mathbb{N}$. Analogously, $T(t)$ satisfies the ODE
$$T^{\prime\prime}+n^2\tau^2T=0$$
for some constant $\tau$. Keeping in mind that $T(0)=0=T(2\pi)$ we obtain the solution $T(t)=\sin(n\tau t)$ with $\tau=\frac{m}{2n}$ for some $m\in \mathbb{N}$. It is standard that the separation ansatz provides us with a complete set of eigenfunctions. In particular, all possible curl eigenvalues $\lambda$ must be of the following form
\begin{equation}
	\label{MEE3}
	\lambda^2=k^2_1+k^2_2+\tau^2=nm^2_1+nm^2_2+\frac{m^2}{4n^2}
\end{equation}
for $(m_1,m_2,m)\in \mathbb{N}^3$. Now if either $m_1\neq 0$ or $m_2\neq 0$, then we must have $\lambda\geq 1$. We will now show that if $m_1=m_2=0$, we can find eigenfields of smaller eigenvalue (recall that~\eqref{MEE3} is only a necessary condition). Assume then that $k_1=k_2=0$, i.e., that $v^3=v^3(t)$, so that we obtain the relations
\begin{align}
&\partial_tv^1=n\lambda v^2\,,\label{MEE4-1}\\
&\partial_tv^2=-n\lambda v^1\,,\\
&n\lambda v^3=\partial_{\phi_1}v^2-\partial_{\phi_2}v^1\,.\label{MEE4}
\end{align}
We can now insert the second relation into the first one to get
\begin{equation}
	\label{MEE5}
-\frac{\partial^2_{tt}v^1}{n^2}=\lambda^2v^1\,,
\end{equation}
from which we, in combination with~\eqref{MEE2}, infer that $\partial^2_{\phi_1}v^1+\partial^2_{\phi_2} v^1=0$. An analogous equation holds for $v_2$. Since any harmonic function on a closed manifold is constant, we conclude that $v^1$ and $v^2$ are constant in the $\phi_1,\phi_2$ variables. Now since $v^1$ and $v^2$ depend only on $t$, we conclude by means of~\eqref{MEE4} that $v^3=0$. From~\eqref{MEE5} and~\eqref{MEE3} we obtain that $v^1$ has the form $v^1(t)=a\sin\left(\frac{m}{2} t\right)+b\cos\left(\frac{m}{2} t\right)$ for some $m\in \mathbb{N}$. Equation~\eqref{MEE4-1} then yields $v^2(t)=a\cos\left(\frac{m}{2}t\right)-b\sin\left(\frac{m}{2}t\right)$.

Finally, in order to make sure that $v\in \mathcal{K}(\Om)$, and hence it is orthogonal to the space of harmonic fields $\cH$, we must demand that
\[
\int_0^{2\pi}v^i(t)dt=0
\]
for $i=1,2$. This condition is satisfied if and only if $m$ is even. Accordingly, the smallest positive curl eigenvalue is obtained when $m=2$, which yields $\lambda=\mu_n=\frac{1}{n}$, as we desired to prove. The corresponding eigenspace is spanned by the vector fields
\begin{equation}
	\label{MEE6}
	v_1(t)=\sin(t)\partial_{\phi_1}+\cos(t)\partial_{\phi_2}\,, \qquad v_2(t)=\cos(t)\partial_{\phi_1}-\sin(t)\partial_{\phi_2}\,.
\end{equation}

After possibly scaling the metrics $\tilde g_n$ on $\Om$ by a positive constant, we may always assume that $|\Om|_{\tilde g_n}<1$ for all $n$. An elementary argument allows us to extend each of these metrics to smooth metrics $g_n$ on $M$ of total volume $1$. Now, take any vector field $\hat w\in \mathcal{K}(\Omega)$. We can then define a vector field on $M$ as
\[
\hat{v}:=\begin{cases}
	\hat w&\text{on }\Omega\,,\\
	0&\text{on }M\setminus \Omega\,.
\end{cases}
\]
It is easy to check that this $L^2$ vector field is in $\mathfrak X_{\mathrm{ex}}^{g_n}(M)$ (because it is divergence-free and its support is contained in a contractible subset of $M$). The variational characterization of the first curl eigenvalue presented in Section~\ref{S.intro} (which also holds for the first curl eigenvalue on a bounded domain $\Om$) implies:
\begin{equation*}
\mu_1(M,g_n)=\inf_{w\in \mathfrak X_{\mathrm{ex}}^{g_n}(M)\text{, }\mathcal{H}_{g_n}(w)>0}\frac{\|w\|^2_{L^2_{g_n}}}{\mathcal{H}_{g_n}(w)}\leq \inf_{\hat w\in \mathcal{K}(\Om)\text{, }\mathcal{H}_{\tilde g_n}(\hat w)>0}\frac{\|\hat{v}\|^2_{L^2_{g_n}}}{\mathcal{H}_{g_n}(\hat{v})}
\end{equation*}
\begin{equation}
\label{E.Var}
=\inf_{\hat w\in \mathcal{K}(\Om_1)\text{, }\mathcal{H}_{\tilde g_n}(\hat w)>0}\frac{\|\hat w\|^2_{L^2_{\tilde g_n}(\Omega)}}{\mathcal{H}_{\tilde g_n}(\hat w)}=\mu_1^{\tilde g_n}(\Omega)=\frac{1}{n}\,.
\end{equation}
In these identities we have used the definition of $\hat{v}$ and the fact that $\mathcal H_{g_n}(\hat v)=\mathcal H_{\tilde g_n}(\hat w)$ because $\hat v$ is supported (and equal to $\hat w$) on $\Om$ and $\curl^{-1}\hat w=\curl^{-1} \hat v$ on $\Om$ up to a curl-free vector field $X$, which does not contribute to the helicity $\mathcal H_{\tilde g_n}(\hat w)$ because $\hat w\in \mathcal K(\Om)$. We conclude that $\mu_1^{g_n}(M)\to 0$ as $n\to \infty$, and the theorem follows.
\end{proof}

\subsection{Fixed conformal and volume classes}

Now we show that the first curl eigenvalue is lower bounded when we fix the conformal class (and the volume) of the metric. As in the proof of Theorem~\ref{MET1}, the variational characterization of the first curl eigenvalue is key in the proof. This result also follows from Jammes' theorem~\cite{Jammes} on the first eigenvalue of the Hodge Laplacian for coexact $1$-forms in dimension~$3$, but our proof is completely different because it is based on the properties of the curl operator.

\begin{theorem}
	\label{ConfT3}
	Let $(M,g_0)$ be a closed Riemannian $3$-manifold. Then there exists a positive constant $C>0$ that only depends on the metric $g_0$ such that for any smooth metric $g$ on $M$ which is conformal to $g_0$ we have the following estimate
	\[
	\min\{|\mu^g_{-1}|,\mu^g_{1}\}\geq \frac{C}{|M|_g^{1/3}}\,.
	\]
\end{theorem}
\begin{remark}
Using a version of the Sobolev inequality with optimal constants~\cite{Li83}, it is not hard to show that $C$ can be chosen to be
$$
C=\Big(\frac{16}{\pi}\Big)^{1/3}
$$
in the case of the round sphere $(\mathbb S^3,\gcan )$.
\end{remark}
\begin{proof}
Let $g=fg_0$ be any smooth metric conformal to $g_0$ (the conformal factor $f$ is any $C^\infty$ positive function). Given a vector field $u\in \mathfrak X_{\mathrm{ex}}^g(M)$, we denote its dual $1$-form (via the metric) by $\om^{g}_u$ and define a second vector field $v$ as
$$v=I(u):=\Big(\star_{g_0}\star_g\om^{g}_u\Big)^{\sharp_{g_0}}\,,$$
where the musical isomorphism $\sharp_{g_0}$ is computed using the metric $g_0$. It is easy to check that $v\in \mathfrak X_{\mathrm{ex}}^{g_0}(M)$ and that the helicities of the two vector fields coincide, i.e., $\mathcal{H}_{g_0}(v)=\mathcal{H}_g(u)$ (this remains true even if the two metrics are not conformal). Further, a straightforward computation shows that
	\[
	g_0(v,v)=\langle \star_{g_0}\star_g\om^{g}_u,\star_{g_0}\star_g\om^{g}_u\rangle_{g_0}=f\langle \om^{g}_u,\om^{g}_u\rangle_{g_0}=f^2\langle \om^{g}_u,\om^{g}_u\rangle_{g}=f^2g(u,u)\,.
	\]
Consequently we find
	\[
	|v|_{g_0}^{3/2}dV_{g_0}=|u|_g^{\frac{3}{2}}dV_g\,,
	\]
and therefore
\begin{equation}\label{ConfE3}
\|u\|_{L^{\frac{3}{2}}_g(M)}=\|v\|_{L^{\frac{3}{2}}_{g_0}(M)}\,,
\end{equation}
i.e., the isomorphism $u\in\mathfrak X_{\mathrm{ex}}^{g}(M)\rightarrow v\in \mathfrak X_{\mathrm{ex}}^{g_0}(M)$ defined above preserves the $L^{\frac{3}{2}}$-norm if $g$ and $g_0$ are conformal (as well as the helicity). Using the H\"{o}lder inequality
\begin{equation}\label{ConfE4}
\|u\|_{L^{\frac{3}{2}}_g(M)}\leq |M|_g^{1/6}\|u\|_{L^2_g(M)}\,,
\end{equation}
it readily follows that if $u\in \mathfrak X_{\mathrm{ex}}^{g}(M)$ then
\begin{equation}\label{eq1}
\frac{\|u\|^2_{L^2_g(M)}}{|\mathcal{H}_g(u)|}\geq \frac{\|u\|^2_{L^{\frac{3}{2}}_g(M)}}{|\mathcal{H}_g(u)||M|_g^{1/3}}=\frac{\|v\|^2_{L^{\frac{3}{2}}_{g_0}(M)}}{|\mathcal{H}_{g_0}(v)||M|_g^{1/3}}\geq \frac{\inf_{v\in \mathfrak X_{\mathrm{ex}}^{g_0}(M)}\frac{\|v\|^2_{L^{{3}/{2}}_{g_0}(M)}}{|\mathcal{H}_{g_0}(v)|}}{|M|_g^{1/3}}\,.
\end{equation}
Next we observe that for every $v\in \mathfrak X_{\mathrm{ex}}^{g_0}(M)$ we can express the corresponding $1$-form $\om_v^{g_0}$ as $\om_v^{g_0}=\delta_{g_0}\beta$ for a suitable $2$-form $\beta$ (as usual, $\delta_{g_0}$ is the codifferential operator). It now follows from \cite[Corollary 2.4.15]{S95} that $\beta$ can be chosen so that
$$\|\beta\|_{W^{1,\frac{3}{2}}_{g_0}(M)}\leq C_0\|\om_v^{g_0}\|_{L^{\frac{3}{2}}_{g_0}(M)}$$
for some constant $C_0>0$ independent of $\om^{g_0}_v$. Now by means of the Sobolev inequality on manifolds~\cite[Theorem 1.3.6]{S95}, and since the Hodge star operator preserves the $W^{k,p}$ norms for all $k,p$ (see~\cite[Equation (3.16)]{S95}) it follows that
\begin{align*}
|\mathcal{H}_{g_0}(v)|&=|\langle \om_v^{g_0},\star_{g_0}\beta\rangle_{{g_0}}|\leq \|\om_v^{g_0}\|_{L^{\frac{3}{2}}_{g_0}(M)}\|\star_{g_0}\beta\|_{L^3_{g_0}(M)}\\
&\leq C_S\|\om_v^{g_0}\|_{L^{\frac{3}{2}}_{g_0}(M)}\|\star_{g_0}\beta\|_{W^{1,\frac{3}{2}}_{g_0}(M)}
=C_S\|\om_v^{g_0}\|_{L^{\frac{3}{2}}_{g_0}(M)}\|\beta\|_{W^{1,\frac{3}{2}}_{g_0}(M)}\\
&\leq C_SC_0\|\om_v^{g_0}\|^2_{L^{\frac{3}{2}}_{g_0}(M)}=C_SC_0\|v\|^2_{L^{\frac{3}{2}}_{g_0}(M)}\,,
\end{align*}
where $C_S>0$ denotes the constant from the Sobolev embedding inequality. Overall we obtain that
\[
\frac{\|v\|^2_{L^{\frac{3}{2}}_{g_0}(M)}}{|\mathcal{H}_{g_0}(v)|}\geq \frac{1}{C_SC_0}
\]
for all $v\in \mathfrak X_{\mathrm{ex}}^{g_0}(M)$ with $\mathcal{H}_{g_0}(v)\neq 0$. Combining this inequality with the estimate~\eqref{eq1}, we obtain
\begin{equation}
	\label{ConfE5}
	\frac{\|u\|^2_{L^{2}_{g}(M)}}{|\mathcal{H}_g(u)|}\geq \frac{1}{C_SC_0|M|_g^{1/3}}
\end{equation}
for any $u\in \mathfrak X_{\mathrm{ex}}^{g}(M)$ and any metric $g$ conformal to $g_0$. We then conclude from the variational characterization of the first curl eigenvalue on $(M,g)$ that
$$\min\{|\mu^g_{-1}|,\mu^g_{1}\}\geq \frac{C}{|M|_g^{1/3}}$$
for some constant $C>0$ that only depends on the reference metric $g_0$, thus completing the proof of the theorem.
\end{proof}

\section{Minimality of the Hopf field and optimal domains for curl in Euclidean space}\label{S.app2}

In this short appendix we show that if the Hopf vector field $B_1$ (see Corollary~\ref{C.Hopf}) is a \emph{global} $L^{\frac{3}{2}}_\gcan$-minimizer in its helicity class on $(\mathbb S^3,\gcan )$, then one can improve the lower bound for the first curl eigenvalue in Euclidean domains established in~\cite[Theorem A.1]{EP22}. Notice that if $B_1$ is a global minimizer in its positive helicity class, then $B_{-1}$ (the anti-Hopf field) is a global minimizer as well in its negative helicity class (simply using the orientation-reversing isometry introduced in Equation~\eqref{Eq.isom}).

Indeed, if we let $N:=(0,0,0,1)\in \mathbb{S}^3$ denote the ``north pole'' and $\pi_N$ denote the stereographic projection from $\mathbb{S}^3\setminus N$ onto $\mathbb{R}^3$, the corresponding pullback metric satisfies
$$(\pi^{-1}_N)^{*}\gcan=h g_E\,,$$
where $g_E$ denotes the Euclidean metric on $\RR^3$ and $h$ is a suitable (explicit) positive smooth function on $\mathbb{R}^3$. Given any divergence-free vector field $u$ of class $L^{\frac{3}{2}}_{\gcan}(\mathbb{S}^3)$ we can define a vector field $v$ on $\mathbb{R}^3$ by setting
$$v:=h^{\frac{3}{2}}(\pi_N)_{*}u\,.$$
It is easy to check that $v$ is of class $L^{\frac{3}{2}}(\mathbb{R}^3)$ and divergence-free with respect to the Euclidean volume. In fact, this identification provides an isomorphism $\mathcal{I}$ between the divergence-free $L^{\frac{3}{2}}$ vector fields on these two spaces (note that we may assign an arbitrary value to $u$ at $N$ because $L^{\frac{3}{2}}$ vector fields are defined only up to a set of measure zero). We remark that this assignment $\mathcal I$ preserves the $L^{\frac{3}{2}}$-norm as well as the helicities of the vector fields with respect to the corresponding metrics (we recall that the helicity of a divergence-free vector field on $\mathbb{R}^3$ is defined as the dual pairing of the vector field with its Biot-Savart potential).

Next, let $\Omega\subset \mathbb{R}^3$ be a bounded domain with smooth boundary and $w$ a vector field in the space $\mathcal K(\Omega)$ where curl defines a self-adjoint operator with compact resolvent, see Section~\ref{SS.A1}. Then, letting
\begin{align*}
\hat{w}:=\begin{cases}
	w & \text{ in }\Omega\,, \\
	0 & \text{ outside }\Omega\,,
\end{cases}
\end{align*}
we easily check that $\hat{w}$ is of class $L^{\frac{3}{2}}(\mathbb{R}^3)$ and is divergence-free. Therefore,
\begin{gather}
	\nonumber
\frac{\|w\|^2_{L^2(\Omega)}}{|\mathcal{H}_{\Omega}(w)|}\geq \frac{\|w\|^2_{L^{\frac{3}{2}}(\Omega)}}{|\mathcal{H}_{\Omega}(w)|\sqrt[3]{|\Omega|}} =\frac{\|\hat{w}\|^2_{L^{\frac{3}{2}}(\mathbb{R}^3)}}{|\mathcal{H}_{\mathbb{R}^3}(\hat{w})|\sqrt[3]{|\Omega|}}=
\frac{\|\mathcal{I}^{-1}(\hat{w})\|^2_{L^{\frac{3}{2}}(\mathbb{S}^3)}}{|\mathcal{H}_{\mathbb{S}^3}(\mathcal{I}^{-1}(\hat{w}))|\sqrt[3]{|\Omega|}}
\\
\nonumber
\geq \frac{\|B_{\pm1}\|^2_{L^\frac{3}{2}(\mathbb{S}^3)}}{|\mathcal{H}_{\mathbb{S}^3}(B_{\pm 1})|\sqrt[3]{|\Omega|}}=\frac{2\sqrt[3]{2\pi^2}}{\sqrt[3]{|\Omega|}}
\end{gather}
where we have used that $\mathcal{H}_{\Omega}(w)=\mathcal{H}_{\mathbb{R}^3}(\hat{w})$~\cite[Appendix A]{EP22}, and that the Hopf vector field $B_1$ (and hence the anti-Hopf field $B_{-1}$) is, by assumption, a global $L^{\frac{3}{2}}_\gcan$-minimizer.

Finally, using the variational characterization of $\mu_{\pm 1}(\Omega)$ as in~\cite{EP22} we obtain
\begin{equation}\label{eq-lb}
\min\{|\mu_{-1}(\Omega)|,\mu_1(\Omega)\}\geq\frac{2\sqrt[3]{2\pi^2}}{\sqrt[3]{|\Omega|}}\,,
\end{equation}
which improves the lower bound
$$\min\{|\mu_{-1}(\Omega)|,\mu_1(\Omega)\}\geq\frac{\Big(\frac{4\pi}{3}\Big)^{\frac13}}{\sqrt[3]{|\Omega|}}\,,$$
obtained in~\cite[Theorem A.1]{EP22}. Note that the inequality~\eqref{eq-lb} must in fact be strict because otherwise the first inequality in the Rayleigh-type quotient above should be an equality for any first eigenfield $w$ of $\mu_{\pm 1}(\Omega)$. This is possible only if $|w|$ is constant on the whole $\Omega$ by means of the equality case in H\"{o}lder's inequality. However, it is well known that any curl eigenfield of constant speed satisfies $\nabla_ww=0$, and so all the integral curves of the vector field $w$ are geodesics, i.e., straight lines, and hence it is impossible that $w$ be tangent to the boundary of $\Omega$. This contradiction implies that the inequality~\eqref{eq-lb} cannot be saturated.

\bibliographystyle{amsplain}

\end{document}